%% file: basefile.tex
\documentclass[a4paper,twoside,10pt,reqno]{amsart}

\numberwithin{equation}{section}
\allowdisplaybreaks

\usepackage[utf8]{inputenc}
\usepackage[T1]{fontenc}
\usepackage[english]{babel}

\usepackage[en-GB]{datetime2} 
\usepackage[shortcuts]{extdash}

\usepackage{amsmath,amssymb,amsthm} 
\usepackage{mathtools} 
\usepackage[cal=cm]{mathalfa} 
\DeclarePairedDelimiter\parentheses{\lparen}{\rparen}
\DeclarePairedDelimiter\vertparen{\lvert}{\rvert}
\DeclarePairedDelimiter\Vertparen{\lVert}{\rVert}

\usepackage{xfrac} 
\usepackage{xcolor}	
\usepackage{paralist} 
\usepackage{xspace}
\usepackage{xparse}
\usepackage{enumitem}
\usepackage{verbatim}
\usepackage{tikz,tikz-cd} 
\usepackage{mathrsfs} 
\usepackage{footmisc} 
\usepackage{bbm} 

\makeatletter
\def\@seccntformat#1{%
	\protect\textup{%
		\protect\@secnumfont%
		\expandafter\protect\csname format#1\endcsname%
		\csname the#1\endcsname%
		\protect\@secnumpunct%
	}%
}

\renewcommand\l@section{\@tocline{1}{0pt}{1pc}{}{}}
\renewcommand\l@subsection{\@tocline{2}{0pt}{2pc}{5pc}{}}
\renewcommand\l@subsubsection{\@tocline{3}{0pt}{4pc}{7pc}{}}

\def\namedlabel#1#2{%
    \begingroup%
    \def\@currentlabel{#2}\label{#1}
    \endgroup
}
\newcommand\newexplainedtextmacro[3]{\newcommand#1{\global\def#1{#2\xspace}#2\footnote{#3}\xspace}}

\newexplainedtextmacro\scs{SCS}{strongly continuous semigroup}

\makeatother

\usepackage{hyperref} 

\hypersetup{%
	pdftitle={Hypocoercivity of Langevin-type dynamics on abstract manifolds},%
	pdfauthor={Maximilian C.\ Mertin},%
	pdfkeywords={Nonlinear Analysis, Hypocoercivity},%
	colorlinks=true,
	linkcolor=blue,
    filecolor=green,
}

\addto\extrasenglish{%
    } 

\makeatletter
\newcommand*\vdef{\mathrel{\rlap{%
    \raisebox{.3ex}{$\m@th\cdot$}}%
    \raisebox{-.3ex}{$\m@th\cdot$}}%
    =}      
\newcommand*\fedv{=\mathrel{\rlap{%
    \raisebox{.3ex}{$\m@th\cdot$}}%
    \raisebox{-.3ex}{$\m@th\cdot$}}%
    }       
\makeatother

\newcommand\longmapsin{\ensuremath{\mathrel{\longrightarrow}}}

\newcommand\nullspace[1]{\ensuremath{\mathrm{Null}\parentheses*{#1}}\xspace}

\DeclareMathOperator\supp{supp}

\newcommand\euler{\ensuremath{\mathord{\mathrm{e}}}}

\NewDocumentCommand\ID{g}{%
    \ensuremath{\mathord{\mathrm{Id}%
        \IfValueT{#1}{_{#1}}%
    }}\xspace%
    }				

\newcommand\indicator[1]{\ensuremath{\mathord{\mathbbm{1}_{#1}}}}

\newcommand\pdot{\ensuremath{\mathbin{\cdot}}}

\newcommand\bincirc{\ensuremath{\mathbin{\circ}}}

\newcommand\tensorprod{\ensuremath{\mathbin{\otimes}}}

\NewDocumentCommand\abs{smg}{%
    \ensuremath{%
        \IfBooleanTF{#1}{\vertparen*{#2}}{\vertparen{#2}}%
        \IfValueT{#3}{_{#3}}%
        }%
    }
\NewDocumentCommand\norm{smg}{%
    \ensuremath{%
        \IfBooleanTF{#1}{\Vertparen*{#2}}{\Vertparen{#2}}%
        \IfValueT{#3}{_{#3}}%
        }%
    }

\NewDocumentCommand\borel{g}{%
    \ensuremath{\mathop{\mathfrak{B}%
        \IfValueT{#1}{(#1)}%
    }}\xspace%
    }

\NewDocumentCommand\setcomplement{sm}{%
    \ensuremath{%
        \IfBooleanTF{#1}{\parentheses*{#2}}{#2}^{\mathrm{c}}%
    }}
\newcommand\diffd{\ensuremath{\mathrm{d}}}

\NewDocumentCommand\Leb{g}{%
    \ensuremath{\mathord{\lambda%
        \IfValueT{#1}{_{#1}}%
    }}\xspace%
    }

\newcommand\spheresurfmeasure{%
    \ensuremath{\mathord{%
        \mathrm{S}
    }}\xspace
    }

\newcommand\stratonovich{\ensuremath{\mathbin{\circ}}}

\newcommand\dualpair[2]{\ensuremath{\mathord{\left\langle#1\mathbin{,}#2\right\rangle}}}

\NewDocumentCommand\scalarprod{mmg}{%
	\ensuremath{\parentheses*{#1\mathbin{,}#2}%
		\IfValueT{#3}{_{#3}}%
	}%
	}

\DeclareMathOperator\proj{pr}

\DeclareMathOperator\linhull{span}

\newcommand\sigmaalgebra{\ensuremath{\sigma}\-/\text{algebra}\xspace}

\newcommand\measurable[2]{\ensuremath{#1}\=/\ensuremath{#2}\-/\text{measurable}}

\NewDocumentCommand\probargs{sO{}g}{%
    \IfValueT{#2}{_{#2}}
    \IfValueT{#3}{%
        \IfBooleanTF{#1}%
            {\probarg*{#3}}%
            {\probarg{#3}}
        }
    }

\DeclarePairedDelimiterX\probarg[1]{\lparen}{\rparen}{%
    \ifnum\currentgrouptype=16 \else \begingroup \fi
    \activatebar#1
    \ifnum\currentgrouptype=16 \else \endgroup \fi
    }

\newcommand\expect{%
    \operatorname{E}%
    \expectargs%
    }

\NewDocumentCommand\expectargs{sO{}g}{%
    \IfValueT{#2}{_{#2}}
    \IfValueT{#3}{%
        \IfBooleanTF{#1}%
            {\expectarg*{#3}}%
            {\expectarg{#3}}
        }
    }

\DeclarePairedDelimiterX\expectarg[1]{[}{]}{%
    \ifnum\currentgrouptype=16 \else \begingroup \fi
    \activatebar#1
    \ifnum\currentgrouptype=16 \else \endgroup \fi
    }

\newcommand\activatebar{%
    \begingroup\lccode`\~=`\|
    \lowercase{\endgroup\let~}\innermid
    \mathcode`|=\string"8000
    }

\newcommand\innermid{\nonscript\:\delimsize\vert\nonscript\:}

\NewDocumentCommand\filtration{gg}{%
	\ensuremath{\mathord{\mathfrak{F}%
		\IfValueT{#1}{_{#1}}%
		\IfValueT{#2}{^{#2}}%
		}}%
	}

\newcommand\normaldist[2]{\ensuremath{\mathord{\mathrm{N}\parentheses*{#1; #2}}}}

\newcommand\contidiff[2]{%
    \ensuremath{\mathord{C^{#1}%
    \parentheses*{#2}}}%
    }

\newcommand\smoothfunc[1]{%
    \ensuremath{\mathord{C^{\infty}%
    \parentheses*{#1}}}%
    }

\newcommand\smoothcompact[1]{%
    \ensuremath{\mathord{C_{c}^{\infty}%
    \parentheses*{#1}}}%
    }

\NewDocumentCommand\bigcalL{mgg}{%
	\ensuremath{\mathord{\mathcal{L}^{#1}%
		\IfValueT{#2}{\parentheses*{#2%
			\IfValueT{#3}{;#3}%
			}}%
		}}%
	}

\NewDocumentCommand\bigL{mgg}{%
	\ensuremath{\mathord{L^{#1}%
		\IfValueT{#2}{\parentheses*{#2%
			\IfValueT{#3}{;#3}%
			}}%
	    }}%
	}

\NewDocumentCommand\Lloc{mgg}{%
	\ensuremath{\mathord{L_{\mathrm{loc}}^{#1}%
		\IfValueT{#2}{\parentheses*{#2%
			\IfValueT{#3}{;#3}%
			}}%
		}}%
	}

\NewDocumentCommand\sobolev{mmgg}{%
    \ensuremath{\mathord{H^{#1,#2}%
        \IfValueT{#3}{\parentheses*{#3%
            \IfValueT{#4}{;#4}%
            }}
        }}%
    }

\NewDocumentCommand\diriboundsobolev{mmgg}{%
    \ensuremath{\mathord{H_0^{#1,#2}%
        \IfValueT{#3}{\parentheses*{#3%
            \IfValueT{#4}{;#4}%
            }} %
        }}%
    }

\newcommand\opdomain[1]{\ensuremath{\mathord{D\parentheses*{#1}}}}
\newcommand\operator[1]{\ensuremath{\mathord{\parentheses*{#1,\opdomain{#1}}}}}%
\NewDocumentCommand\adjoint{sm}{%
    \ensuremath{
        \mathchoice%
        {\displaystyle\IfBooleanTF{#1}{\parentheses*{#2}}{#2}^{\ast}}%
        {\textstyle\IfBooleanTF{#1}{\parentheses*{#2}}{#2}^{\raise.2ex\hbox{$\scriptstyle\ast$}}}%
        {\scriptstyle\IfBooleanTF{#1}{\parentheses*{#2}}{#2}^{\ast}}%
        {\scriptscriptstyle\IfBooleanTF{#1}{\parentheses*{#2}}{#2}^{\raise-.1ex\hbox{$\scriptscriptstyle\ast$}}}%
        }%
    }

\NewDocumentCommand\trace{g}{%
	\ensuremath{\mathord{\mathrm{tr}%
	    \IfValueT{#1}{_{#1}}%
	}}%
	}

\NewDocumentCommand\grad{gg}{%
	\ensuremath{\mathord{\nabla%
	    \IfValueT{#1}{_{\mkern-3mu#1}}%
	    \IfValueT{#2}{^{#2}}%
	}}%
	}
\newcommand\spheregrad{%
    \ensuremath{\grad{\sphere}}\xspace 
    }
\newcommand\wspheregrad{%
    \ensuremath{\grad{\wvertmetric* {\scriptscriptstyle\vert\mathrm{S}}}}\xspace %
    }

\NewDocumentCommand\divergence{g}{%
	\ensuremath{\mathord{\mathrm{div}%
	    \IfValueT{#1}{_{#1}}%
	}}%
	}
\newcommand\spheredivergence{%
    \ensuremath{\divergence{\sphere}}\xspace 
    }

\NewDocumentCommand\laplace{gg}{%
	\ensuremath{\mathop{\Delta%
	    \IfValueT{#1}{_{#1}}%
	    \IfValueT{#2}{^{#2}}%
	}}%
	}
\newcommand\spherelaplace{%
    \ensuremath{\laplace{\sphere}}\xspace 
    }

\NewDocumentCommand\connection{gg}{%
	\ensuremath{\mathord{\nabla%
	    \IfValueT{#1}{_{#1}}%
	    \IfValueT{#2}{^{#2}}%
	}}\xspace%
	}

\NewDocumentCommand\Liederivative{g}{%
	\ensuremath{\mathop{\mathscr{L}%
	    \IfValueT{#1}{_{#1}}%
	}}%
	}

\newcommand\Liebracket[2]{\ensuremath{\mathord{\left[#1\mathbin{,}#2\right]}}}

\NewDocumentCommand\Liespan{mg}{%
    \ensuremath{\mathord{%
        \mathrm{Lie}%
        \IfValueT{#2}{_{#2}}%
        \parentheses*{#1}%
    }}\xspace%
    }

\newcommand\insertmorphism{\ensuremath{\mathrel{%
    \raisebox{\depth}{\scalebox{1}[-1]{$\lnot$}}%
    }} %
    }

\newcommand\Nnum{\ensuremath{\mathbb{N}}\xspace}

\newcommand\Rnum{\ensuremath{\mathbb{R}}\xspace}

\newcommand\indexset{\ensuremath{\mathord{\mathbb{I}}}\xspace}

\newcommand\pathprob{\ensuremath{\mathord{\mathbb{P}}}\xspace}

\newcommand\callidim[1]{\text{\usefont{U}{BOONDOX-cal}{m}{n}#1}}

\newcommand\posfold{\ensuremath{\mathord{\mathbb{M}}}\xspace}
\newcommand\posdim{\callidim{n}\xspace}

\newcommand\basefold{\ensuremath{\mathord{\mathbb{B}}}\xspace}

\newcommand\totalfold{\ensuremath{\mathord{\mathbb{E}}}\xspace}

\NewDocumentCommand\ball{gggg}{%
	\ensuremath{\mathord{\mathbb{U}%
        \IfValueTF{#1}{(#1,\IfValueTF{#2}{#2}{1})}{\IfValueT{#2}{\parentheses*{0,#2}}}
	    \IfValueT{#3}{_{#3}}%
	    \IfValueT{#4}{^{#4}}%
	}}\xspace%
	}

\NewDocumentCommand\sphere{gg}{%
	\ensuremath{\mathord{\mathbb{S}%
	    \IfValueT{#1}{^{#1}}%
	    \IfValueT{#2}{_{#2}}%
	}}\xspace%
	}

\NewDocumentCommand\torus{g}{%
	\ensuremath{\mathord{\mathbb{T}%
	    \IfValueT{#1}{^{#1}}%
	}}\xspace%
	}

\NewDocumentCommand\bundleproj{g}{%
    \ensuremath{\mathord{\pi%
        \IfValueT{#1}{_{#1}}}
    }\xspace
}

\NewDocumentCommand\tanproj{g}{%
    \ensuremath{\pi%
        \IfValueTF{#1}{_{#1}}{_0}%
    }\xspace%
    }

\newcommand\unittanproj{\ensuremath{\mathord{\pi_{0\vert\mathrm{S}}}}\xspace}

\newcommand\vertproj{\ensuremath{\mathord{\mathrm{vpr}}}\xspace}

\newcommand\horproj{\ensuremath{\mathord{\mathrm{hpr}}}\xspace}

\newcommand\tantanproj{\tanproj{1}}

\NewDocumentCommand\tanbundle{smg}{%
	\ensuremath{\mathord{%
	\IfValueTF{#3}{%
        \mathchoice%
        {\displaystyle\mathrm{T}_{#3}%
            \IfBooleanT{#1}{^{\ast}}}%
        {\textstyle\mathrm{T}_{\raise.2ex\hbox{$\scriptscriptstyle#3$}}%
            \IfBooleanT{#1}{^{\raise.2ex\hbox{$\scriptscriptstyle\ast$}}}}%
        {\scriptstyle\mathrm{T}_{#3}%
            \IfBooleanT{#1}{^{\ast}}}%
        {\scriptscriptstyle\mathrm{T}_{\raise-.1ex\hbox{$\scriptscriptstyle#3$}}%
            \IfBooleanT{#1}{^{\raise-.1ex\hbox{$\scriptscriptstyle\ast$}}}}%
    }{\mathrm{T}\IfBooleanT{#1}{^{\ast}}}%
	#2}}\xspace%
	}

\NewDocumentCommand\unittanbundle{mg}{%
	\ensuremath{\mathord{%
	\IfValueTF{#2}{%
        \mathchoice%
        {\displaystyle\mathrm{S}_{#2}}%
        {\textstyle\mathrm{S}_{\raise.2ex\hbox{$\scriptscriptstyle#2$}}}%
        {\scriptstyle\mathrm{S}_{#2}}%
        {\scriptscriptstyle\mathrm{S}_{\raise-.1ex\hbox{$\scriptscriptstyle#2$}}}%
    }{\mathrm{S}}%
	#1}}\xspace%
	}

\NewDocumentCommand\normalbundle{mg}{%
	\ensuremath{\mathord{%
	\IfValueTF{#2}{%
        \mathchoice%
        {\displaystyle\mathrm{N}_{#2}}%
        {\textstyle\mathrm{N}_{\raise.2ex\hbox{$\scriptscriptstyle#2$}}}%
        {\scriptstyle\mathrm{N}_{#2}}%
        {\scriptscriptstyle\mathrm{S}_{\raise-.1ex\hbox{$\scriptscriptstyle#2$}}}%
    }{\mathrm{N}}%
	#1}}\xspace%
	}

\NewDocumentCommand\tantanbundle{smg}{%
	\ensuremath{\mathord{%
	\IfValueTF{#3}{%
        \mathchoice%
        {\displaystyle\mathrm{T}_{#3}%
            \IfBooleanT{#1}{^{\ast}}}%
        {\textstyle\mathrm{T}_{\raise.2ex\hbox{$\scriptscriptstyle#3$}}%
            \IfBooleanT{#1}{^{\raise.2ex\hbox{$\scriptscriptstyle\ast$}}}}%
        {\scriptstyle\mathrm{T}_{#3}%
            \IfBooleanT{#1}{^{\ast}}}%
        {\scriptscriptstyle\mathrm{T}_{\raise-.1ex\hbox{$\scriptscriptstyle#3$}}%
            \IfBooleanT{#1}{^{\raise-.1ex\hbox{$\scriptscriptstyle\ast$}}}}%
    }{\mathrm{T}\IfBooleanT{#1}{^{\ast}}}%
	\mathrm{T}#2}}\xspace%
	}

\NewDocumentCommand\tanunittanbundle{smg}{%
	\ensuremath{\mathord{%
	\IfValueTF{#3}{%
        \mathchoice%
        {\displaystyle\mathrm{T}_{#3}%
            \IfBooleanT{#1}{^{\ast}}}%
        {\textstyle\mathrm{T}_{\raise.2ex\hbox{$\scriptscriptstyle#3$}}%
            \IfBooleanT{#1}{^{\raise.2ex\hbox{$\scriptscriptstyle\ast$}}}}%
        {\scriptstyle\mathrm{T}_{#3}%
            \IfBooleanT{#1}{^{\ast}}}%
        {\scriptscriptstyle\mathrm{T}_{\raise-.1ex\hbox{$\scriptscriptstyle#3$}}%
            \IfBooleanT{#1}{^{\raise-.1ex\hbox{$\scriptscriptstyle\ast$}}}}%
    }{\mathrm{T}\IfBooleanT{#1}{^{\ast}}}%
	\mathrm{S}#2}}\xspace%
	}

\NewDocumentCommand\normalunittanbundle{smg}{%
	\ensuremath{\mathord{%
	\IfValueTF{#3}{%
        \mathchoice%
        {\displaystyle\mathrm{N}_{#3}%
            \IfBooleanT{#1}{^{\ast}}}%
        {\textstyle\mathrm{N}_{\raise.2ex\hbox{$\scriptscriptstyle#3$}}%
            \IfBooleanT{#1}{^{\raise.2ex\hbox{$\scriptscriptstyle\ast$}}}}%
        {\scriptstyle\mathrm{N}_{#3}%
            \IfBooleanT{#1}{^{\ast}}}%
        {\scriptscriptstyle\mathrm{N}_{\raise-.1ex\hbox{$\scriptscriptstyle#3$}}%
            \IfBooleanT{#1}{^{\raise-.1ex\hbox{$\scriptscriptstyle\ast$}}}}%
    }{\mathrm{N}\IfBooleanT{#1}{^{\ast}}}%
	\mathrm{S}#2}}\xspace%
	}

\NewDocumentCommand\vertbundle{mg}{%
	\ensuremath{\mathord{%
	\IfValueTF{#2}{%
        \mathchoice%
        {\displaystyle\mathrm{V}_{#2}}%
        {\textstyle\mathrm{V}_{\raise.2ex\hbox{$\scriptscriptstyle#2$}}}%
        {\scriptstyle\mathrm{V}_{#2}}%
        {\scriptscriptstyle\mathrm{V}_{\raise-.1ex\hbox{$\scriptscriptstyle#2$}}}%
    }{\mathrm{V}}%
	#1}}\xspace%
	}

\NewDocumentCommand\horbundle{mg}{%
	\ensuremath{\mathord{%
	\IfValueTF{#2}{%
        \mathchoice%
        {\displaystyle\mathrm{H}_{#2}}%
        {\textstyle\mathrm{H}_{\raise.2ex\hbox{$\scriptscriptstyle#2$}}}%
        {\scriptstyle\mathrm{H}_{#2}}%
        {\scriptscriptstyle\mathrm{H}_{\raise-.1ex\hbox{$\scriptscriptstyle#2$}}}%
    }{\mathrm{H}}%
	#1}}\xspace%
	}

\NewDocumentCommand\framebundle{sgg}{%
	\ensuremath{\mathord{\mathrm{F}%
	\IfNoValueTF{#3}%
        {\IfBooleanT{#1}{_{\text{orth}}}}%
        {_{\IfBooleanT{#1}{\text{orth};}#3}
    }%
	\IfValueT{#2}{#2}%
    }}\xspace%
	}

\NewDocumentCommand\tensorbundle{mgg}{%
	\ensuremath{\mathord{\mathrm{T}%
	\IfValueT{#2}{^{#2}}%
	\IfValueT{#3}{_{\raise.2ex\hbox{$\scriptstyle#3$}}}%
	\parentheses*{#1}}}\xspace%
	}

\NewDocumentCommand\densitybundle{mgg}{%
	\ensuremath{\mathord{\vert\Lambda\vert%
	_{\raise.2ex\hbox{$\scriptstyle#1$}}%
	\IfValueT{#2}{^{#2}}%
	\IfValueT{#3}{\parentheses*{#3}}}}\xspace%
	}

\NewDocumentCommand\basemetric{sgggg}{%
	\ensuremath{\mathord{\mathrm{b}%
    \IfValueT{#4}{_{#4}}\IfValueT{#5}{^{#5}}%
    \IfBooleanF{#1}{%
        \IfNoValueTF{#2}%
            {\parentheses*{\cdot\mathbin{,}\IfNoValueTF{#3}{\cdot}{#3}}}%
            {\parentheses*{#2\mathbin{,}\IfNoValueTF{#3}{\cdot}{#3}}}%
	    }%
    }}\xspace%
	}

\NewDocumentCommand\wbasemetric{sgggg}{%
	\ensuremath{\mathord{\boldsymbol{\mathrm{b}}%
    \IfValueT{#4}{_{#4}}\IfValueT{#5}{^{#5}}%
    \IfBooleanF{#1}{%
        \IfNoValueTF{#2}%
            {\parentheses*{\cdot\mathbin{,}\IfNoValueTF{#3}{\cdot}{#3}}}%
            {\parentheses*{#2\mathbin{,}\IfNoValueTF{#3}{\cdot}{#3}}}%
	    }%
    }}\xspace%
	}

\NewDocumentCommand\riemannmetric{sgggg}{%
	\ensuremath{\mathord{\mathrm{m}%
    \IfValueT{#4}{_{#4}}\IfValueT{#5}{^{#5}}%
    \IfBooleanF{#1}{%
        \IfNoValueTF{#2}%
            {\parentheses*{\cdot\mathbin{,}\IfNoValueTF{#3}{\cdot}{#3}}}%
            {\parentheses*{#2\mathbin{,}\IfNoValueTF{#3}{\cdot}{#3}}}%
	    }%
    }}\xspace%
	}

\NewDocumentCommand\wriemannmetric{sgggg}{%
	\ensuremath{\mathord{\boldsymbol{\mathrm{m}}%
    \IfValueT{#4}{_{#4}}\IfValueT{#5}{^{#5}}%
    \IfBooleanF{#1}{%
        \IfNoValueTF{#2}%
            {\parentheses*{\cdot\mathbin{,}\IfNoValueTF{#3}{\cdot}{#3}}}%
            {\parentheses*{#2\mathbin{,}\IfNoValueTF{#3}{\cdot}{#3}}}%
	    }%
    }}\xspace%
	}

\NewDocumentCommand\sasakimetric{sgggg}{%
	\ensuremath{\mathord{\mathrm{s}%
    \IfValueT{#4}{_{#4}}\IfValueT{#5}{^{#5}}%
    \IfBooleanF{#1}{%
        \IfNoValueTF{#2}%
            {\parentheses*{\cdot\mathbin{,}\IfNoValueTF{#3}{\cdot}{#3}}}%
            {\parentheses*{#2\mathbin{,}\IfNoValueTF{#3}{\cdot}{#3}}}%
	    }%
    }}\xspace%
	}

\NewDocumentCommand\wsasakimetric{sgggg}{%
	\ensuremath{\mathord{\boldsymbol{\mathrm{s}}%
    \IfValueT{#4}{_{#4}}\IfValueT{#5}{^{#5}}%
    \IfBooleanF{#1}{%
        \IfNoValueTF{#2}%
            {\parentheses*{\cdot\mathbin{,}\IfNoValueTF{#3}{\cdot}{#3}}}%
            {\parentheses*{#2\mathbin{,}\IfNoValueTF{#3}{\cdot}{#3}}}%
	    }%
    }}\xspace%
	}

\NewDocumentCommand\vertmetric{sgggg}{%
	\ensuremath{\mathord{\mathrm{v}%
    \IfValueT{#4}{_{#4}}\IfValueT{#5}{^{#5}}%
    \IfBooleanF{#1}{%
        \IfNoValueTF{#2}%
            {\parentheses*{\cdot\mathbin{,}\IfNoValueTF{#3}{\cdot}{#3}}}%
            {\parentheses*{#2\mathbin{,}\IfNoValueTF{#3}{\cdot}{#3}}}%
	    }%
    }}\xspace%
	}

\NewDocumentCommand\wvertmetric{sgggg}{%
	\ensuremath{\mathord{\boldsymbol{\mathrm{v}}%
    \IfValueT{#4}{_{#4}}\IfValueT{#5}{^{#5}}%
    \IfBooleanF{#1}{%
        \IfNoValueTF{#2}%
            {\parentheses*{\cdot\mathbin{,}\IfNoValueTF{#3}{\cdot}{#3}}}%
            {\parentheses*{#2\mathbin{,}\IfNoValueTF{#3}{\cdot}{#3}}}%
	    }%
    }}\xspace%
	}
\NewDocumentCommand\hormetric{sgggg}{%
	\ensuremath{\mathord{\mathrm{h}%
    \IfValueT{#4}{_{#4}}\IfValueT{#5}{^{#5}}%
    \IfBooleanF{#1}{%
        \IfNoValueTF{#2}%
            {\parentheses*{\cdot\mathbin{,}\IfNoValueTF{#3}{\cdot}{#3}}}%
            {\parentheses*{#2\mathbin{,}\IfNoValueTF{#3}{\cdot}{#3}}}%
	    }%
    }}\xspace%
	}

\NewDocumentCommand\whormetric{sgggg}{%
	\ensuremath{\mathord{\boldsymbol{\mathrm{h}}%
    \IfValueT{#4}{_{#4}}\IfValueT{#5}{^{#5}}%
    \IfBooleanF{#1}{%
        \IfNoValueTF{#2}%
            {\parentheses*{\cdot\mathbin{,}\IfNoValueTF{#3}{\cdot}{#3}}}%
            {\parentheses*{#2\mathbin{,}\IfNoValueTF{#3}{\cdot}{#3}}}%
	    }%
    }}\xspace%
	}

\NewDocumentCommand\configmetric{sgggg}{%
	\ensuremath{\mathord{\mathrm{q}%
    \IfValueT{#4}{_{#4}}\IfValueT{#5}{^{#5}}%
    \IfBooleanF{#1}{%
        \IfNoValueTF{#2}%
            {\parentheses*{\cdot\mathbin{,}\IfNoValueTF{#3}{\cdot}{#3}}}%
            {\parentheses*{#2\mathbin{,}\IfNoValueTF{#3}{\cdot}{#3}}}%
	    }%
    }}\xspace%
	}

\NewDocumentCommand\measurablesec{mg}{%
	\ensuremath{\mathord{%
	    \Gamma\parentheses*{#1\IfValueT{#2}{;#2}}%
	}}\xspace%
	}

\NewDocumentCommand\contidiffsec{mgg}{%
	\ensuremath{\mathord{%
        \mathchoice%
        {\displaystyle\Gamma^{#1}}%
        {\textstyle\Gamma^{\raise.2ex\hbox{$\scriptscriptstyle#1$}}}%
        {\scriptstyle\Gamma^{#1}}%
        {\scriptscriptstyle\Gamma^{\raise-.1ex\hbox{$\scriptscriptstyle#1$}}}%
	\IfValueT{#2}{(#2\IfValueT{#3}{;#3})}%
	}}\xspace%
	}

\NewDocumentCommand\contsec{gg}{%
	\ensuremath{\mathord{%
        \mathchoice%
        {\displaystyle\Gamma^0}%
        {\textstyle\Gamma^{\raise.2ex\hbox{$\scriptscriptstyle0$}}}%
        {\scriptstyle\Gamma^0}%
        {\scriptscriptstyle\Gamma^{\raise-.1ex\hbox{$\scriptscriptstyle0$}}}%
	\IfValueT{#1}{(#1\IfValueT{#2}{;#2})}%
	}}\xspace%
	}

\NewDocumentCommand\smoothsec{gg}{%
	\ensuremath{\mathord{%
        \mathchoice%
        {\displaystyle\Gamma^{\infty}}%
        {\textstyle\Gamma^{\raise.2ex\hbox{$\scriptscriptstyle\infty$}}}%
        {\scriptstyle\Gamma^{\infty}}%
        {\scriptscriptstyle\Gamma^{\raise-.1ex\hbox{$\scriptscriptstyle\infty$}}}%
	\IfValueT{#1}{\parentheses*{#1\IfValueT{#2}{;#2}}}%
	}}\xspace%
	}

\NewDocumentCommand\smoothcompactsec{gg}{%
	\ensuremath{\mathord{%
        \mathchoice%
        {\displaystyle\Gamma_c^{\infty}}%
        {\textstyle\Gamma_c^{\raise.2ex\hbox{$\scriptscriptstyle\infty$}}}%
        {\scriptstyle\Gamma_c^{\infty}}%
        {\scriptscriptstyle\Gamma_c^{\raise-.1ex\hbox{$\scriptscriptstyle\infty$}}}%
	\IfValueT{#1}{(#1\IfValueT{#2}{;#2})}%
	}}\xspace%
	}

\NewDocumentCommand\smoothintsec{mgg}{%
	\ensuremath{\mathord{%
        \mathchoice%
        {\displaystyle\Gamma^{\infty,#1}}%
        {\textstyle\Gamma^{\raise.2ex\hbox{$\scriptscriptstyle\infty,#1$}}}%
        {\scriptstyle\Gamma^{\infty,#1}}%
        {\scriptscriptstyle\Gamma^{\raise-.1ex\hbox{$\scriptscriptstyle\infty,#1$}}}%
	\IfValueT{#2}{\parentheses*{#2\IfValueT{#3}{;#3}}}%
	}}\xspace%
	}

\NewDocumentCommand\diffintsec{mmgg}{%
	\ensuremath{\mathord{%
        \mathchoice%
        {\displaystyle\Gamma^{#1,#2}}%
        {\textstyle\Gamma^{\raise.2ex\hbox{$\scriptscriptstyle#1,#2$}}}%
        {\scriptstyle\Gamma^{#1,#2}}%
        {\scriptscriptstyle\Gamma^{\raise-.1ex\hbox{$\scriptscriptstyle#1,#2$}}}%
	\IfValueT{#3}{\parentheses*{#3\IfValueT{#4}{;#4}}}%
	}}\xspace%
	}

\NewDocumentCommand\diffforms{mmg}{%
	\ensuremath{\mathord{\Omega^{#1}%
	\parentheses*{#2\IfValueT{#3}{;#3}}%
	}}\xspace%
	}

\makeatletter
    \def\convarrow#1{%
        \mathchoice%
        {\displaystyle\downarrow^{#1}_0}%
        {\textstyle\kern-.7ex\downarrow^{#1}_0\kern-.2ex}%
        {\scriptstyle\kern-.2ex\downarrow^{#1}_0\kern.2ex}%
        {\scriptscriptstyle\downarrow^{#1}_0}%
    }
\makeatother

\DeclareMathOperator\paralleltransport{pt}

\newcommand\locprod{\ensuremath{\mathbin{\otimes_{\mathrm{loc}}}}}

\DeclareMathOperator\horlift{hl}
\NewDocumentCommand\hlfunc{sm}{%
    \ensuremath{\IfBooleanTF{#1}%
    {\parentheses*{#2}^{\mathrm{h}}}%
    {#2^{\mathrm{h}}}%
    }\xspace
}

\DeclareMathOperator\vertlift{vl}
\NewDocumentCommand\vlfunc{sm}{%
    \ensuremath{\IfBooleanTF{#1}%
    {\parentheses*{#2}^{\mathrm{v}}}%
    {#2^{\mathrm{v}}}%
    }
}

\DeclareMathOperator\tanlift{tl}

\newcommand\Ufield{\ensuremath{\mathcal{U}}\xspace}

\newcommand\Xfield{\ensuremath{\mathcal{X}}\xspace}
\newcommand\Yfield{\ensuremath{\mathcal{Y}}\xspace}

\NewDocumentCommand\spray{g}{\ensuremath{%
    \mathcal{H}%
    \IfValueT{#1}{_{#1}}%
    }\xspace%
    }
\newcommand\canonfield{\ensuremath{\mathcal{V}}\xspace}
\NewDocumentCommand\normalfield{g}{\ensuremath{%
    \mathcal{N}%
    \IfValueT{#1}{_{#1}}%
    }\xspace%
    }

\newcommand\configfold{\ensuremath{\mathsf{Q}}\xspace} 

\newcommand\Lagrange{\ensuremath{\mathsf{L}}\xspace}

\NewDocumentCommand\Hamilton{g}{\ensuremath{%
    \mathsf{H}%
    \IfValueT{#1}{_{#1}}%
    }\xspace%
}

\newcommand\Poissonbracket[2]{\ensuremath{\mathord{\left\{#1\mathbin{,}#2\right\}}}}

\theoremstyle{plain}
\newtheorem{theorem}{Theorem}[section]
\newtheorem{lemma}[theorem]{Lemma}
\newtheorem{corollary}[theorem]{Corollary}
\newtheorem{proposition}[theorem]{Proposition}

\theoremstyle{definition}
\newtheorem{definitionx}[theorem]{Definition}
\newenvironment{definition}%
    {\pushQED{\qed}\definitionx}%
    {\popQED\enddefinitionx}
\newtheorem{examplex}[theorem]{Example}
\newenvironment{example}%
    {\pushQED{\qed}\examplex}%
    {\popQED\endexamplex}
\newtheorem{conditionx}[theorem]{Condition}
\newenvironment{condition}%
    {\pushQED{\qed}\conditionx}%
    {\popQED\endconditionx}
\newtheorem{assumptionx}[theorem]{Assumption}
    {\pushQED{\qed}\assumptionx}%
    {\popQED\endassumptionx}
\newtheorem{problemx}[theorem]{Problem}
    {\pushQED{\qed}\problemx}%
    {\popQED\endproblemx}

\theoremstyle{remark}
\newtheorem{remarkx}[theorem]{Remark}
\newenvironment{remark}%
    {\pushQED{\qed}\remarkx}%
    {\popQED\endremarkx}
\newtheorem{notationx}[theorem]{Notation}
\newenvironment{notation}%
    {\pushQED{\qed}\notationx}%
    {\popQED\endnotationx}
\newtheorem{constructionx}[theorem]{Construction}
    {\popQED\endconstructionx}

\begin{document}

\title[Hypocoercivity of manifold-valued Langevin dynamics]%
    {Hypocoercivity of Langevin-type dynamics on abstract smooth manifolds}
\author[M.~Grothaus]{%
    Martin Grothaus%
    }
\address{%
    Department of Mathematics, Functional Analysis and Stochastic Analysis Group \\ %
    Technische Universit\"at Kaiserslautern \\ %
    P.\,O.\ Box 3049, 67663 Kaiserslautern, Germany}
\email{grothaus@mathematik.uni-kl.de}
\author[M.\,C.~Mertin]{%
    Maximilian Constantin Mertin %
    }
\address{%
    Department of Mathematics, Functional Analysis and Stochastic Analysis Group \\ %
    Technische Universit\"at Kaiserslautern \\ %
    P.\,O.\ Box 3049, 67663 Kaiserslautern, Germany}
\email{mertin@mathematik.uni-kl.de}
\date{\DTMdisplaydate{\the\year}{\the\month}{\the\day}{-1}} 

\makeatletter
    \@namedef{subjclassname@2020}{\textup{2020} Mathematics Subject Classification}
\makeatother


\begin{abstract}
    In this article we investigate hypocoercivity of Langevin-type dynamics in nonlinear smooth geometries.
    The main result stating exponential decay to an equilibrium state with explicitly computable rate of convergence 
    is rooted in an appealing Hilbert space strategy by Dolbeault, Mouhot and Schmeiser.
    This strategy was extended in~\cite{HypocoercJFA} 
    to Kolmogorov backward evolution equations in contrast to the dual Fokker-Planck framework.
    We use this mathematically complete elaboration 
    to investigate wide ranging classes of Langevin-type SDEs
    in an abstract manifold setting, 
    i.\,e.\ (at least) the position variables obey certain smooth side conditions.
    Such equations occur e.\,g.\ as fibre lay-down processes in industrial applications. 
    We contribute the Lagrangian-type formulation of such geometric Langevin dynamics 
    in terms of (semi-)sprays 
    and point to the necessity of fibre bundle measure spaces 
    to specify the model Hilbert space.
\end{abstract}

\maketitle\mbox{}
\pagenumbering{roman}
\thispagestyle{empty}

\pagenumbering{arabic}

\input{introduction}
\input{preliminaries}
\input{nonconstrained}
\input{constrained}
\input{ergodicity}

\begin{appendix}
\label{app}

\input{app-loc-coord}
\input{app-misc}

\end{appendix}

\clearpage

\bibliography{%
    lit-DiffGeom,%
    lit-Fibre_Lay-Down,%
    lit-Mechanics,%
    lit-StochAna,%
    lit-FuAna,%
    lit-Misc%
    }
\bibliographystyle{amsalpha}

\end{document}

%% file: introduction.tex
\section{Introduction}
\label{sec:intro}

A huge amount of research is going on in the area of hypocoercivity, hypoellipticity and diverse analytic methods 
to study long-time behaviour of degenerated stochastically perturbed systems. 
Herein, we concentrate on a hypocoercivity method applied to (geometric) Langevin equations, 
see~\cite{CoffeyKalmykovWaldron} for some background of these equations and physical or chemical applications.
For the hypocoercivity approach we think of Langevin equations 
as evolution equations of Kolmogorov backward type. 
Formulated as an abstract Cauchy problem in a Hilbert space 
a clever choice of an entropy functional gives rise of a certain norm on this Hilbert space 
measuring the desired exponential decay towards an equilibrium.
This is the fundamental idea by J.~Dolbeault, C.~Mouhot and C.~Schmeiser
for a hypocoercivity strategy, see~\cite{DMS15}.
However, we use the Kolmogorov backward (hypocoercivity) setting 
developed in~\cite{HypocoercJFA},  
because our focus lies on SDEs.
It's by no means clear how to apply their~\ref{thm:hypocoercivity-thm-named}
in case of stochastically perturbed mechanics on an abstract position manifold, 
since there are various approaches and terminologies:
Y.~Gliklikh discusses Langevin equations on manifolds of It\^o\-/type, 
see~\cite[Section~17]{Gliklikh}; 
in his opinion the It\^o~formulation is the most natural one. 
However, he rarely talks about generators
whereas by~\cite[Theorem~V.1.2]{IkedaWatanabe} one easily gets 
the infinitesimal generators of certain Stratonovich SDEs. 
Besides, in~\cite[Section~V.4]{IkedaWatanabe} it is explained 
how the theory of diffusions on manifolds in terms of Stratonovich SDEs
as motion in the frame bundle is strongly connected to It\^o's stochastic parallel displacement.
V.~Kolokoltsov uses a notion of stochastic Hamiltonian systems 
to study a `curvilinear Ornstein-Uhlenbeck processes' on the cotangent space, 
see~\cite[Chapter~4]{Kolokoltsov}. 
Using other ideas from classical mechanics 
one could investigate critical points of the stochastic Hamilton\-/Pontryagin action integral
see e.\,g.~\cite{Bou-RabeeOwhadi}, in particular~\cite[Theorem~3.2]{Bou-RabeeOwhadi}. 
This might be interesting from a computational point of view
and could be linked to the other formulations via local Lagrangian vector fields.
But Kolokoltsov's approach is directly based on a certain generator -- 
another approach of this kind can be found in~\cite{Soloveitchik}.
Kolokoltsov prefers local coordinate forms, 
however he is aware of~\cite{Joergensen} 
providing a construction of an `Ornstein-Uhlenbeck process' in the tangent space in invariant form.
E.~J{\o}rgensen uses the McKean-Gangolli injection scheme, 
see~\cite{McKean} and~\cite{Gangolli}, 
to construct his process projecting a process in the frame bundle 
which is a slightly inconvenient state space for real life applications.

The term (classical) Langevin equation refers in the purely Euclidean setting 
to the following system of equations:
\begin{align}\label{eq:classical-Langevin}
\begin{aligned}
    \diffd x_t 
    &= 
    v_t
    \ \diffd t
    \\
    \diffd v_t 
    &= 
    -\grad{}\Psi(x_t)
    \ \diffd t 
    + \sigma
    \stratonovich\diffd W_t 
    -\alpha\pdot v_t
    \ \diffd t
    ,
\end{aligned}
\end{align}
where $x_t\in\Rnum_x^\callidim{d}$ are positions in a space $\Rnum_x^\callidim{d}$ and 
$v_t\in\Rnum_v^\callidim{d}$ are velocities
for all times $t\in[0,\infty)$ respectively -- 
the spaces $\Rnum_x^\callidim{d}$ and $\Rnum_v^\callidim{d}$ are thought as 
independent copies of~$\Rnum^\callidim{d}$. 
The model parameter~$\alpha$ is interpreted as a friction parameter, similarly
$\sigma$~as a diffusion parameter -- 
both are nonnegative. 
The potential $\Psi\colon\,\Rnum_x^\callidim{d}\to\Rnum$ satisfies certain (weak) regularity properties 
and $W=(W_t)_{t\in[0,\infty)}$ is a $\callidim{d}$\=/dimensional Wiener process.
In fibre lay-down applications, 
where the Langevin equation is used as a surrogate model, 
one additionally would assume that 
the (Euclidean) norm of the velocities is~1 constantly, i.\,e.\ 
$v_t\in\sphere{\callidim{d}-1}\subseteq\Rnum_v^\callidim{d}$ for all times~$t$.

We follow the philosophy of describing a stochastic dynamic 
via its Kolmogorov backwards generator in invariant form.
In~\autoref{sec:Langevin}, 
we stay rather close to classical Lagrangian mechanics, 
since the `space of velocities~$v$' 
$\configfold = \tanbundle\posfold$ 
serves as configuration manifold\footnote{%
    E.\,g.\ in~\cite{MarsdenRatiu}, this is also called configuration \emph{space}, 
    however we reserve this term for the configuration space as a model of multiparticle systems. 
    The configuration space formalism might become handy, 
    if one studies fibre lay-down models of multiple filaments at once.
    } 
over the `space of positions~$x$' namely~\posfold,  
and we consider an evolution in the velocity phase space 
$\tanbundle\configfold = \tantanbundle\posfold$, 
so in a (double) tangent space. 
Together with the fact that 
Langevin-type equations are second order differential equations 
this leads quite naturally to Ehresmann connections and semisprays; 
both concepts are closely related on a purely geometric level.
We have been imbued with these ideas during the reading of works by I.~Bucataru
e.\,g.~\cite{habil-Bucataru,BucataruConstantinescuDahl}. 
In turn, those are rooted in results e.\,g.\ by M.~Crampin, J.~Grifone or R.~Miron. 
Later on, in~\autoref{sec:fld} 
we talk about fibre lay-down models 
and demonstrate how to geometrically implement an algebraic side condition 
like normalised velocities.
For sake of completeness, we also mention~\cite{Bismut}
wherein J.-M.~Bismut talks about Langevin processes in terms of hypoelliptic Laplacians, 
i.\,e.\ as a diffusion interpolating Brownian motion and geodesic flow.  

A key achievement of this paper is that
we explicitly don't need parallelisability of~\posfold.
It's a wide spread mistake to work in tangent bundles 
treating them like trivial bundles, 
even though it's very well-known that 
most of the spheres are not parallelisable -- 
e.\,g.~\sphere{2} as the Hairy Ball Theorem shows.
Indeed, since the works~\cite{Kervaire} by M.~Kervaire 
and independently \cite{BottMilnor} by R.~Bott and J.~Milnor it's known that
\sphere{\callidim{d}-1} is parallelisable exactly for $\callidim{d}\in\{1,2,4,8\}$.
E.\,g.\ E.~J\o{}rgensen was sensitised to the issue 
as the last remark in~\cite[Section 1]{Joergensen} shows.
If our model fails to capture the geometry of general spheres, 
we didn't find a reasonable model.
For that reason, we abolish notation like `$(x,v)\in\tanbundle\posfold$' 
with position~$x$ and velocity component~$v$ completely. 
Instead we will be very careful to always emphasise the bundle structure:
Throughout this paper, 
we always denote by~\tanproj a tangent bundle projection no matter what the base manifold is 
and extract the information on the position from a tangent vector~$v$ 
via $x = \tanproj(v)$ -- 
i.\,e.\ \tanproj~serves as an accessor or `getter' method.
In the absence of a product structure of tangent spaces,  
we rely on the `almost product' structure induced by the universal property of local trivialisation.
As the example of spheres illustrates, 
this is not just a technicality to be wiped out crudely with embeddings into larger Euclidean spaces.

Note that our configuration manifold is going to be $\configfold = \tanbundle\posfold$ -- 
or some sub\-/fibre bundle, see~\autoref{sec:fld}.
Therefore, we don't need~\posfold itself to be orientable. 
The commonly used integration of differential forms on manifolds instead of functions 
is extended by integration wrt.\ so-called 1\=/densities, 
where the latter concept does not need orientability.
We won't discuss this in~\autoref{sec:preliminaries}, 
but refer to~\cite[Section~11.4]{Analysis-Folland} and~\cite[\S~3.4.1]{Nicolaescu}  
as well as to~\cite[p.~34]{Heat-kernels-BGV} for the existence 
and to~\cite[Example~3.4.2]{Nicolaescu} for elementary properties 
of the canonical 1\=/density associated to the Riemannian metric.
Such a notion of integration on~\posfold is absolutely sufficient for our purposes.
The integration by parts formula, 
which enables us to use techniques related to generalised Dirichlet form theory in the first place, 
has to hold not for integration over the position manifold~\posfold, 
but for integration over the configuration manifold~\configfold 
which will automatically be orientable.
Hence, we will just talk about a `Riemannian volume measure \Leb{\riemannmetric*}' 
on the Riemannian manifold~$(\posfold,\riemannmetric*)$: 
either there is some orientation and this measure is induced 
by the canonical volume form~$\diffd\Leb{\riemannmetric*}$,  
or there is none and the measure is induced
by the canonical 1\=/density~$\abs{\diffd\Leb{\riemannmetric*}}$. 

In~\autoref{sec:ergo} we first prove existence of Markov processes 
solving the SDEs~\eqref{eq:geomLangevin} and~\eqref{eq:fld} treated in the preceding sections. 
By finding appropriate cores for the corresponding generators~\operator{L} 
and proving m-dissipativity in~\autoref{sec:Langevin} or~\autoref{sec:fld} respectively, 
we have a strongly continuous semigroup $(T_t)_{t\in[0,\infty)}$ generated by~\operator{L}
on the model Hilbert space~$H=\bigL2{\configfold}{\mu}$, 
where~$\mu$ is the equilibrium distribution. 
For a suitable test function $g\colon\,\configfold\to\Rnum$
the assignment $u(t,\cdot)=T_tg$ yields a solution 
to the abstract Cauchy problem $\dot{u}=Lu$ on the model Hilbert space.
Using the theory of generalised Dirichlet forms 
we show that there is a {\configfold}\=/valued Markov process $(X_t)_{t\in[0,\infty)}$
which solves the $L$\=/martingale problem, 
has~$\mu$ as invariant measure, 
is conservative, 
and most importantly it's properly associated to~\operator{L} in the resolvent sense.
The latter means that the transition resolvent 
$\int_{(0,\infty)} \exp(-as)\expect[\ID]{g(X_t)} \,\Leb(\diffd s)$
is a quasi-continuous $\mu$\=/version of the resolvent 
$\int_{(0,\infty)} \exp(-as)T_tg \,\Leb(\diffd s)$
for all $a\in(0,\infty)$ and functions~$g\in H$.
Second, we show that these Markov processes are {\bigL2}\=/exponentially ergodic 
in the sense of~\cite{ConradGrothausNparticle}.

We summarise our main results as follows:
\begin{itemize}
    \item
    We heavily make use of rather geometric concepts like semisprays and Ehresmann connections
    to formulate analytic problems and objects in invariant form.
    Thereby, we choose a quite accessible Lagrangian-type approach to higher-order SDEs on manifolds.
    Even though the geometric tools themselves are well-known, 
    they have not been used 
    to treat higher-order SDEs on abstract manifolds systematically.
    \item
    The applications of the general Hilbert space hypocoercivity strategy 
    presented in~\cite{HypocoercJFA} and~\cite{HypocoercMFAT}  
    are extended to the case of a quite abstract position manifold.
    During this course we see 
    that strong mixing with exponential convergence to equilibrium of the corresponding semigroups are features of Langevin-type equations 
    all across (finite-dimensional) smooth geometry with just view natural geometric assumptions.
    The main theorems are~\hyperref[thm:Langevin:main]{Theorem~\ref{thm:Langevin:main}} 
    and~\hyperref[thm:fld:main]{Theorem~\ref{thm:fld:main}}.
    \item 
    Among possible applications of Langevin equations on manifolds 
    we spotlight their usage as surrogate model of fibre lay-down processes 
    in industrial production of nonwovens.
    See~\autoref{sec:fld} and in particular~\hyperref[thm:fld:main]{Theorem~\ref{thm:fld:main}}.
    Compare these results to e.\,g.~\cite{GKMW07,KlarMaringerWegener,GS13}.
    \item 
    In~\autoref{sec:ergo} 
    we first prove in~\hyperref[thm:ergo:existence]{Theorem~\ref{thm:ergo:existence}} 
    the existence of a Markov process solving the martingale problem for our Langevin-type SDEs, 
    namely~\eqref{eq:geomLangevin} and~\eqref{eq:fld}.
    Afterwards we deduce that these solutions are {\bigL2}\=/exponentially ergodic in the sense of~\cite{ConradGrothausNparticle}, 
    meaning that the rate of ergodicity corresponds to exponential convergence of the semigroups.
    In the end, we argue that the {\bigL2}\=/exponential rate is even optimal.
    See~\hyperref[cor:ergodicity]{Corollary~\ref{cor:ergodicity}} and the subsequent remark.
\end{itemize}

%% file: preliminaries.tex
\section{Preliminaries}
\label{sec:preliminaries}

Before we give a brief recap on the general hypocoercivity method 
and establish several geometrical tools, 
we fix 
the assumptions on the position manifold.

\begin{condition}[Position manifold~\ref{cond:M}]
\namedlabel{cond:M}{(M)}
\hfill

\begin{enumerate}[label={(M\arabic*)}, ref={(M\arabic*)}]
    \item\label{cond:M:general}
    \emph{general geometry:}\ 
        Let $(\posfold,\riemannmetric*)$ be a 
        real, 
        finite dimensional, 
        connected 
        Riemannian manifold 
        with $\posdim\vdef\dim(\posfold)\geq2$.
    \item\label{cond:M:complete}
    \emph{completeness:}\ 
        Let~\posfold endowed with the intrinsic metric be a complete metric space\footnote{
            The intrinsic metric induces the same topology as given on~\posfold originally. 
            This is a consequence of the very definition of a (topological) manifold 
            as in~\cite{Lee}
            which we choose to use here.
        }.
\end{enumerate}
\vspace*{-1.5em}
\end{condition}

Note that following a common way of speaking,  
the manifold as described in~\ref{cond:M:general} has trivial boundary $\partial\posfold = \emptyset$. 
In this paper, we do not address boundary problems for sake of simplicity. 
From an analytic point of view, 
assumption~\ref{cond:M:complete} is plausible in it's own right
and necessary e.\,g.\ for a discussion of Sobolev spaces on noncompact manifolds, 
see~\cite[Chapter~3]{Hebey}.
By the Hopf-Rinow Theorem, 
we have that~\ref{cond:M:complete} is equivalent to geodesic completeness of~\posfold.
If we additionally assume an orientation on a manifold 
which satisfies~\ref{cond:M:general} and~\ref{cond:M:complete}, 
then it is well-known that
the usual Laplace-Beltrami operator on smooth test functions is essentially self-adjoint. 
See~\cite{Gaffney-harmonic-op,Gaffney-special-Stokes,Gaffney-heat-eq} 
for generalisations of the standard case on compact manifolds 
as well as~\cite[Section II.5]{LawsonMichelsohn}  
for general statements on so-called Dirac operators on spinors.
Moreover, we also want to recommend~\cite{Wolf73}.
If we would assume nontrivial boundary, 
then essentially self-adjointness of the Laplace-Beltrami operator is rather delicate.
This challenging problem has been tackled in~\cite{PhD-Perez-Pardo} 
and the noteworthy paper~\cite{ILP-JFA}.


\subsection{General hypocoercivity method}
\label{subsec:general-method}

At this point, we want to give an almost criminally brief overview of the hypocoercivity method 
in the abstract Hilbert space setting.
If the reader is familiar with this topic, 
this section just clarifies some notation.

Originally, the method was algebraically developed 
by J.~Dolbeault, C.~Mouhot and C.~Schmeiser
in~\cite{DMS09} and~\cite{DMS15} 
studying linear kinetic equations -- 
algebraically in the sense that
issues of operator domains have been neglected. 
Substantial contributions on the general hypocoercivity approach can be found 
in~\cite{Villani07} and~\cite{Villani09} -- 
e.\,g.\ the oncoming condition~\ref{cond:Langevin:P3} on the potentials Hessian appears in~\cite[Theorem~35]{Villani09} first.
Our main reference is~\cite{HypocoercJFA} for two reasons.
On conceptional side, 
the authors really tackled the long-time behaviour of solutions of the SDE
via the Kolmogorov backward setting 
instead of investigating the Fokker-Planck equation.
In the case of Euclidean position space, normalised velocity, 
and the invariant measure having a density wrt.\ Lebesgue measure
there is an isometric isomorphism the Kolmogorov backward to the Fokker-Planck setting, 
see~\cite{unpublishedGS}. 
However, in a general manifold it is not clear how to map between these two settings.
On technical side, 
to gain a rigorous proof of the~\ref{thm:hypocoercivity-thm-named} 
and to check its assumptions in applications 
several domain issues have to be taken into account, 
which are often times just omitted.
For broader discussion and more detailed explanations 
we also refer to~\cite{PhD-Stilgenbauer}.
The main result, we aim to apply to our fibre lay-down model,  
tells us what microstructure we have to expect in the fleece: 
The faster the convergence to the equilibrium state 
the more uniform the nonwoven material will appear.
The formal result reads as follows.

\begin{theorem}[Hypocoercivity Theorem]
\namedlabel{thm:hypocoercivity-thm-named}{Hypocoercivity Theorem}\label{thm:hypocoercivity-thm}
Assume the conditions~\ref{cond:D} as well as~\ref{cond:H} (as given below)
and denote by $(T_t)_{t\in[0,\infty)}$ 
the operator semigroup generated by~$L$ on the Hilbert space~$H$. 

Then, there exist constants $\kappa_1,\kappa_2\in(0,\infty)$ 
computable in terms of the constants $\Lambda_m$, $\Lambda_M$, $c_1$ and $c_2$ appearing in the assumptions 
such that it holds 
\begin{displaymath}
    \norm{ T_tg-\scalarprod{g}{1}{H} }{H} 
    \leq 
    \kappa_1 \euler^{-\kappa_2 t}\norm{ g-\scalarprod{g}{1}{H} }{H}
    \qquad\text{for all times } t\geq0
\end{displaymath}
and for all $g\in H$. 
\end{theorem}
\begin{proof}
    See~\cite[Theorem~2.18]{HypocoercJFA};
    in its proof one learns 
    how to compute the constants~$\kappa_1$ and~$\kappa_2$.
\end{proof}

Now, we are going to explain what the sets of conditions~\ref{cond:D} and~\ref{cond:H} are.
We just mention as a fact that there are more general data assumptions such that
the~\ref{thm:hypocoercivity-thm-named} is still valid, 
see~\cite[Section~2.2.3]{PhD-Stilgenbauer}.

\begin{condition}[Data conditions~\ref{cond:D}]
\namedlabel{cond:D}{(D)}
\hfill

\begin{enumerate}[label={(D\arabic*)}, ref={(D\arabic*)}]
    \item\label{cond:D1}
        \emph{model Hilbert space:}\ 
        Let $(E,\mathfrak{E},\mu)$ be a probability space and 
        define the Hilbert space~$H$ to be $\bigL2{E}{\mu}=\bigL2{\mu}$.
    \item\label{cond:D2}
        \emph{Strongly continuous semigroup and its infinitesimal generator:}\ 
        Let $(L,\opdomain{L})$ be a linear operator on~$H$ and 
        $(T_t)_{t\in[0,\infty)}$ be the strongly continuous semigroup generated by $L$, i.\,e.\ 
        $T_0 =\ID_H$ and $T_tf\rightarrow f$ as $t\downarrow0$ for all $f\in H$.
    \item\label{cond:D3}
        \emph{Core property:}\ 
        Let $D\subseteq\opdomain{L}$ be dense in~$H$ and 
        an operator core of $(L,\opdomain{L})$, i.\,e.\ 
        the closure of $(L,D)$ coincides with $(L,\opdomain{L})$.
    \item\label{cond:D4}
        \emph{SAD-Decomposition of generator into symmetric and antisymmetric part:}\ 
        Let $(S,\opdomain{S})$ be symmetric and 
        let $(A,\opdomain{A})$ be closed and antisymmetric on $H$ s.\,t.\ 
        $D \subseteq \opdomain{S}\cap\opdomain{A}$ 
        and the restriction of $L$ to the core can be decomposed as $L\vert_{D}=S-A$.
    \item\label{cond:D5}
        \emph{Projection:}\ 
        Let $P\colon\, H\to H$ be an orthogonal projection such that 
        $P(H)\subseteq\opdomain{S}$ and $SP=0$ as well as 
        $P(D)\subseteq\opdomain{A}$ and $AP(D)\subseteq\opdomain{A}$.
        Define 
        \begin{displaymath}
            P_S\colon\ H \longmapsin H,\ 
            f \longmapsto Pf + \scalarprod{f}{1}{H}
        .
        \end{displaymath}
    \item\label{cond:D6}
        \emph{Invariant measure:}\ 
        Let~$\mu$ be invariant for $(L,D)$ in the sense that 
        \begin{displaymath}
            \scalarprod{Lf}{1}{H} 
            = 
            \int_E Lf\ \diffd\mu 
            = 
            0
            \qquad\text{for all } f\in D
        .
        \end{displaymath}
    \item\label{cond:D7}
        \emph{semigroup conservativity:}
        Let $1\in\opdomain{L}$ and $L1=0$.
\end{enumerate}
\vspace*{-1.5em}
\end{condition}

\begin{condition}[Hypocoercivity conditions~\ref{cond:H}]
\namedlabel{cond:H}{(H)}
\hfill

\begin{enumerate}[label={(H\arabic*)}, ref={(H\arabic*)}]
    \item\label{cond:H1}
        \emph{algebraic relation:}
        \begin{displaymath}
            PAP\vert_D = 0
        \end{displaymath}
    \item\label{cond:H2}
        \emph{microscopic coercivity:}
        \begin{displaymath}
            \exists\,\Lambda_m\in(0,\infty)\,
            \forall\,f\in D\colon\
            \Lambda_m \|(\ID-P_S)f\|_H^2 \leq -\scalarprod{Sf}{f}{H}
        \end{displaymath}
    \item\label{cond:H3}
        \emph{macroscopic coercivity:}
        \begin{displaymath}
            \exists\, \Lambda_M\in(0,\infty)\,
            \forall\, f\in\opdomain{\adjoint{(AP)}(AP)} \colon\ 
            \Lambda_M \|Pf\|_H^2 \leq \|APf\|_H^2
        \end{displaymath}
    \item\label{cond:H4}
        \emph{boundedness of auxiliary operators:}
        \begin{align*}
            \exists\,c_1, c_2\in(0,\infty)\,\forall\,f\in D\colon{} 
            & \|BSf\|_H \leq c_1 \|(\ID-P_j)f\|_H \\
            \land\ 
            & \|BA(\ID-P)f\|_H \leq c_2 \|(\ID-P_j)f\|_H
        \end{align*}
        for $B\vdef (\ID-(AP)\adjoint{(AP)})\adjoint{(AP)}$ on~$D(\adjoint{(AP)})$
        and projections $P_j\in\{P,P_S\}$, $j\in\{1,2\}$.
\end{enumerate}
\vspace*{-1.5em}
\end{condition}


\subsection{Bundle measures and weighted bundles}
\label{sec:weighting}

We start with defining bundle measures:
Thinking of measures on some manifold~\basefold 
we should search for induced measures in bundles over~\basefold, 
i.\,e.\ natural measures on the respective total spaces.
Proceeding on the measure theoretic path in the direction of Radon-Nikod\'ym derivatives, 
the well-known notion of weighted manifolds is revisited.
The concept of multiplying a density weight function to a volume measure 
is not to be confused with conformal transformations of manifolds. 
Indeed, conformal transformations affect the geometry; 
the difference is the definition of the divergence and 
thus the integration by parts formula~\eqref{eq:weighted-ibp} below.
As far as we know, 
there is very few literature on measure theory paired with the more geometric concept of fibre bundles. 
After developing a notion of bundle measures on our own, 
we found the seemingly unknown paper~\cite{Goetz}, 
where the author concedes more degrees of topological freedom. 
Also, there is~\cite[Section~3.4.5]{Nicolaescu} 
wherein the author describes a `fibred calculus' just for smoothly indexed families of manifolds.
We give these references for sake of completeness, 
even though we stay with our bespoke construction of bundle measures.

Throughout this section, 
let $\bundleproj\colon\,\totalfold\longmapsin\basefold$ be a smooth fibre bundle 
with standard fibre~$F$ 
and define $\totalfold_b\vdef\bundleproj^{-1}(\{b\})$ for all $b\in\basefold$.
Furthermore, we suppose that
the fibre is actually a measurable space $(F,\mathfrak{F})$.
The property of local trivialisation yields a natural \sigmaalgebra 
$\borel{\basefold}\locprod\mathfrak{F}\vdef\sigma(\mathfrak{G})$ on~\totalfold with generator
\begin{displaymath}
    \mathfrak{G}
    \vdef
    \left\{
        \varphi^{-1}(U\times U_F)
    \ \middle\vert\ 
        \begin{matrix}
            b\in\basefold,\ 
            U_F\in\mathfrak{F},\ 
            W\text{ chart domain at~}b, \\ 
            U_b\subseteq\basefold \text{ open neighbourhood of~}b \text{ s.\,t.}\\
            \text{diffeomorphism }\varphi\text{ renders the diagram} \\
            \text{in~\autoref{fig:local-trivialisation} commutative},\\
            U\vdef U_b\cap W.
        \end{matrix}
    \right\}
.
\end{displaymath}
One might say that $\borel{\basefold}\locprod\mathfrak{F}$ is the canonical \sigmaalgebra on~\totalfold,  
however we will refer to it as the \emph{local product\-/\sigmaalgebra}. 
If $\mathfrak{F}=\borel{F}$, then 
$
    \borel{\basefold}\locprod\mathfrak{F} 
    = 
    \borel{\totalfold} 
$.
\begin{figure}[ht]\large
    \begin{tikzcd}
	    \arrow{r}{\varphi} \arrow[swap]{d}{\bundleproj} \bundleproj^{-1}(U_b) 
        & U_b\times F \arrow{dl}{\proj_{1}} \\
	    U_b & 
    \end{tikzcd}
    \caption{Local trivialisation of a fibre bundle}\label{fig:local-trivialisation}
\end{figure}

In the next definition, 
we restrict ourselves to probability measures 
as this is the sole situation of interest in the present paper.

\begin{definition}[bundle measure]
\label{def:bundle-measure}
    Let $\mu_{\basefold}$ and $\nu_F$ be probability measures 
    on $(\basefold,\borel{\basefold})$ and~$(F,\mathfrak{F})$ respectively.
    The \emph{(fibre) bundle measure on the total space \totalfold} is the pullback measure 
    $\mu_{\totalfold} \vdef \bundleproj^\ast \mu_{\basefold}$ 
    of $\mu_{\basefold}$ wrt.~\bundleproj supplemented with the fibre measure $\nu_F$.
    That is the unique probability measure on~\totalfold satisfying
    \begin{displaymath}
        \int_{\totalfold} f\ \diffd \mu_{\totalfold}
        =
        \int_{\basefold} \int_{\bundleproj^{-1}(\{b\})} %
            f\vert_{\bundleproj^{-1}(\{b\})}(e) %
        \ \diffd \nu_F(e)\,\diffd \mu_{\basefold}(b)
    \end{displaymath}
    for all bounded \measurable{\borel{\basefold}\locprod\mathfrak{F}}{\borel{\Rnum}} 
    functions $f\colon\,\totalfold\to\Rnum$.
\end{definition}

\begin{remark}
    In~\hyperref[def:bundle-measure]{Definition~\ref{def:bundle-measure}}, 
    the measure~$\nu_F$ on the fibre~$\bundleproj^{-1}(\{b\})$ is thought as an independent copy of~$\nu_F$ defined on~$F$.
    By definition a bundle measure yields a disintegration, 
    but obviously the former concept is motivated by the local product structure of fibre bundles 
    and tries to find an analogue of product measures respecting this structure, 
    whereas the latter is kind of the `factorisation' of measures 
    that are not necessarily product measures.
    In all generality, 
    it's a difficult problem to give conditions for a disintegration to exists;
    in our situation it appears as a byproduct.
\end{remark}

The following lemma is stated as an equivalent characterisation of bundle measures, 
however it could be formulated as an existence statement, since its proof is constructive.

\begin{lemma}[bundle measures locally are product measures]
\label{lem:loc-prod-measure}
    Let $\mu_{\basefold}$ and $\nu_F$ be probability measures on \basefold and $F$ respectively.
    Denote by $\mu_{\totalfold}\vdef \bundleproj^\ast \mu_{\basefold}$ the (fibre) bundle measure on \totalfold supplemented with fibre measure~$\nu_F$.
    Then, $\mu_{\totalfold}$ is the unique measure on \totalfold 
    which locally trivialises to the product measure $\mu_{\basefold}\mathbin{\otimes}\nu_F$ in the following sense: 
    By definition of fibre bundles, 
    for any $b\in\basefold$ there is a neighbourhood $U_b\subseteq\basefold$ of~$b$ 
    as well as a diffeomorphism~$\varphi$ 
    which renders the diagram in~\autoref{fig:local-trivialisation} commutative.
    Let $V\in\borel{\basefold}\locprod\mathfrak{F}$, $W$~a chart domain at~$b$ 
    and define $U\vdef W\cap U_b$ as well as $V_U\vdef V\cap \bundleproj^{-1}(U)$. 
    Then, $\mu_{\totalfold}$ obeys the transformation rule
    \begin{displaymath}
        \mu_{\totalfold}
        =
        \parentheses*{\varphi^{-1}}_\ast
        \parentheses*{
            \mu_{\basefold} \tensorprod \nu_F
        }
        \qquad\text{on }V_U
    .
    \end{displaymath}
    We say that $\mu_{\totalfold}$ is a loc\=/product measure and 
    introduce $\mu_{\totalfold}=\mu_{\basefold}\locprod\nu_F$ as the corresponding notation.
\end{lemma}
\begin{proof}
    First, we construct the bundle measure just from the local transformation rule --
    basically, that's the proof of existence of such a measure.
    Note that measurable sets of the same form as~$V_U$ 
    generate the local product\-/\sigmaalgebra on~\totalfold, 
    that they are linked to cylinder sets of the product\-/\sigmaalgebra on $U\times F$ via~$\varphi$,
    and that they form a family stable wrt.\ intersections, 
    when the empty set is included of course.
    Since $\varphi$ is continuously invertible, its inverse is measurable. 
    As the pushforward measure of~$\mu_{\totalfold}$ wrt.~$\varphi$ should be $\mu_{\basefold}\mathbin{\otimes}\nu_F$,  
    we declare~$\mu_{\totalfold}$ to be 
    the pushforward measure of $\mu_{\basefold}\mathbin{\otimes}\nu_F$ wrt.~$\varphi^{-1}$. 
    It satisfies the desired transformation rule by construction.
    By Fubini-Tonelli this is the particular instance of the integral equation 
    \begin{displaymath}
        \int_{\totalfold} f\ \diffd\mu_{\totalfold}
        =
        \int_{\basefold}\int_{\bundleproj^{-1}(\{b\})}
            f\vert_{\bundleproj^{-1}(\{b\})}(e)\
        \diffd\nu_F(e)\,\diffd \mu_{\basefold}(b)
    \end{displaymath}
    with $f=\indicator{V_U}$. 
    We use this to obtain the general equation via 
    approximating bounded measurable functions with simple functions and 
    apply Lebesgue dominated convergence.
    The uniqueness assertion is fulfilled by the general uniqueness theorem for measures 
    which finishes the proof.
\end{proof}

\begin{example}[
M\"obius strip]
    A M\"obius strip can be thought as fibre bundle
    with base manifold $\basefold = \sphere{1}$
    and standard fibre~$F$ being an open interval of finite length. 
    The natural volume measure on the strip coincides on~$\mathfrak{G}$
    with a bundle measure 
    where the base measure~$\mu_{\sphere{1}}$ is the volume measure on the circle 
    and the fibre measure is the Lebesgue measure on an interval: $\nu_F=\Leb$.
    Since~$\mathfrak{G}\cup\{\emptyset\}$ is a generator stable wrt.\ intersections, 
    the volume measure on the M\"obius strip coincides with the Lebesgue\-/type bundle measure.
    For a set $V_U\vdef V\cap\bundleproj^{-1}(U)$ we have that
    \begin{displaymath}
        \mu_{\totalfold}(V_U)
        =
        \int_{\bundleproj(V_U)\times\proj_2\bincirc\varphi(V_U)}
            \abs*{\det\parentheses*{d\varphi^{-1}}}
        \ \diffd \mu_{\basefold}\tensorprod\nu_F
    ,
    \end{displaymath}
    where $\proj_2$ denotes projection to the second component.
\end{example}

\begin{remark}
\label{rem:L2dense-testfunctions}
    At this point, it should be clear that 
    \bigL2{\totalfold}{\mu_{\basefold}\locprod\nu_F}
    is not isomorphic to the Hilbert space tensor product
    $
        \bigL2{\basefold}{\mu_{\basefold}}
        \tensorprod
        \bigL2{F}{\nu_F}
    $.
    Anyway, there might be nice spaces of test functions 
    that are dense in \bigL2\=/spaces for bundle measures.
    For sake of simplicity, 
    suppose that the base measure is absolutely continuous wrt.\ a given Riemannian volume measure 
    and moreover that 
    $F\subseteq\Rnum^\callidim{d}$ is a smooth submanifold, 
    $\mathfrak{F}\vdef F\cap\borel{\Rnum^\callidim{d}}$
    and $\nu_F$ is absolutely continuous wrt.\ restricted Lebesgue measure.
    Let $\tanproj^{-1}(U) = \varphi^{-1}(U\times F)\in\mathfrak{G}$. 
    By~\cite[Theorem~II.10]{ReedSimonI}
    we know that
    $\smoothcompact{U}\tensorprod\smoothcompact{F}$
    is dense in
    $\bigL2{U\times F}{\mu_{\basefold}\tensorprod\nu_F}$, 
    hence
    \begin{align*}
        \varphi_\ast\parentheses*{
            \smoothcompact{U}\tensorprod\smoothcompact{F}
        }
        &=
        \varphi_\ast\smoothcompact{U}\tensorprod\varphi_\ast\smoothcompact{F}
        \\
        &\vdef
        \linhull\left\{
            (f\pdot g) \bincirc\varphi
        \ \middle\vert\ 
            f\in\smoothcompact{U},
            g\in\smoothcompact{F}
        \right\}
    \end{align*}
    is dense in
    $\bigL2{\tanproj^{-1}(U)}{\mu_{\basefold}\locprod\nu_F}$. 
    From a partition of unity argument we can infer 
    that~\smoothcompact{\totalfold} is dense in 
    \bigL2{\totalfold}{\mu_{\basefold}\locprod\nu_F}.
\end{remark}
 
\begin{definition}[weighted (fibre) bundles]
\label{def:weighted-bundle}
    Consider functions
    $\rho_{\basefold}\colon\,\basefold\to[0,\infty)$
    and $\rho_F\colon\,F\to[0,\infty)$.
    A function~$\rho\colon\,\totalfold\to[0,\infty)$ is called 
    \emph{bundle weighting with base weight~$\rho_{\basefold}$ and fibre weight~$\rho_F$}
    if it locally trivialises to the product function 
    $\rho_{\basefold}\pdot\rho_F$ 
    in the following sense: 
    Consider the local trivialisation over an open neighbourhood~$U_b\subseteq\basefold$ at~$b$
    with diffeomorphism~$\varphi$ rendering the diagram in~\autoref{fig:local-trivialisation} commutative.
    Then, $\rho$~satisfies
    \begin{displaymath}
        \rho(v)
        =
        \rho_{\basefold}(\bundleproj(v))\pdot\rho_F(\proj_2\bincirc\varphi(v))
        \qquad\text{for all }v\in\bundleproj^{-1}(U_b)
    .
    \end{displaymath}
    A fibre bundle together with a bundle weight is called a \emph{weighted (fibre) bundle}.
\end{definition}

\begin{lemma}[existence and uniqueness of bundle weightings]
    In the situation of~\hyperref[def:weighted-bundle]{Definition~\ref{def:weighted-bundle}}
    a bundle weight~$\rho$ exists 
    and is uniquely determined by base weight and fibre weight.
    Therefore, we introduce the corresponding notation 
    $\rho = \rho_{\basefold}\locprod\rho_F$.
\end{lemma}
\begin{proof}
    In the local trivialisation 
    the only possible weight function is given via
    \begin{displaymath}
        \rho
        =
        ((\rho_{\basefold}\bincirc\proj_1) \pdot (\rho_F\bincirc\proj_2))\bincirc\varphi
        \qquad\text{on }\bundleproj^{-1}(U_b)
    .
    \end{displaymath}
    Fix an open cover of~\basefold by chart domains, 
    then the preimages form an open cover of~\totalfold. 
    With a partition of unity subordinate to the latter cover, 
    we can glue the definitions of~$\rho$ in the respective local trivialisations together
    which finishes the proof.
\end{proof}

\begin{remark}[local Radon-Nikod\'ym derivatives]
    Suppose some bundle measure $\mu_{\totalfold} = \mu_{\basefold}\locprod\nu_F$ such that
    $\mu_{\basefold}$ has a Radon-Nikod\'ym derivative 
    $\rho_{\basefold}\vdef\frac{d\mu_{\basefold}}{dm_{\basefold}}$ 
    wrt.\ another measure~$m_{\basefold}$ on~$(\basefold,\borel{\basefold})$ 
    and also $\nu_F$ has a Radon-Nikod\'ym derivative
    $\rho_F\vdef\frac{d\nu_F}{dm_F}$ 
    wrt.\ another measure~$m_F$ on~$(F,\mathfrak{F})$.
    Then, we think of the induced bundle weighting 
    $\rho = \rho_{\basefold}\locprod\rho_F$
    as a local Radon-Nikod\'ym derivative
    in view of it holds
    \begin{align*}
        \mu_{\totalfold} = \mu_{\basefold}\locprod\nu_F
        &=
        (\rho_{\basefold} m_{\basefold})\locprod(\rho_F m_F)
        \\
        &=
        (\rho_{\basefold}\locprod\rho_F) (m_{\basefold}\locprod m_F)
        =
        \rho (m_{\basefold}\locprod m_F)
    .
    \end{align*}
    For sake of brev'ty, we call~$\rho$ \emph{loc\=/density of~$\mu_{\totalfold}$}.
\end{remark}

\begin{remark}[Ehresmann connection and weighting]
    We emphasise again that weighting does not effect geometry: 
    The Ehresmann connection induced by the metric~\riemannmetric* on~\posfold, 
    see~\autoref{subsec:Ehresmann-Sasaki} below, 
    does not change in the course of the weighting procedure.
    Indeed, the corresponding Levi-Civita connection \connection{}{\riemannmetric*} 
    is not affected by reweighting as one can easily check 
    using the Leibniz rule when checking the metric compatibility condition.
    Also several geometric objects associated to this connection 
    just depend on the original metric~\riemannmetric* and not on the weight function.
\end{remark}

\begin{example}[trivial bundles]
    If the fibre bundle is trivial, 
    then bundle measures are pushforwards of product measures 
    wrt.\ the trivialisation isomorphism.
    Loc\=/densities basically are products of the densities for the respective measures.

    One particularly simple example arises with the standard fibre being a singleton.
    Then, every fibre measure is absolutely continuous wrt.\ the Dirac measure for the single point in~$F$
    and fibre weightings reduce to multiplication with a constant factor.
\end{example}

\begin{example}[weighted manifolds]
\label{ex:weighted-manifold}
    Consider an orientable Riemannian manifold~$(\basefold,\basemetric*)$ 
    and a strictly positive, nonconstant, smooth\footnote{%
        Smoothness is actually not required, 
        we could use e.\,g.\ a loc\=/Lipschitzian weight function instead.
        But in order to have~\hyperref[eq:weighted-ibp]{Equation~\eqref{eq:weighted-ibp}} 
        we need that at least locally a weak gradient of~$\rho_{\basefold}$ exists 
        and a weak version of Stokes Theorem.
    } weight function $\rho_{\basefold}\in\smoothfunc{\basefold}$. 
    The \emph{Riemannian metric weighted by~$\rho_{\basefold}$} 
    or just \emph{$\rho_{\basefold}$\=/weighted metric} on~\basefold
    is given as
    \begin{displaymath}
        \wbasemetric{v}{w} 
        \vdef
        \basemetric{\rho_{\basefold}v}{\rho_{\basefold}w} 
    \end{displaymath}
    for all $v,w\in\tanbundle{\basefold}$ with $\tanproj(v)=\tanproj(w)$.
    We do not keep the usual definition of gradients for this weighted metric, 
    but in fact choose the \emph{$\rho_{\basefold}$\=/weighted gradient} as 
    $\grad{\wbasemetric*} \vdef \frac1{\rho_{\basefold}}\grad{\basemetric*}$.
    I.\,e.\ for all smooth vector fields~$\Xfield$ on~\basefold 
    and functions $f\in\smoothfunc{\basefold}$ holds 
    $
        \wbasemetric{\Xfield}{\grad{\wbasemetric*}f}
        =
        \rho_{\basefold}\pdot\partial_{\Xfield}f
    $.
    The \emph{$\rho_{\basefold}$\=/weighted divergence} is to be defined as
    \begin{displaymath}
        \divergence{\wbasemetric*}
        \vdef
        \frac1{\rho_{\basefold}}\, \divergence{\basemetric*}(\rho_{\basefold}^2 \pdot\ID)
        ,
    \end{displaymath}
    since the \emph{$\rho_{\basefold}$\=/weighted Laplace-Beltrami operator}
    \begin{displaymath}
        \laplace{\wbasemetric*} 
        \vdef 
        \divergence{\wbasemetric*}(\grad{\wbasemetric*})
        = 
        \frac1{\rho_{\basefold}}\, 
        \divergence{\basemetric*}(\rho_{\basefold}\grad{\basemetric*})
    \end{displaymath}
    shall obey the following integration by parts formula:
    \begin{equation}
    \label{eq:weighted-ibp}
        \int_{\basefold} 
            \laplace{\wbasemetric*}f \pdot g
        \ \diffd\Leb{\wbasemetric*}
        =
        - \int_{\basefold}
            \wbasemetric{\grad{\wbasemetric*}f}{\grad{\wbasemetric*}g}
        \ \diffd\Leb{\wbasemetric*}
        =
        - \int_{\basefold}
            \basemetric{\grad{\basemetric*}f}{\grad{\basemetric*}g}
        \ \diffd\Leb{\wbasemetric*}
        =
        \int_{\basefold} 
            f \pdot \laplace{\wbasemetric*}g
        \ \diffd\Leb{\wbasemetric*}
    \end{equation}
    for all $f,g\in\smoothcompact{\posfold}$, 
    where~$\Leb{\basemetric*}$ and $\Leb{\wbasemetric*}=\rho_{\basefold}\Leb{\basemetric*}$ 
    refer to Riemannian volume measures. 
    If the weight function is constant, 
    then the definition of the gradient is not to be changed: 
    $\grad{\wbasemetric*}=\grad{\basemetric*}$.
\end{example}

\begin{notation}
\label{not:Lp-measure}
    For sake of readability, 
    we write~\bigL{p}{\basefold}{\wbasemetric*} 
    instead of~\bigL{p}{\basefold}{\Leb{\wbasemetric*}}.
    Similar notation is used for other spaces depending on a weighted Riemannian volume measure.
\end{notation}

\begin{remark}[adjoint vector fields wrt.\ weighted metric]
\label{rem:adjoint-sec}
    Let $(\basefold,\basemetric*)$ be an orientable Riemannian manifold.
    The general form of an adjoint vector field is given via the Divergence Theorem: 
    If \Xfield is a vector field on~\basefold, 
    then its adjoint wrt.\ the Riemannian metric, 
    i\,e.\ wrt.\ the \bigL2{\basefold}{\basemetric*}\=/scalar product, 
    is $\adjoint{\Xfield} = - \Xfield - \divergence_{\basemetric*}(\Xfield)$.
    Thus, solenoidal vector fields could be viewed as antisymmetric operators.
    Introducing a smooth, nonconstant weight function~$\rho_{\basefold}$ on~\basefold as above yields
    \begin{align*}
        \int_{\basefold} \Xfield f\pdot g \ \diffd\Leb{\wbasemetric*}
        &=
        \int_{\basefold} 
            - f\Xfield(\rho_{\basefold} g)
            - \divergence{\basemetric*}(\Xfield)\rho_{\basefold} fg 
        \ \diffd\Leb{\basemetric*}
        \\
        &=
        \int_{\basefold} 
            - \rho_{\basefold} f \Xfield g
            - fg \Xfield\rho_{\basefold}
            - \divergence{\basemetric*}(\Xfield)\rho_{\basefold} fg 
        \ \diffd\Leb{\basemetric*}
        \\
        &=
        \int_{\basefold} 
            - f \Xfield g
            - fg \,\frac1{\rho_{\basefold}}\, \Xfield\rho_{\basefold}
            - \divergence{\basemetric*}(\Xfield) fg 
        \ \diffd\Leb{\wbasemetric*}
    \end{align*}
    for all $f,g\in\smoothcompact{\basefold}$.
    Hence, the adjoint of~\Xfield wrt.~\wbasemetric* reads as 
    \begin{displaymath}
        \adjoint{\Xfield} 
        = 
        - \Xfield 
        - \divergence_{\basemetric*}(\Xfield) 
        - \frac1{\rho_{\basefold}}\Xfield\rho_{\basefold}
    .
    \end{displaymath}
    Therefore, we would correct the differential operator~\Xfield 
    for a vector field~\Xfield solenoidal wrt.~\basemetric* 
    by the \emph{logarithmic derivative of $\rho$ along~\Xfield} to become antisymmetric again, 
    i.\,e.\ replace~\Xfield by~$\Xfield+\frac1{\rho_{\basefold}}\Xfield\rho_{\basefold}$.
\end{remark}

\begin{lemma}[weighted Laplace-Beltrami in terms of logarithmic derivative]
\label{lem:weighted-Laplace}
    Again, let $(\basefold,\basemetric*)$ be an orientable Riemannian manifold 
    weighted by~$\rho_{\basefold}$ strictly positive, nonconstant and smooth.
    The weighted Laplace-Beltrami is written as
    \begin{displaymath}
        \laplace{\wbasemetric*} 
        = 
        \laplace{\basemetric*} 
        + 
        \frac1{\rho_{\basefold}}\,\basemetric{\grad{\basemetric*}\rho_{\basefold}}{\grad{\basemetric*}}
        =
        \laplace{\basemetric*} 
        + 
        \frac1{\rho_{\basefold}}\,\grad{\basemetric*}\rho_{\basefold}
    .
    \end{displaymath}
    We call the second summand the \emph{logarithmic derivative of $\rho_{\basefold}$}.
    In particular, for $\rho_{\basefold}=\exp(-\psi)$ with $\psi\in\smoothfunc{\basefold}$ 
    we obtain
    \begin{displaymath}
        \laplace{\wbasemetric*} 
        =
        \laplace{\basemetric*} - \grad{\basemetric*}\psi
    .
    \end{displaymath}
\end{lemma}
\begin{proof}
    Substantially, the proof looks like in the Euclidean case 
    as we need just the Leibniz rule, Stokes Theorem and the defining characterisation of the gradient.
\end{proof}

Later on, we will have to talk about Poincar\'e inequalities. 
For this purpose among others, 
we shall fix some notions concerning sections in general. 
In a nutshell, we are relying on the assumption that 
smooth vector fields viewed as first order differential operators 
could equivalently seen as smooth sections.
This is true in all applications we are interested in, 
but we are aware of counterexamples like noncommutative tori, 
see~\cite{Rosenberg}. 
In this paper, we just leave issues of noncommutative geometry aside 
and won't mention them again.

\begin{notation}[space of sections]
    We denote by~\measurablesec{\basefold}{\totalfold} the space of measurable sections.
    If no confusion is possible, 
    we omit the base manifold writing just~\measurablesec{\totalfold}.
    Moreover, we denote the space of $m$-times continuously differentiable sections 
    by~\contidiffsec{m}{\basefold}{\totalfold} or just~\contidiffsec{m}{\totalfold}. 
    As usual the differentiability parameter equals the regularity of the differentiable structure, 
    thus it is $m=\infty$ in our context for sake of simplicity.
\end{notation}

For the rest of this subsection, 
we consider~$F$ to be a Banach space, i.\,e.\ 
the fibre bundle is a Banach bundle, 
see e.\,g.~\cite[Chapter~III]{SergeLang}.
This gives us a section~$\abs{\cdot}{\totalfold}$ 
in the bundle of functions $\totalfold\to\Rnum$ over~\basefold such that
$
    \abs{\cdot}{\totalfold, b}\vdef\abs{\cdot}{\totalfold}(b)\colon\,
    \totalfold_b\to\Rnum
$ 
is a norm 
and $(\totalfold_b,\abs{\cdot}{\totalfold, b})$ is a Banach space
for all $b\in\basefold$.
One might think of~$\abs{\cdot}{\totalfold}$ as a `Riemannian norm', 
but should be careful since it is not necessarily related to a Riemannian metric.
Furthermore, we assume a $\sigma$\=/finite measure $\mu_{\basefold}$  on~$(\basefold,\borel{\basefold})$.

\begin{definition}[Integrable sections]
    Let $p\in[1,\infty)$ and $\Xfield\in\measurablesec{\totalfold}$.
    We call~\Xfield \emph{$p$\=/integrable wrt.~$\mu_{\basefold}$},
    if the integral  
    $
        \int_{\basefold}
            \abs{ \Xfield(b) }{\totalfold,b}^p
        \ \diffd\mu_{\basefold}(b)
    $
    is finite.
    By~\bigL{p}{\basefold\to\totalfold}{\mu_{\basefold}} 
    we denote the set of equivalence classes of $p$\=/integrable sections 
    wrt.\ equality {$\mu_{\basefold}$}\=/almost everywhere.
    Clearly, we endow~\bigL{p}{\basefold\to\totalfold}{\mu_{\basefold}} with the norm 
    \begin{displaymath}
        \norm{\cdot}{\bigL{p}{\basefold\to\totalfold}{\mu_{\basefold}}}
        \vdef
        \parentheses*{
            \int_{\basefold}
                \abs{ \cdot }{\totalfold,b}^p
            \ \diffd\mu_{\basefold}(b)
        }^{\sfrac1p}
    .
    \end{displaymath}
    As usual for the case $p=\infty$, we define~\bigL\infty{\basefold\to\totalfold}{\mu_{\basefold}}
    as the space of measurable sections $\Xfield\in\measurablesec{\totalfold}$ which are bounded almost everywhere, i.\,e.\ 
    there is a $c\in(0,\infty)$ such that 
    $\abs{ \Xfield(b) }{\totalfold,b} \leq c$
    for $\mu_{\basefold}$\=/almost all $b\in\basefold$.
    The norm of such a $\Xfield\in\bigL\infty{\basefold\to\totalfold}{\mu_{\basefold}}$ 
    is the infimum of all such bounds~$c$.
\end{definition}

Completeness of these spaces of integrable sections is shown 
like in the Fischer-Riesz Theorem for usual \bigL{p}\=/spaces.
One might say that we defined~\bigL{p}{\basefold\to\totalfold}{\mu_{\basefold}} 
as a `direct integral of Banach spaces'.
This is merely a verbalisation of a more general concept of 
integrating homoousios fibres wrt.\ some ($\sigma$\=/finite) measure on the base space.
This belongs to mathematical folklore and 
we can not discuss this construction here in exhausting detail.

\begin{example}[direct integral of Hilbert space fibres]
    Let \basefold be a {\callidim{b}}\=/dimensional Riemannian manifold, 
    the standard fibre~$F$ be $\Rnum^{\callidim{d}}$ for some natural number~\callidim{d}
    and let a section~\basemetric* 
    in the bundle of symmetric bilinear forms
    $\mathrm{Sym}(\totalfold)\rightarrow\basefold$
    such that it is pointwise positive semidefinite. 
    Then, \basemetric* induces a section of norms via 
    \begin{displaymath}
        \abs{ e }{\basemetric*}
        \vdef
        \basemetric{e}{e}{\bundleproj(e)}^{\frac12}
        \qquad\text{for all }e\in\totalfold
    \end{displaymath}
    and all fibres are Hilbert spaces, 
    since all norms on~$\Rnum^{\callidim{d}}$ are equivalent.
    Thus, we can describe the space of square-integrable sections 
    as the direct integral 
    $\int_{\basefold}^{\oplus} \totalfold_b \ \diffd\mu_{\basefold}(b)$
    wrt.~$\mu_{\basefold}$.
    This gives us a Hilbert space again.
    Of course, $(\basefold,\basemetric*)$ being a Riemannian manifold
    would be the most interesting case, i.\,e.\ 
    $\totalfold=\tanbundle\basefold$ with $F=\Rnum^{\callidim{b}}$ 
    and the section~\basemetric* being a Riemannian metric.
    Then, embracing~\hyperref[not:Lp-measure]{Notation~\ref{not:Lp-measure}} we have that
    \begin{displaymath}
        \bigL2{\basefold\to\tanbundle\basefold}{\basemetric*}
        =
        \int_{\basefold}^{\oplus} \tanbundle{\basefold}{b} \ \diffd\Leb{\basemetric*}(b)
    .
    \end{displaymath}
\end{example}

In principle, a function $u\in\Lloc1{\basefold}{\basemetric*}$ on a Riemannian manifold $(\basefold,\basemetric*)$
is weakly differentiable 
if the composition $u\bincirc b$ is weakly differentiable for every chart~$b$.
For the more formal (topological) definition of Sobolev spaces~\sobolev{m}{p}{\basefold} on~\basefold 
we refer to~\cite[Section~2.2]{Hebey}.
As in the case of Sobolev spaces, 
we are not content with characterising a weak gradient in charts 
as vector field with locally integrable components, 
but from the existence proof for Sobolev spaces as function spaces 
we know the coordinate-free description of weak gradients:
A function~$u$ is weakly differentiable 
if there is an element 
$\Ufield\in\Lloc1{\basefold\to\tanbundle\basefold}{\basemetric*}$ satisfying
\begin{displaymath}
    \int_{\basefold} 
        \basemetric{\Xfield}{\Ufield}
    \ \diffd\Leb{\basemetric*}
    =
    -
    \int_{\basefold} 
        \divergence{\basemetric*}(\Xfield) \pdot u
    \ \diffd\Leb{\basemetric*}
    \qquad\text{for all }\Xfield\in\smoothcompactsec{\tanbundle\basefold}
.
\end{displaymath}
We denote such a weak gradient as~$\Ufield = \grad{\basemetric*}u$ in analogy to the usual gradient.
In case of $u\in\sobolev1{p}{\basefold}$ 
we have that 
$\grad{\basemetric*}u\in\bigL{p}{\basefold\to\tanbundle\basefold}{\basemetric*}$.
The previous characterisation instances to
\begin{displaymath}
    \int_{\basefold} 
        \Ufield\varphi
    \ \diffd\Leb{\basemetric*}
    =
    \int_{\basefold} 
        \basemetric{\grad{\basemetric*}\varphi}{\Ufield}
    \ \diffd\Leb{\basemetric*}
    =
    -\int_{\basefold} 
        u\laplace{\basemetric*}\varphi
    \ \diffd\Leb{\basemetric*}
    \qquad\text{for all }\varphi\in\smoothcompact{\basefold}
.
\end{displaymath}
If 
$
    \int_{\basefold} 
        u \laplace{\basemetric*}\varphi
    \,\diffd\Leb{\basemetric*}
    =
    0
$ for all test functions $\varphi\in\smoothcompact{\basefold}$, 
then we call~$u$ \emph{weakly harmonic}.
Using this definition 
we bypass the problem to extend the divergence operator -- 
which is taking the trace of the covariant derivative -- 
with domain~\smoothsec{\tanbundle\basefold} 
to all the possible weak gradients in the Banach space \bigL{p}{\basefold\to\tanbundle\basefold}. 
In view of Weyl's Lemma characterising weakly harmonic functions 
those functions can be thought as harmonic functions, 
since we don't have boundaries in this paper.

\begin{example}[loc-Lipschitz potentials]
\label{ex:loc-Lipschitz}
    We know that
    if~$\psi$ is Lipschitzian\footnote{%
        We say that functions~$f\colon\,\basefold\to\Rnum$ are Lipschitzian if
        there is a positive constant~$C$ such that 
        $\vert f(x)-f(y) \vert \leq C d_{\basemetric*}(x,y)$ 
        for all $x,y\in\basefold$, 
        where we denote by~$d_{\basemetric*}$ the intrinsic metric induced by the Riemannian metric~\basemetric*.
        As a consequence, 
        if we carelessly embed~$\basefold$ into some~$\Rnum^{\callidim{d}}$ large enough, 
        the embedding might change the structure of the metric space and 
        thus the class of Lipschitz functions.
    } 
    with compact support, 
    then $\psi\in\sobolev1p{\basefold}$ for all $p\in[1,\infty]$
    meaning that~$\grad{\basemetric*}\psi$ exists as a weak gradient, 
    see~\cite[Proposition~2.4]{Hebey}. 
    Indeed, if~\basefold is a compact manifold, 
    then the proposition applies to any Lipschitzian~$\psi$; 
    in general, we have not a compact manifold, 
    but consider arbitrary compact subsets instead.
\end{example}

\subsection{Ehresmann connections and Sasaki metric}
\label{subsec:Ehresmann-Sasaki}

In this subsection, 
we introduce Ehresmann connections 
as a decomposition of the double tangent bundle in terms of a Whitney sum 
as well as the Sasaki metric on the tangent space 
as the Riemannian metric respecting the entire Ehresmann connection.

\begin{definition}[vertical bundle]
    Let $\bundleproj\colon\,\totalfold\longmapsin\basefold$ be a (smooth) fibre bundle. 
    The \emph{space of vertical (tangent) vectors} is 
    $\vertbundle{\totalfold}\vdef\nullspace{d\bundleproj}$, 
    the nullspace of the differential 
    $d\bundleproj\colon\,\tanbundle{\totalfold}\to\tanbundle{\basefold}$.
    Vertical vectors are thought as being tangent to the fibres of~\bundleproj.
    This yields the so-called \emph{vertical bundle} 
    $\tanproj\colon\,\vertbundle{\totalfold}\longmapsin\totalfold$. 
    As $d_v\bundleproj$ is surjective for all $v\in\totalfold$, 
    the vertical bundle is a smooth subbundle. 
    Smooth sections in this bundle are called \emph{vertical}.

    Additionally, if $\bundleproj\colon\,\totalfold\longmapsin\basefold$ happens to be a vector bundle, 
    than we can define the \emph{vertical lift at $v$}  
    $\vertlift_v\colon\ 
    \totalfold_{\bundleproj(v)} \to \tanbundle{\totalfold}{v}$
    for $v\in\totalfold$ fixed 
    via the action on arbitrary test functions $f\in\smoothfunc{\totalfold_{\bundleproj(v)}}$ as
    $
        \dualpair{\vertlift_v(w)}{d_vf}
        =
        \frac{d}{dt} f(v+tw)\bigr\vert_{t=0}
    $.
    If $\totalfold=\tanbundle\basefold$, then
    \begin{displaymath}
        \dualpair{\vertlift_v(w)}{d_v(df)}
        =
        \dualpair{w}{df}
        \qquad\text{for all }f\in\smoothfunc{\basefold} 
    \end{displaymath}
    determines the lift uniquely.
    The smooth section 
    $\canonfield\in\smoothsec{\totalfold}{\vertbundle{\totalfold}}$ given as 
    $v\mapsto\canonfield(v)\vdef\vertlift_v(v)$ 
    is called \emph{canonical vector field}.
    Furthermore, 
    the \emph{vertical projection \vertproj} is given at $v\in\totalfold$ as 
    the projection mapping $\tanbundle{\totalfold}{v}\to\vertbundle{\totalfold}{v}$.
\end{definition}

Now, we restrict ourselves to the case of vector bundles, but 
a definition of Ehresmann connections for fibre bundles can be found 
e.\,g.\ in~\cite[Section~9]{KolvarMichorSlovak}. 

\begin{definition}[Ehresmann connection]
    Let $\bundleproj\colon\,\totalfold\longmapsin\basefold$ be a (smooth) vector bundle. 
    A (smooth) subbundle $\horbundle{\totalfold}\leq\tanbundle{\totalfold}$
    is called~\emph{Ehresmann connection} or \emph{horizontal (tangent) bundle} if 
    \begin{displaymath}
        \tanbundle{\totalfold}{v} 
        = 
        \vertbundle{\totalfold}{v} 
        \oplus 
        \horbundle{\totalfold}{v}
        \qquad\text{for all }
        v\in\totalfold
        .
    \end{displaymath}
    For sake of readability we just write 
    $\tanbundle{\totalfold} = \vertbundle{\totalfold} \oplus \horbundle{\totalfold}$
    in the sense of a Whitney sum.

    The \emph{horizontal lift at $v$} of $w\in \totalfold$ is the unique vector $\horlift_v(w)\in\horbundle{\totalfold}{v}$ such that
    \begin{displaymath}
        w = \dualpair{\horlift_v(w)}{d\bundleproj}
        .
    \end{displaymath} 
    Finally, the projection of tangent vectors to their horizontal parts is denoted by~\horproj.
\end{definition}

Compare this usage of the terms `vertical' and `horizontal' 
to the usage in stochastic analysis as e.\,g.\ in~\cite[Section~V.4]{IkedaWatanabe} or~\cite[Chapter~2]{Hsu}.
Luckily, we escaped the frame bundle via the McKean-Gangolli injection scheme 
as described by E.~J\o{}ergensen.

Henceforth, we consider $\basefold = \posfold$ with the properties~\ref{cond:M}, 
since this is the sole instance of interest for this paper. 
Furthermore, we just take the tangent space $\totalfold = \tanbundle\posfold$, 
but note that we are interested in $\unittanbundle\posfold\subseteq\tanbundle\posfold$ 
in case of the fibre lay-down model.
Moreover, let us specify the one Ehresmann connection 
we always will consider without further mentioning: 
the Riemannian horizontal bundle.

\begin{definition}[connector map and Riemannian horizontal bundle]
    Let $U\subseteq\posfold$ be a neighbourhood of $o\in\posfold$ 
    with preimage $V\vdef \tanproj^{-1}(U)\subseteq\tanbundle{\posfold}$ such that
    the exponential $\exp_o\colon\,\tanbundle{\posfold}{o}\to\posfold$ 
    maps a 0\=/neighbourhood to~$V$ diffeomorphically. 
    Let $\tau\colon\,V\to\tanbundle{\posfold}{o}$ denote parallel transport of~$v\in V$ along the unique geodesic arc connecting~$\pi(v)$ and~$o$.
    Let $r_{-u}\colon\,\tanbundle{\posfold}{o}\to\tanbundle{\posfold}{o}$ 
    be the translation $w\mapsto w-u$ by the vector $u\in\tanbundle{\posfold}{o}$.
    Now, consider the mapping
    \begin{displaymath}
        \kappa\colon\ V\longmapsin\posfold,\ 
        v \longmapsto \parentheses*{\exp_{o} \bincirc r_{-u} \bincirc \tau}(v)
    .
    \end{displaymath}
    The dependency on the chart vanishes when passing to the differential  
    \begin{displaymath}
        d_u\kappa\colon\ \tantanbundle\posfold{u}\longmapsin\tanbundle\posfold{\tanproj(u)},\ 
        a \longmapsto \dualpair{a}{d(\exp_{\tanproj(u)} \bincirc r_{-u} \bincirc \tau)}
    \end{displaymath}
    which is called the \emph{connector map}, 
    cf.~\cite[Section~2]{Dombrowski}. 
    Also see~\cite[Section~II.4]{Sakai} for explanation in terms of local coordinates.
    Via the assignment
    $\horbundle{\totalfold} \vdef \nullspace{d\kappa}$ 
    we gain an Ehresmann connection
    which we call the \emph{Riemannian horizontal bundle}, 
    see e.\,g\ \cite[Appendix~(ii)]{Dombrowski}.
    As the exponential map depends on the given Riemannian metric, 
    so does this horizontal bundle.
\end{definition}

Some authors call $d\kappa$ the vertical projection 
and $d\tanproj$ the horizontal projection. 
E.\,g.\ in~\hyperref[not:func-lift]{Notation~\ref{not:func-lift}} 
we use the mappings in a way that would justify such a naming. 
However, we do not recommend this terminology 
and introduced the notions~`$\vertproj$' and~`$\horproj$' to avoid confusion. 
Concerning the vertical or horizontal lift of functions 
there is a well-established consent what it should be.

\begin{notation}[vertical/horizontal lift of functions]
\label{not:func-lift}
    Let~$f_0$ be a real\-/valued function with domain in~\posfold.
    We call the pullback of~$f_0$  wrt.\ the tangent bundle projection~\tanproj
    the \emph{vertical lift of~$f_0$}. 
    For sake of brev'ty we define $\vlfunc{f_0}\vdef\tanproj^\ast f_0$ 
    whenever it's defined.

    A direct analogy would be that 
    the horizontal lift of~$f_0$ is the pullback wrt.\ the connector map 
    and we declare $\hlfunc{f_0}\vdef\kappa^\ast f_0$
    whenever it's defined.
    However, this is not straight forward, 
    since~$\kappa$ depends on the choice of base point~$o$ of the exponential and the translation vector~$u$.
    At first, the horizontal lift of a smooth function~$f_0$ is a function
    \begin{displaymath}
        (u,v)\longmapsto \parentheses*{f_0\bincirc\exp_{\tanproj(u)}\bincirc r_{-u}\bincirc \tau}(v)
    \end{displaymath}
    or a bevy of functions indexed by~$u$; 
    we should exclude the case that~$u$ and~$v$ are in the same fibre 
    as this would yield the vertical lift again. 
    We do not worry too much about well-definedness and differentiability: 
    Due to assumption~\ref{cond:M:complete} and the Hopf-Rinow Theorem 
    the exponential mapping at~$\tanproj(u)$ is defined everywhere on the tangent space,
    furthermore it is almost everywhere a diffeomorphism by~\cite[Lemma~III.4.4]{Sakai}.
    But as aforementioned earlier 
    we can project from the double tangent space into the tangent space via~$d\tanproj$ and~$d\kappa$ 
    and get 
    $df_0 (d\tanproj + d\kappa) = d\vlfunc{f_0} + df_0\bincirc d\kappa$.
    The \emph{horizontal lift of~$f_0$} is a function~\hlfunc{f_0} such that 
    $
        \dualpair{a}{d\hlfunc{f_0}} 
        =
        \dualpair{a}{df_0\bincirc d\kappa} 
    $
    for all $a\in\tantanbundle\posfold$ -- 
    algebraically this makes perfectly sense and is unique up to constant offsets.
    A more analytic intuition would be 
    $
        \hlfunc{f_0}(v)
        =
        \int_{\tantanbundle\posfold{v}} df_0\bincirc d\kappa\,\diffd\Leb
    $, 
    where we consider the Lebesgue measure on the standard fibre $F=\Rnum^{2\posdim}$ of the double tangent bundle.
\end{notation}

\begin{lemma}
\label{lem:vert-hor-acting}
    Let $f_0\in\smoothfunc{\posfold}$ 
    and $\Xfield\in\smoothsec{\tanbundle\posfold}$.
    Then, the following relations hold
    \begin{align*}
        \vertlift(\Xfield)\vlfunc{f_0} = 0
        \quad&\text{and}\quad
        \horlift(\Xfield)\vlfunc{f_0} = \vlfunc*{\Xfield f_0},
        \\
        \vertlift(\Xfield)\hlfunc{f_0} = \vlfunc*{\Xfield f_0}
        \quad&\text{and}\quad
        \horlift(\Xfield)\hlfunc{f_0} = 0
    .
    \end{align*}
\end{lemma}
\begin{proof}
    Elementary.
\end{proof}

Next, we define the Sasaki metric. 
It's original definition was given in the famous paper~\cite{Sasaki58} 
containing many other results of general interest. 
Sasaki refered just to the Riemannian horizontal bundle and so do we, 
but it's clear that the following definition can easily be adapted to a situation without connector map.

\begin{definition}[Sasaki metric]
\label{def:Sasaki}
    The \emph{Sasaki metric \sasakimetric* on \tanbundle{\posfold}} 
    is uniquely characterised as 
    the Riemannian metric respecting inner products in the given Ehresmann connection, i.\,e.\ 
    the metric \sasakimetric* is \emph{natural} in the sense of~\cite{GudmundssonKappos} meaning that 
    \begin{displaymath}
        \sasakimetric{\horlift(\Xfield)}{\horlift(\Yfield)}
        =
        \riemannmetric{\Xfield}{\Yfield} \bincirc\tanproj
        \qquad\text{and}\qquad
        \sasakimetric{\horlift(\Xfield)}{\vertlift(\Yfield)}
        =
        0
    \end{displaymath}
    for all $\Xfield,\Yfield\in\smoothsec{\tanbundle{\posfold}}$,  
    and additionally the metric respects the inner product under vertical lifting in the sense that
    \begin{displaymath}
        \sasakimetric{\vertlift(\Xfield)}{\vertlift(\Yfield)}
        =
        \riemannmetric{\Xfield}{\Yfield} \bincirc\tanproj
    \end{displaymath}
    for all $\Xfield,\Yfield\in\smoothsec{\tanbundle{\posfold}}$.  
    This metric is explicitly given as the sum of pullbacks of the metric tensor of the base manifold:
    \begin{align*}
        \sasakimetric{a}{b}
        &{}\vdef
        \kappa^\ast\riemannmetric{a}{b}
        +
        \tanproj^\ast\riemannmetric{a}{b}
        \\
        &=
        \riemannmetric{\dualpair{a}{d\kappa}}{\dualpair{b}{d\kappa}}
        +
        \riemannmetric{\dualpair{a}{d\tanproj}}{\dualpair{b}{d\tanproj}}
    \end{align*}
    for all $a,b\in\tantanbundle{\posfold}$.
    We call the component $\hormetric*\vdef\tanproj^\ast\riemannmetric*$ the \emph{horizontal (Sasaki) metric}.
    Mutatis mutandis, the \emph{vertical (Sasaki) metric} \vertmetric* is $\vertmetric*\vdef\kappa^\ast\riemannmetric*$.
    This yields $\sasakimetric* = \vertmetric* + \hormetric*$.
\end{definition}

\begin{notation}
    If we consider the $\rho_{\posfold}$\=/weighted manifold $(\posfold,\wriemannmetric*)$, 
    then the weighting procedure is naturally reflected 
    in the horizontal bundle via the vertical lift.  
    Explicitly, we think of $(\tanbundle\posfold,\whormetric*)$ 
    as a $\vlfunc{\rho_{\posfold}}$\=/weighted Riemannian manifold, 
    where~\whormetric* is the weighted version of the horizontal Sasaki metric~\hormetric*. 
    In the first place, the vertical bundle is not affected.
    If we weight the fibre $F=\Rnum^{\posdim}$ by~$\rho_F$, 
    then the vertical metric changes and the horizontal one does not. 
    The weighted vertical Sasaki metric is denoted by~\wvertmetric*.
    If and only if a bundle weighting is specified and no confusion possible,
    we denote the corresponding weighted Sasaki metric by~\wsasakimetric*
    and the weighted tangent bundle just as $(\tanbundle\posfold,\wsasakimetric*)$.
\end{notation}

It turns out that the volume form corresponding to the Sasaki metric --
    it exists independently of orientability of~\posfold 
    as tangent spaces always are orientable --
induces a volume measure 
with very neat loc-product structure.

\begin{lemma}[Sasakian volume measure]
    The volume measure~\Leb{\sasakimetric*} wrt.\ the Sasaki metric 
    coincides with the bundle measure on~\tanbundle{\posfold} 
    supplemented with Lebesgue fibre measure, 
    i.\,e.\ it holds 
    $\Leb{\sasakimetric*} = \Leb{\riemannmetric*} \locprod \Leb$, 
    where we abbreviated the $\posdim$\=/fold Lebesgue measure 
    $\Leb^{\otimes\posdim}$ just by~\Leb.
\end{lemma}
\begin{proof}
    Form~\hyperref[def:Sasaki]{Definition~\ref{def:Sasaki}}
    we know that the Sasaki metric can be written as the sum 
    $\sasakimetric* = \vertmetric* + \hormetric*$ 
    of vertical and horizontal metric.
    We encounter this situation when considering a product of Riemannian manifolds and 
    thus, we know that the Sasakian volume measure is a product measure basically. 
    Indeed, let a chart $(v^j)_{j=1}^{2\posdim}$ with domain $V\subseteq\tanbundle{\posfold}$
    respecting the Ehresmann connection 
    in the sense that $(\partial v^j)_{j=1}^{\posdim}$ provides local basis for 
    either the vertical or the horizontal vector fields and 
    $(\partial v^j)_{j=\posdim+1}^{2\posdim}$ provides a basis for the complementary type of vector fields;
    the matrix representation of the Sasaki metric is of block diagonal form 
    with zero matrix at south west and north east position. 
    Therefore, the Sasakian volume form~$\diffd\Leb{\sasakimetric*}$ reads in those coordinates as 
    \begin{displaymath}
        \diffd\Leb{\sasakimetric*}
        =
        \sqrt{\abs*{\tanproj^\ast\riemannmetric*}}
        \sqrt{\abs*{\kappa^\ast\riemannmetric*}}
        \bigwedge_{j=1}^{2\posdim} \diffd v^j
        .
    \end{displaymath} 
    Wlog.\ $V$ is preimage of a chart domain in~\posfold, i.\,e.\ $V=\tanproj^{-1}(U)$, 
    such that there is a diffeomorphism~$\varphi$
    rendering the diagram in~\autoref{fig:local-trivialisation} commutative.
    Hence the pushforward measure $\varphi_\ast\Leb{\sasakimetric*}$ has to coincide 
    with product volume measure on $U\times F$, 
    where the fibre $F=\Rnum^{\posdim}$ is naturally endowed with the Lebesgue measure. 
    But~\hyperref[lem:loc-prod-measure]{Lemma~\ref{lem:loc-prod-measure}}
    uniquely determines the bundle measure on~\tanbundle{\posfold} 
    supplemented the fibre measure $\nu_F=\Leb$, 
    so the Sasakian volume measure has to be this bundle measure.
\end{proof}

Note that by~\hyperref[lem:vert-hor-acting]{Lemma~\ref{lem:vert-hor-acting}} 
we already know that 
the gradients~\grad{\vertmetric*} and~\grad{\hormetric*} wrt.\ vertical and horizontal Sasaki metric respectively
satisfy $\grad{\vertmetric*}\vlfunc{f_0} = 0$ and $\grad{\hormetric*}\hlfunc{f_0} = 0$.
In the next lemma we give more insight into the three basic operators 
induced by the metrics from the Ehresmann connection.

\begin{lemma}[Sasakian gradient, divergence and Laplacian]
\label{lem:Sasaki-operators}
    The gradient, divergence and Laplace-Beltrami operators 
    corresponding to the vertical Sasaki metric are characterised by
    \begin{displaymath}
        \grad{\vertmetric*}\hlfunc{f_0}
        =
        \vertlift\parentheses*{\grad{\riemannmetric*}f_0\bincirc\tanproj}
        ,\
        \divergence{\vertmetric*}\parentheses*{\vertlift{\Xfield}}
        =
        \parentheses*{\divergence{\riemannmetric*}\Xfield}\bincirc\tanproj
        \text{ and }
        \laplace{\vertmetric*}\hlfunc{f_0}
        =
        \parentheses*{\laplace{\riemannmetric*}f_0}\bincirc\tanproj
    \end{displaymath}
    for all $f_0\in\smoothfunc{\posfold}$ 
    and $\Xfield\in\smoothsec{\tanbundle{\posfold}}$. 
    Similarly, for the case of the horizontal Sasaki metric we have that
    \begin{displaymath}
        \grad{\hormetric*}\vlfunc{f_0}
        =
        \horlift\parentheses*{\grad{\riemannmetric*}f_0\bincirc\tanproj}
        ,\
        \divergence{\hormetric*}\parentheses*{\horlift{\Xfield}}
        =
        \parentheses*{\divergence{\riemannmetric*}\Xfield}\bincirc\tanproj
        \text{ and }
        \laplace{\hormetric*}\vlfunc{f_0}
        =
        \parentheses*{\laplace{\riemannmetric*}f_0}\bincirc\tanproj
    \end{displaymath}
    for all $f_0\in\smoothfunc{\posfold}$ 
    and $\Xfield\in\smoothsec{\tanbundle{\posfold}}$. 

    Eventually, we have 
    \begin{displaymath}
        \divergence{\sasakimetric*}(\Yfield)
        =
        \divergence{\vertmetric*}(\vertproj(\Yfield))
        +
        \divergence{\hormetric*}(\horproj(\Yfield))
        \qquad\text{for all }\Yfield\in\smoothsec{\tanbundle\posfold}{\tantanbundle\posfold}
    .
    \end{displaymath}
\end{lemma}
\begin{proof}
    We restrict ourselves to the vertical case, since the other statements follow analogously.
    Let $f_0\in\smoothfunc{\posfold}$ 
    and $\Xfield\in\smoothsec{\tanbundle{\posfold}}$ arbitrary be fixed. 
    Then, by definition and from~\hyperref[lem:vert-hor-acting]{Lemma~\ref{lem:vert-hor-acting}} 
    we know that the following two equations hold simultaneously 
    which characterises the vertical gradient:
    \begin{align*}
        \vertmetric{\vertlift\Xfield}{\grad{\vertmetric*}\hlfunc{f_0}}
        &=
        (\Xfield f_0)\bincirc\tanproj
        \\
        \text{and}\quad
        \vertmetric{\vertlift\Xfield}{\vertlift\parentheses*{\grad{\riemannmetric*}f_0}}
        &=
        \riemannmetric{\Xfield}{\grad{\riemannmetric*}f_0}
        =
        (\Xfield f_0)\bincirc\tanproj
    .
    \end{align*}
    Regarding the divergence, 
    we use the definition in terms of the Lie derivative, 
    which is mostly preferred in manifold theory, 
    and apply Cartan's magical formula.
    We have for the Riemannian metrics \riemannmetric* and~\vertmetric* the defining equations of divergence:
    \begin{align*}
        \Liederivative{\Xfield} \Leb{\riemannmetric*}
        &=
        \divergence{\riemannmetric*}(\Xfield) \pdot \Leb{\riemannmetric*}
        \\
        \text{and}\quad
        \Liederivative{\vertlift\Xfield} \Leb{\vertmetric*}
        &=
        \divergence{\vertmetric*}(\vertlift\Xfield) \pdot \Leb{\vertmetric*}
    ,
    \end{align*}
    where `\Liederivative' denotes the Lie derivative.
    Again, the symbol~\Leb{\riemannmetric*} is to be interpreted 
    as either volume form or 1\=/density.
    Now, in case of a volume form Cartan's formula yields the equation
    \begin{align*}
        \Liederivative{\vertlift\Xfield} \Leb{\vertmetric*}
        &=
        d\parentheses*{\vertlift(\Xfield)\insertmorphism \kappa^\ast\Leb{\riemannmetric*}}
        +
        \vertlift(\Xfield)\insertmorphism d\kappa^\ast\Leb{\riemannmetric*}
        =
        \parentheses*{d \Xfield\insertmorphism\Leb{\riemannmetric*}}\bincirc\tanproj
        +
        \vertlift(\Xfield)\insertmorphism d\kappa^\ast\Leb{\riemannmetric*}
        \\
        &=
        \parentheses*{d \Xfield\insertmorphism\Leb{\riemannmetric*}}\bincirc\tanproj
        +
        \parentheses*{\Xfield\insertmorphism d\Leb{\riemannmetric*}}\bincirc\tanproj
        =
        \parentheses*{\Liederivative{\Xfield} \Leb{\riemannmetric*}}\bincirc\tanproj
    .
    \end{align*}
    Thus, the statement follows; 
    this can locally be reused for the case of a 1\=/density, 
    cf.\ the proof of~\cite[Proposition~3.4.3]{Nicolaescu}.
    Combining the results for gradient and divergence 
    we get the equation for the Laplace-Beltrami operator.

    Finally, let $\Yfield\in\smoothsec{\tanbundle\posfold}{\tantanbundle\posfold}$.
    Recall that 
    the Sasakian volume measure is realised by a volume form~$\diffd\Leb{\sasakimetric*}$ 
    which could be written as wedge product of a vertical and horizontal volume form: 
    $
        \diffd\Leb{\sasakimetric*}
        =
        \diffd\Leb{\vertmetric*}
        \wedge
        \diffd\Leb{\hormetric*}
    $. 
    Then, the divergence wrt.\ Sasaki metric is characterised by 
    \begin{align}
        \Liederivative{\Yfield} \diffd\Leb{\sasakimetric*}
        &=
        \Liederivative{\Yfield}
        \parentheses*{
            \diffd\Leb{\vertmetric*} \wedge \diffd\Leb{\hormetric*}
        }
        =
        \Liederivative{\Yfield}
            \diffd\Leb{\vertmetric*}
        \wedge
        \diffd\Leb{\hormetric*}
        +
        \diffd\Leb{\vertmetric*}
        \wedge
        \Liederivative{\Yfield}
             \diffd\Leb{\hormetric*}
        \notag
        \\
        &\begin{aligned}
        =
        \Liederivative{\vertproj(\Yfield)}
            \diffd\Leb{\vertmetric*}
        \wedge
        \diffd\Leb{\hormetric*}
        +
        \diffd\Leb{\vertmetric*}
        \wedge
        \Liederivative{\horproj(\Yfield)}
             \diffd\Leb{\hormetric*}
        \end{aligned}
        \label{eq:vertical-horizontal-Lie-derivative}
        \\
        &=
        \parentheses*{
            \divergence{\vertmetric*}(\vertproj(\Yfield))
            +
            \divergence{\hormetric*}(\horproj(\Yfield))
        }
        \pdot
        \diffd\Leb{\vertmetric*}
        \wedge
        \diffd\Leb{\hormetric*}
        =
        \divergence{\sasakimetric*}(\Yfield) \pdot \diffd\Leb{\sasakimetric*}
        \notag
    ,
    \end{align}
    where we used the local coordinate description of Lie derivatives in line~\eqref{eq:vertical-horizontal-Lie-derivative}.
\end{proof}

\begin{remark}[Sasaki gradient of loc-density]
\label{rem:grad-loc-density}
    Basically, the Sasakian gradient of 
    a tangent bundle weighting~$\rho = \rho_{\posfold}\locprod\rho_F$ 
    decomposes as 
    \begin{displaymath}
        \grad{\sasakimetric*}\rho
        =
        \vlfunc{\rho_{\posfold}} \pdot \grad{\vertmetric*}\rho_F
        \oplus
        \rho_F \pdot \grad{\hormetric*}\vlfunc{\rho_{\posfold}}
    \end{displaymath}
    reminiscent of the Leibniz rule.
    Here, we think of~$\rho_F$ as an independent copy defined on the pointwise tangent spaces. 
    Actually, the symbol 
    $\grad{\vertmetric*}\rho_F\vdef\vertproj(\grad{\sasakimetric*}\rho)$
    is an intuitive short hand for the vertical component of the local vector field 
    $
        \varphi_2^{-1}(\ID,\grad{\mathrm{euc}}\rho_F\bincirc\proj_2\bincirc\varphi_1)
    $, 
    where the diffeomorphisms $\varphi_1$, $\varphi_2$ render the diagram 
    in~\autoref{fig:trivialisation-for-grad} commutative.
    \begin{figure}[ht]\large
    \begin{tikzcd}
	    \arrow{r}{\varphi_2} 
        \arrow[swap]{d}{\tanproj} 
        \arrow[swap, bend right=55]{dd}{\tantanproj} 
        \tantanproj^{-1}(U)
        & 
        \tanproj^{-1}\parentheses*{ U\times\Rnum^{\posdim} }
        \arrow{d}{\tanproj} 
        \arrow{r}{\simeq} 
        &
        \tanproj^{-1}(U)\times\Rnum^{\posdim}
        \arrow{dl}{(\tanproj,\ID)} 
        \\
	    \arrow{r}{\varphi_1} 
        \arrow[swap]{d}{\tanproj} 
        \tanproj^{-1}(U) 
        & 
        U\times\Rnum^{\posdim} \arrow{dl}{\proj_{1}} 
        &
        \\
	    U & &
    \end{tikzcd}
    \caption{Local trivialisation reflected in the double tangent bundle}\label{fig:trivialisation-for-grad}
    \end{figure}
\end{remark}

We define the basic concept of semisprays 
which arises in the abstract study of second order ordinary differential equations -- 
particularly, of Newton's (second) law on manifolds, see~\cite[Section~4]{Gliklikh} --,  
and also encodes geometric extra structure in several interesting ways.
As a consequence of the discussion, 
the direct sum $\canonfield\oplus\spray$ of canonical vector field and semispray~\spray 
is understood as the diagonal mapping
\begin{displaymath}
    \tanbundle\posfold\to\tantanbundle\posfold,\ 
    v\longmapsto \vertlift(v)\oplus\horlift(v)
.
\end{displaymath}

\begin{definition}[Semispray]
\label{def:semispray}
    A section $\spray\in\smoothsec{\tanbundle{\posfold}}{\tantanbundle{\posfold}}$ 
    is a \emph{semispray} if it satisfies
    $\dualpair{\spray}{d\tanproj} = \ID_{\tanbundle{\posfold}}$ 
    or equivalently if
    any integral curve $s\colon\,\indexset\to\tanbundle{\posfold}$ 
    takes the form $s = (\tanproj\bincirc s)'$.
    A curve $c\colon\,\indexset\to\posfold$ is called \emph{geodesic of the semispray \spray} if
    there is an integral curve $s\colon\,\indexset\to\tanbundle{\posfold}$ 
    such that $c = \tanproj\bincirc s$.
    Equivalently, $c$ is geodesic if $\spray\bincirc c' = c''$.
    For a local coordinate form of semisprays 
    see the~\autoref{app:loc-coord}.
\end{definition}

\begin{example}[Semispray associated to an Ehresmann connection]
\label{ex:Riemann-spray}
    Consider an Ehresmann connection 
    $\tantanbundle{\posfold} = \vertbundle{\tanbundle{\posfold}} \oplus \horbundle{\tanbundle{\posfold}}$.
    Then, given $v\in\tanbundle{\posfold}$ there is a unique horizontal vector $a$ -- 
    namely the horizontal lift of $v$ -- 
    such that $v = \dualpair{a}{d\tanproj}$. 
    Furthermore, $a$~depends on~$v$ smoothly.
    Hence, there is $\spray\in\smoothsec{\tanbundle{\posfold}}{\horbundle{\tanbundle\posfold}}$
    such that $\ID_{\tanbundle{\posfold}} = \dualpair{\spray}{d\tanproj}$. 
    Then, \spray~is a semispray 
    and we say that it is the \emph{semispray associated to the Ehresmann connection}. 
    In particular, there is the semispray~$\spray_{\riemannmetric*}$ 
    associated to the Riemann metric~\riemannmetric* via the corresponding Ehresmann connection.
    We call it \emph{Riemannian semispray}. 

    Another common name is \emph{geodesic spray}, 
    since it can be constructed just in terms of geodesics corresponding to the Riemannian metric, 
    see~\cite[Lemma~III.2.3]{doCarmo}. 
    So, the geodesics of this semispray are just classical geodesics.
    In fact, this semispray satisfies a homogeneity condition making it a (full) spray 
    whence the name -- 
    we will use the spray structure a few times indeed.
\end{example}

\begin{remark}
\label{rem:Ehresmann-connection-semispray-relation}
    In a nutshell, choosing an Ehresmann connection adds the same geometric information 
    as choosing a covariant derivative or a semispray. 
    See~\cite{habil-Bucataru} and various references therein.
    We think that this observation might be a good starting point for 
    generalisations even to rough geometries, 
    as it translates well to Lagrange and Finsler spaces.
\end{remark}

\begin{example}[Semisprays induced by Lagrangians]
\label{ex:spray-as-gradient}
    Let the Lagrangian $\Lagrange = \frac12\,\abs{\cdot}{\riemannmetric*}^2$.
    Then, the Lagrangian vector field~\spray{\Lagrange} 
    is a semispray\footnote{In fact, it is the geodesic spray again.}, 
    see~\cite[Section~7.5]{MarsdenRatiu}.
    As outlined in~\cite[Section~VII.6]{SergeLang} 
    this relation can be translated to general Riemannian manifolds modelled on some Hilbert space. 
    Note that existence of such a vector field is due to 
    the more general result~\cite[Proposition~VII.5.9]{SergeLang}: 
    Let $\configfold\vdef\tanbundle\posfold$, 
    then the cotangent space of~\configfold naturally\footnote{%
        Up to sign conventions: 
        Let~$\Theta$ denote the canonical 1\=/form, 
        then we define the 2\=/form~$\Omega$ by $\Omega=-d\Theta$.
    } yields the structure of a symplectic manifold
    $(\tanbundle*\configfold,\Omega)$.
    Furthermore, we make~\configfold into a Riemannian manifold
    endowing it with a `natural metric' in the sense of Gudmundsson-Kappos.
    This specifies the musical isomorphisms, in particular~$\sharp$.
    Thus, there is a unique section 
    $\spray{\Lagrange}\in\smoothsec{\tanbundle\configfold}$ such that
    \begin{displaymath}
        \parentheses*{\Omega^\sharp\bincirc\spray{\Lagrange}}_w(a)
        \vdef
        \Omega_w^\sharp(\spray{\Lagrange}(w),a) 
        = 
        d_w\Lagrange(a)
        \qquad\text{for all }w\in\configfold,a\in\tanbundle\configfold{w}
    , 
    \end{displaymath}
    where $\Omega^\sharp$ denotes the pullback of~$\Omega$ wrt.~$\sharp$.
    With that, we can say that 
    $
        \spray{\Lagrange} 
        = 
        (d\,\Lagrange)^\sharp
        =
        \grad\Lagrange
    $, 
    where the gradient is taken wrt.\ the `natural' metric chosen before.
    See~\cite[Section~3]{metric-nonlin-connections-Bucataru} 
    for another short discussion of semisprays on general Lagrange spaces.
\end{example}

\begin{remark}[Revisiting {\hyperref[ex:Riemann-spray]{Example~\ref{ex:Riemann-spray}}}]
\label{rem:Riemann-spray-revisited}
    We shall characterise the Riemannian semispray using~\cite[Proposition~VII.5.9]{SergeLang}
    similar to the previous example.
    Again, let $\configfold\vdef\tanbundle\posfold$ 
    and the symplectic manifold $(\tanbundle*\configfold,\Omega)$.
    Let $\hormetric*=\tanproj^\ast\riemannmetric*$ be the horizontal metric, 
    which is the simplest `natural' metric,  
    and denote the pullback of the canonical 2\=/form~$\Omega$ wrt.\ this metric by~\hlfunc{\Omega}.
    Locally we can think of 
    $a^\prime,b^\prime\in\tanbundle*\configfold{w}$, $w\in\configfold$,  
    as tuples $a^\prime=(u_1,u_2^\prime)$ and $b^\prime=(v_1,v_2^\prime)$ 
    with $u_1,u_2,v_1,v_2\in\configfold$ 
    and 
    $
        u_2^\prime=\riemannmetric{\cdot}{u_2}{x},
        v_2^\prime=\riemannmetric{\cdot}{v_2}{x}\in\configfold^\ast
    $, 
    where $x\vdef\tanproj(w)$.
    Then, the form~{\hlfunc{\Omega}}\xspace reads as
    \begin{align*}
        \hlfunc{\Omega}_w((u_1,u_2),(v_1,v_2))
        &=
        \riemannmetric{u_1}{v_2}{x}
        -
        \riemannmetric{v_1}{u_2}{x}
        \\
        &=
        \dualpair{u_1}{v_2^\prime} 
        - 
        \dualpair{v_1}{u_2^\prime}
        =
        \Omega_w(a^\prime,b^\prime)
    .
    \end{align*}
    There is a 1\=/form~$\omega$ 
    which reads in such local regimes as 
    $\omega_w(b)=\omega_w(v_1,v_2)=\riemannmetric{w}{v_1-v_2}{\tanproj(w)}$.
    The unique section $\spray\in\smoothsec{\tanbundle\configfold}$ 
    such that $\hlfunc{\Omega}\bincirc\spray=\omega$ 
    is the Riemannian semispray.
\end{remark}

\begin{example}[Euclidean case]
\label{ex:spray-acting}
    The Riemannian semispray~\spray{\riemannmetric*} 
    acts on vertically lifted functions~$f=\vlfunc{f_0}$ with $f_0\in\smoothfunc{\posfold}$ as 
    \begin{displaymath}
        \spray{\riemannmetric*}f
        =
        \hormetric%
            {\spray{\riemannmetric*}}%
            {\grad{\hormetric*}\vlfunc{f_0}}
        =
        \riemannmetric%
            {\ID_{\tanbundle\posfold}}%
            {\grad{\riemannmetric*}f_0 \bincirc\tanproj}%
            {\tanproj}
    .
    \end{displaymath}
    In case of $\posfold = \Rnum^{\posdim}$ with standard Riemannian metric 
    this action is written as  
    $   \spray{\mathrm{euc}}f(x,v) 
        = 
        \scalarprod{v}{\nabla_x f_0(x)}{\mathrm{euc}}
    $
    for smooth functions 
    $f\colon\,\Rnum^{\posdim}_x\times\Rnum^{\posdim}_v\to\Rnum,\,(x,v)\mapsto f_0(x)$ 
    and $x,v\in\Rnum^{\posdim}$.

    By analogy, the canonical vector field~\canonfield acts on
    horizontally lifted functions~$f=\hlfunc{f_0}$ with $f_0\in\smoothfunc{\posfold}$ as 
    \begin{displaymath}
        \canonfield f
        =
        \vertmetric%
            {\canonfield}%
            {\grad{\vertmetric*}\hlfunc{f_0}}
        =
        \riemannmetric%
            {\ID_{\tanbundle\posfold}}%
            {\grad{\riemannmetric*}f_0 \bincirc\tanproj}%
            {\tanproj}
    .
    \end{displaymath}
    In the Euclidean case,  
    this action is written as  
    $   \canonfield f(x,v) 
        = 
        \scalarprod{v}{\nabla_v f_0(v)}{\mathrm{euc}}
    $
    for smooth functions 
    $f\colon\,\Rnum^{\posdim}_x\times\Rnum^{\posdim}_v\to\Rnum,\,(x,v)\mapsto f_0(v)$ 
    and $x,v\in\Rnum^{\posdim}$.
\end{example}

The following theorem is well-known, 
but comes in a few quite different formulations e.\,g.: 
The geodesic flow preserves the volume of~\tanbundle{\posfold}. 
We just stick to the formulation below, 
since it tells us that 
the Riemannian semispray is an antisymmetric operator 
wrt.~\bigL2{\tanbundle\posfold}{\sasakimetric*}\=/scalar product. 

\begin{theorem}[Liouville's Theorem]
\label{thm:Liouville-thm}
\namedlabel{thm:named:Liouville-thm}{Liouville's Theorem}
    The semispray $\spray{\riemannmetric*}$ is solenoidal wrt.\ Sasaki metric~\sasakimetric*.
\end{theorem}
\begin{proof}
    A proof could be done via 
    explicitly calculating the Sasakian volume form in normal coordinates.
    This approach is taken from the neat book by M.\ do~Carmo, 
    cf.~\cite[Exercise~3.14]{doCarmo}.
\end{proof}

The final lemma of this section 
characterises the test functions on the tangent space 
exploiting again the `almost product' structure due to local trivialisation.
This result will be very important arguing 
for a reasonable set of test functions being a core of the (Langevin/fibre lay-down) generator.

\begin{lemma}[Tensor product of lifted test functions]
\label{lem:D0isdense}
    Define the tensor product space of pulled back test functions
    \begin{align*}
        D_0
        &{}\vdef
        \tanproj^\ast\smoothcompact{\posfold}
        \tensorprod
        \kappa^\ast\smoothcompact{\posfold}
        \\
        &=
        \linhull
        \left\{ 
            \tanproj^\ast f_0 
            \tensorprod 
            \kappa^\ast g_0 
            \vdef
            \tanproj^\ast f_0 
            \pdot 
            \kappa^\ast g_0 
            \ \middle\vert\ 
            f_0, g_0\in\smoothcompact{\posfold}
        \right\}
        \\
        &{}=
        \linhull
        \left\{ 
            \vlfunc{f_0} 
            \tensorprod 
            \hlfunc{g_0} 
            \ \middle\vert\ 
            f_0, g_0\in\smoothcompact{\posfold}
        \right\}
    .
    \end{align*}
    Then, $D_0$ is dense\footnote{%
        Dense wrt.~the usual locally convex topology induced by appropriate seminorms, 
        which implies uniform convergence of all derivatives on compacts; 
        for the well-known Euclidean case see~\cite[Example~2.4.10]{Horvath} 
        and for the generalised geometric case see~\cite[Section~3.1.3]{GKOS}.
        In common notation for test function spaces with this topology, 
        the result reads as 
        $\tanproj^\ast\mathcal{D}(\posfold)\tensorprod\kappa^\ast\mathcal{D}(\posfold)$
        is dense in
        $\mathcal{D}(\tanbundle\posfold)$.
        By the way, 
        the authors of~\cite{GKOS} not only discuss generalised distribution spaces rigorously,  
        they shed some light on integration of 1\=/densities from a very analytic perspective. 
    }
    in $\smoothcompact{\tanbundle{\posfold}}$.
\end{lemma}
\begin{proof}
    Consider a local trivialisation in $U\vdef U_x\cap W$ 
    for a chart domain $W\subseteq\posfold$ at $x\in\posfold$
    with $V\vdef\tanproj^{-1}(U)$ and the diffeomorphism~$\varphi$
    rendering the diagram in~\autoref{fig:local-trivialisation} commutative.
    Then, we immediately conclude the relations
    \begin{displaymath}
        \tanproj^\ast\smoothcompact{U}
        =
        \varphi^\ast\parentheses*{ \smoothcompact{U}\tensorprod\{1\} }
        \qquad\text{and}\qquad
        \kappa^\ast\smoothcompact{U}
        =
        \varphi^\ast\parentheses*{ \{1\}\tensorprod\smoothcompact{\Rnum^\posdim} }
        .
    \end{displaymath}
    Thus, by~\cite[Proposition~4.8.1]{Horvath} the tensor product 
    $
    \tanproj^\ast\smoothcompact{U}
    \tensorprod
    \kappa^\ast\smoothcompact{U}
    $
    is dense in~%
    $
    \smoothcompact{V} 
    =
    \varphi^\ast\smoothcompact{U\times\Rnum^{\posdim}}
    $
    and the proof is finished via a partition of unity argument.
\end{proof}

%% file: nonconstrained.tex
\section{Hypocoercivity for geometric Langevin dynamics}
\label{sec:Langevin}

In this section we apply the abstract Hilbert space hypocoercivity method
to the Langevin equation with some Riemannian manifold as position space, 
see~\hyperref[thm:Langevin:main]{Theorem~\ref{thm:Langevin:main}} below.
That is the direct generalisation of the situation of~\cite{HypocoercMFAT}.
The techniques of proving we learned from~\cite{HypocoercJFA}. 
Our interest arose from industrial fibre lay-down applications 
and qualitative analysis of the nonwoven.
In this context, the position manifold~\posfold could e.\,g.\ 
reflect a sagging conveyor belt 
or a belt moving over a cylindrical roller.
As we mentioned in the introduction, 
Langevin dynamics have wide ranging applications
and our approach offers the freedom to include any `smooth' side condition on the position variable.

Consider the following Stratonovich~SDE in~\tanbundle\posfold: 
\begin{equation}
\label{eq:geomLangevin}
    \diffd \eta
    =
    \spray{\riemannmetric*}
    \ \diffd t
    +
    \vertlift_\eta(-\grad{\riemannmetric*}\Psi) 
    \ \diffd t
    + 
    \sigma \pdot \vertlift_\eta\left( %
        \sum_{j=1}^\posdim \frac{\partial}{\partial{x_\eta^j}} %
        \right)
    \stratonovich\diffd W_t
    -
    \alpha \pdot \canonfield
    \ \diffd t
,
\end{equation}
where $\eta\colon\,\indexset\to\tanbundle\posfold$ is a curve with time interval~\indexset 
and $\parentheses*{x_\eta^1,x_\eta^2,\ldots,x_\eta^\posdim}$ is a chart at~$\tanproj(\eta)$ 
providing normal coordinates. 
Recall that~\canonfield denotes the canonical vector field.
Bellow we specify certain assumptions on the potential $\Psi\colon\,\posfold\to\Rnum$.
The nonnegative model parameters~$\alpha$ and~$\sigma$ are related by 
$\sigma = \sqrt{\sfrac{2\alpha}\beta}$, 
where~$\beta$ is a nonnegative rescaling of the potential as $\Phi=\beta\Psi$. 
Recall that 
the horizontal motion in \hyperref[eq:geomLangevin]{Equation~\eqref{eq:geomLangevin}}, i.\,e.\ 
$
    \horproj(\diffd \eta)
    =
    \spray{\riemannmetric*}
    \ \diffd t
$,
reflects the natural requirement of 
$
    (\tanproj \bincirc \eta)^\prime
    =
    \eta
$ 
which is the form that any integral curve of~\spray{\riemannmetric*} attains by definition. 
We call~\hyperref[eq:geomLangevin]{Equation~\eqref{eq:geomLangevin}}
the \emph{Langevin equation on~\posfold} 
or just \emph{geometric Langevin equation}.

Using~\cite[Theorem~V.1.2]{IkedaWatanabe} 
the Kolmogorov generator~$L$ is given as
\begin{equation}
\label{eq:Langevin-generator}
    L
    =
    \spray{\riemannmetric*}
    -\frac1\beta\,
    \grad{\vertmetric*}\hlfunc{\Phi}
    +
    \frac\alpha\beta\,\laplace{\vertmetric*}
    -
    \alpha\canonfield
.
\end{equation}
This operator is defined for all smooth functions on~\tanbundle\posfold, 
but we consider the domain $D\vdef\smoothcompact{\tanbundle\posfold}$ of smooth test functions 
with compact support.
Note that $(L,D)$ is densely defined 
by~\hyperref[rem:L2dense-testfunctions]{Remark~\ref{rem:L2dense-testfunctions}}.
We usually call~$L$ as in~\hyperref[eq:Langevin-generator]{Equation~\eqref{eq:Langevin-generator}}
the \emph{Langevin generator}.
Compare this generator to~\cite[Equation~(1.2)]{HypocoercMFAT} 
taking~\hyperref[ex:spray-acting]{Example~\ref{ex:spray-acting}} into account.
As we will see in~\hyperref[lem:Langevin:SAD]{Lemma~\ref{lem:Langevin:SAD}}, 
the Langevin generator basically decomposes into two components: 
the vertical diffusion 
and the (not entirely horizontal) component liaising the appropriate notion of second order differential equations.
Compare this decomposition to so-called hypoelliptic Laplacians; 
via this concept J.-M.~Bismut links Brownian motion on manifolds and geodesic flow 
in order to find a Langevin process in~\cite{Bismut}. 

Before we start checking the data and hypocoercivity conditions, 
we shall fix conditions on the potential in the geometric Langevin equation, 
and thus on the base weight
$
    \rho_{\posfold}
    =
    \exp(-\Phi)
    =
    \exp(-\beta\Psi)
$. 

\begin{condition}[Potential conditions~\ref{cond:named:Langevin:P}]
\label{cond:Langevin:P}
\namedlabel{cond:named:Langevin:P}{(P)}
\hfill

\begin{enumerate}[label={(P\arabic*)}, ref={(P\arabic*)}]
    \item\label{cond:Langevin:P1}
        \emph{General regularity and boundedness:}\ 
        Let $\Phi=\beta\Psi$ a loc\=/Lipschitzian potential 
        which is bounded from below 
        and such that
        $
            \Leb{\wriemannmetric*}
            =
            \rho_{\posfold}\,\Leb{\riemannmetric*}
            =
            \exp(-\Phi)\,\Leb{\riemannmetric*}
        $ 
        is a probability measure on $(\posfold,\borel{\posfold})$.
    \item\label{cond:Langevin:P2}
        \emph{Poincar\'e inequality:}\ 
        The weighted Riemannian measure~\Leb{\wriemannmetric*}
        satisfies the Poincar\'e inequality
        \begin{equation}
        \label{eq:Langevin:Poincare}
            \norm{ %
                \grad{\riemannmetric*} f_0 %
            }{\bigL2{\posfold\to\tanbundle\posfold}{\wriemannmetric*}}^2
            \geq
            \Lambda
            \norm{ %
                f_0-\scalarprod{f_0}{1}{\bigL2{\posfold}{\wriemannmetric*}} %
            }{\bigL2{\posfold}{\wriemannmetric*}}^2 
        \end{equation}
        for all $f_0\in\smoothcompact{\posfold}$
        and some $\Lambda\in(0,\infty)$.  
    \item\label{cond:Langevin:P3}
        \emph{Hessian dominated by gradient:}\ 
        Assume $\Phi\in\contidiff2{\posfold}$.
        There is a constant $c\in(0,\infty)$ such that
        \begin{displaymath}
            \abs{ \operatorname{Hess}_{\riemannmetric*}(\Phi)(x) } 
            \leq 
            c\, (1+\abs{\grad{\riemannmetric*}\Phi(x)}{\riemannmetric*})
            \qquad\text{holds for all } x\in\posfold
        .
        \end{displaymath}
        Here, `$\operatorname{Hess}_{\riemannmetric*}(\cdot)$' 
        denotes the Hessian wrt.\ the given Riemannian metric 
        and the norm `$\abs{\operatorname{Hess}_{\riemannmetric*}(\cdot)}$' 
        of the Hessian is the Frobenius tensor norm induced by the Riemannian metric.
\vspace*{-3ex}
\end{enumerate}
\end{condition}

As explained in~\cite[Remark~3.16]{HypocoercJFA} 
the condition~\ref{cond:Langevin:P3} as above can be weakened.
Clearly, the~\hyperref[eq:Langevin:Poincare]{Poincar\'e inequality~\eqref{eq:Langevin:Poincare}}
is the most restrictive condition wrt.\ 
the geometry of the weighted position manifold~$(\posfold,\wriemannmetric*)$.
It's satisfied for compact manifolds, see~\cite[Theorems~2.10, 2.11]{Hebey}. 
For the noncompact case we refer to~\cite[Lemma~3.1]{Hebey}, 
where the necessary geometric assumption is that 
the Ricci curvature is bounded from below by some multiple of~\wriemannmetric*. 
Furthermore, we point out that 
in~\hyperref[lem:Langevin:SAD]{Lemma~\ref{lem:Langevin:SAD}}
the nasty assumption of a weakly harmonic potential appears. 
In view of Weyl's theorem it would be quite a restrictive assumption. 
That's why 
in~\hyperref[prop:Langevin:SAD]{Proposition~\ref{prop:Langevin:SAD}} 
we remove this assumption 
for the general setting of this section.
Note that by~\hyperref[lem:weighted-M-complete]{Lemma~\ref{lem:weighted-M-complete}}
the Riemannian manifold $(\posfold,\wriemannmetric*)$ weighted by a potential as above
is complete again.
\smallskip

We formulate the main theorem of this section:
\begin{theorem}[Hypocoercivity of the geometric Langevin dynamic]\hfill
\label{thm:Langevin:main}

    Let $\alpha,\beta\in(0,\infty)$ 
    and $(\posfold,\riemannmetric*)$ be a Riemannian manifold satisfying~\ref{cond:M}. 
    We assume that the potential $\Phi\colon\,\posfold\to\Rnum$ 
    fulfils the conditions~\ref{cond:named:Langevin:P} above.
    Denote by~$\nu$ the zero-mean Gaussian measure with covariance matrix~$\beta^{-1}\ID$
    and define 
    $
        \mu
        \vdef
        \Leb{\wriemannmetric*}\locprod\nu  
    $.

    Then, the Langevin operator 
    \begin{displaymath}
        \parentheses*{
            L, 
            \smoothcompact{\tanbundle{\posfold}}
        }
        =
        \parentheses*{
            \frac\alpha\beta\,\laplace{\vertmetric*}
            -
            \alpha\canonfield
            +
            \spray{\riemannmetric*}
            -
            \frac1\beta\,
            \grad{\vertmetric*}\hlfunc{\Phi},
            \smoothcompact{\tanbundle{\posfold}}
        }
    \end{displaymath}
    is closable in $H\vdef\bigL2{\tanbundle\posfold}{\mu}$. 
    Moreover, its closure~\operator{L} generates a strongly continuous contraction semigroup $(T_t)_{t\in[0,\infty)}$.
    Finally, there are constants $\kappa_1,\kappa_2\in(0,\infty)$ 
    computable in terms of $\alpha$, $\beta$, $\Lambda$ and $c$
    such that for all $g\in H$ and times $t\in[0,\infty)$ holds 
    \begin{displaymath}
        \norm{ T_tg-\scalarprod{g}{1}{H} }{H} 
        \leq 
        \kappa_1 \euler^{-\kappa_2 t}\norm{ g-\scalarprod{g}{1}{H} }{H}
    .
    \end{displaymath}
\end{theorem}

Clearly, we are going to prove it by applying the~\ref{thm:hypocoercivity-thm-named}.

\subsection{Data conditions}
\label{subsec:Langevin:D}

\begin{definition}[model Hilbert space~\ref{cond:D1}]
\label{def:Langevin:D1}
    Consider the probability space 
    \begin{displaymath} 
        (E,\mathfrak{E},\mu)
        =
        \parentheses*{%
            \tanbundle{\posfold}, 
            \borel{\tanbundle{\posfold}},
            \Leb{\wsasakimetric*}
        }
    ,
    \end{displaymath} 
    where $\Leb{\wsasakimetric*} = \Leb{\wriemannmetric*} \locprod \nu$ 
    is the weighted Sasaki volume measure with 
    $\wriemannmetric*$ is weighted by 
    $\rho_{\posfold}\vdef \exp(-\Phi)= \exp(-\beta\Psi)$ with $\beta\in(0,\infty)$ such that 
    \Leb{\wriemannmetric*} is a probability measure on~$(\posfold,\borel{\posfold})$, 
    and $\nu = \normaldist{0}{\beta^{-1}\ID_\posdim}$ 
    is the zero-mean normal distribution on the fibre~$F=\Rnum^\posdim$ 
    with covariance matrix~$\beta^{-1}\ID_\posdim$.
    In other words, 
    \Leb{\wsasakimetric*} has the loc\=/density 
    $\exp(-\Phi) \locprod \prod_{j=1}^\posdim\,\varphi_{0,\beta^{-1}}\bincirc\proj_j$, 
    where~$\varphi_{0,\beta^{-1}}$ denotes the density of a one-dimensional normal distribution with variance~$\beta^{-1}$.
    The model Hilbert space is 
    $H \vdef \bigL2{E}{\mu} = \bigL2{\tanbundle{\posfold}}{\wsasakimetric*}$, 
    cf.~\hyperref[not:Lp-measure]{Notation~\ref{not:Lp-measure}}.
\end{definition}

\begin{lemma}[SAD-decomposition~\ref{cond:D3},~\ref{cond:D4},~\ref{cond:D6}]
\label{lem:Langevin:SAD}
    Let condition~\ref{cond:Langevin:P1} hold and let~$\Phi$ be weakly harmonic.
    Consider the SAD-decomposition $L=S-A$ on~$D$ with
    \begin{align*}
        Sf 
        &\vdef
        \frac\alpha\beta\,\laplace{\wvertmetric*}f
        =
        \frac\alpha\beta\,\laplace{\vertmetric*}f
        -
        \alpha\canonfield f
        \\
        \quad\text{and}\quad
        Af 
        &=
        -\spray{\wriemannmetric*}f
        \vdef
        -\spray{\riemannmetric*}f
        +
        \frac1\beta\,
        \parentheses*{\grad{\vertmetric*}\hlfunc{\Phi}}(f)
    \end{align*}
    for all $f\in D$.
    Then, the following assertions hold:
    \begin{enumerate}[label={(\roman*)}]
        \item 
        \label{itm:Langevin:SAD:S}
        $(S,D)$ is symmetric and negative semidefinite.
        \item 
        \label{itm:Langevin:SAD:A}
        $(A,D)$ is antisymmetric.
        \item 
        \label{itm:Langevin:SAD:L}
        For all $f\in D$ we have that 
        $Lf \in \bigL1{\tanbundle{\posfold}}{\mu}$ and 
        $\int_{\tanbundle{\posfold}} Lf\ \diffd\mu = 0$.
    \end{enumerate}
\end{lemma}
\begin{proof}
\hfill

    \begin{enumerate}[label={(\roman*)}]
        \item 
        Combining the form of~$\grad{\sasakimetric*}\rho$ 
        from~\hyperref[rem:grad-loc-density]{Remark~\ref{rem:grad-loc-density}}
        with the result on the weighted Laplace-Beltrami, 
        see~\hyperref[lem:weighted-Laplace]{Lemma~\ref{lem:weighted-Laplace}}, 
        we know that 
        we have to look at $\frac1{\rho_F}\grad{\vertmetric*}\rho_F$.
        Condensing notation a bit 
        we calculate that
        \begin{align*}
            \frac1{\rho_F}\grad{\vertmetric*}\rho_F(f)
            &=
            \grad{\mathrm{euc}}\parentheses*{%
                -\frac\beta2\,\abs{\ID_F}{\mathrm{euc}}^2
            }(f)
            =
            -\beta\dualpair{\ID_{\tanbundle\posfold}}{df_0}
            =
            -\beta\pdot\canonfield f
        \end{align*}
        holds for all $f=\hlfunc{f_0}\in\kappa^\ast\smoothcompact{\posfold}$. 
        Hence, it follows 
        \begin{displaymath} 
            \scalarprod{Sf}{g}{H}
            =
            -\int_{\tanbundle\posfold} 
                \vertmetric{\grad{\vertmetric*}f}{\grad{\vertmetric*}g}
            \ \diffd\mu
            \qquad\text{for all }f,g\in D_0
        ,
        \end{displaymath} 
        and therefore $(S,D_0)$ is symmetric and nonpositive definite. 
        Since~$S$ is well-defined on~$D$ and~$D_0$ is dense in~$D$ wrt.\ graph norm, 
        part~\ref{itm:Langevin:SAD:S} follows directly. 
        \item
        By~\ref{thm:named:Liouville-thm} 
        and similar reasoning as in~\hyperref[rem:adjoint-sec]{Remark~\ref{rem:adjoint-sec}}
        we know the adjoint of the Riemannian semispray wrt.~\bigL2{\tanbundle\posfold}{\wsasakimetric*}\=/scalar product:
        \begin{align*}
            \adjoint{\spray{\riemannmetric*}}
            &=
            -\spray{\riemannmetric*}
            - 0
            - \frac1\rho\, \spray{\riemannmetric*}\rho
            =
            -\spray{\riemannmetric*}
            + \frac1{\rho_F}\, \beta\rho_F \pdot \spray{\riemannmetric*}\vlfunc{\Psi}
            \\
            &=
            -\spray{\riemannmetric*}
            + \beta \pdot \spray{\riemannmetric*}\vlfunc{\Psi}
        .
        \end{align*}
        Furthermore, we compute the adjoint wrt.~\bigL2{\tanbundle\posfold}{\wsasakimetric*}\=/scalar product
        of $\grad{\vertmetric*}\hlfunc{\Psi} = \frac1\beta\,\grad{\vertmetric*}\hlfunc{\Phi}$
        using that~$\Psi$ is weakly harmonic
        and adapting~\hyperref[rem:grad-loc-density]{Remark~\ref{rem:grad-loc-density}} accordingly:
        \begin{align*}
            \adjoint*{\grad{\vertmetric*}\hlfunc{\Psi}}
            &=
            -
            \grad{\vertmetric*}\hlfunc{\Psi}
            - 
            \underbrace{
                \divergence{\sasakimetric*}\parentheses*{%
                    \grad{\vertmetric*}\hlfunc{\Psi}%
                    }%
                }_{=0}
            -
            \frac1\rho\, \parentheses*{ \grad{\vertmetric*}\hlfunc{\Psi} }(\rho)
            \\
            &=
            -\grad{\vertmetric*}\hlfunc{\Psi}
            -\frac1\rho\, 
            \vlfunc{\rho_{\posfold}} 
            \pdot
            \underbrace{%
                \parentheses*{ \grad{\vertmetric*}\hlfunc{\Psi} }(\rho_F)%
            }_{%
                =\,\vertmetric%
                {\grad{\vertmetric*}\hlfunc{\Psi}}%
                {\grad{\vertmetric*}\rho_F}%
            }
            \\
            &=
            -\grad{\vertmetric*}\hlfunc{\Psi}
            -\underbrace{%
                \vertmetric%
                {\grad{\vertmetric*}\hlfunc{\Psi}}%
                {-\beta\canonfield}%
            }_{%
                =\,\hormetric%
                {\grad{\hormetric*}\vlfunc{\Psi}}%
                {-\beta\spray{\riemannmetric*}}%
            }
            =
            -\grad{\vertmetric*}\hlfunc{\Psi}
            +\beta \pdot \spray{\riemannmetric*}\vlfunc{\Psi}
        .
        \end{align*}
        Hence, $(A,D)$ is antisymmetric. 
        Since $(A,D)$ is densely defined on~$H$, 
        it is closable.
        \item
        The integrability statement is clear: 
        Let $h\in D$ arbitrary, 
        then we conclude that
        $\int_{\tanbundle{\posfold}} Sh\,\diffd\mu = 0$
        by part~\ref{itm:Langevin:SAD:S}
        and 
        $\int_{\tanbundle{\posfold}} Ah\,\diffd\mu = 0$
        by part~\ref{itm:Langevin:SAD:A}.
        We finish the proof pointing out 
        that $(L,D)$ is densely defined on~$H$ and dissipative, 
        thus it is closable.
    \end{enumerate}
\end{proof}

\begin{notation}
    Due~\hyperref[lem:Langevin:SAD]{Lemma~\ref{lem:Langevin:SAD}}
    $(S,D)$, $(A,D)$ and $(L,D)$ are closable.
    The closures we denote by~\operator{S}, \operator{A} and~\operator{L} respectively.
\end{notation}

From the easiest example of Euclidean space 
we learn that one actually doesn't need a weakly harmonic potential. 
To see this, we make use of Poisson manifolds.

\begin{example}
    Let $\posfold\vdef\Rnum_x^\posdim$.
    Then, the inverse~$\begin{pmatrix} 0 & -\ID{\posdim} \\ \ID{\posdim} & 0\end{pmatrix}$ 
    of `the' symplectic matrix 
    yields an almost complex structure~$\mathbb{J}$ on 
    $\tanbundle{\Rnum_x^\posdim}\simeq\Rnum_x^\posdim\times\Rnum_v^\posdim$.
    It's the same as the one constructed 
    in~\cite[Paragraph~5]{Dombrowski} or later in~\cite{TachibanaOkumura}, 
    by our convention of listing the vertical component first and the horizontal one second.
    Moreover, it is \emph{compatible} with the Euclidean metric 
    and the canonical symplectic form~$\Omega$ on 
    $\tanbundle*{\Rnum_x^\posdim}\simeq\Rnum_x^\posdim\times\parentheses*{\Rnum^\posdim}^\ast$ 
    in the sense that
    \begin{displaymath}
        \scalarprod%
            {v}%
            {w}%
            {\mathrm{euc}}%
        =
        \Omega(v,\mathbb{J}w)
        \quad\text{and}\quad
        \Omega(v,w)
        =
        \scalarprod%
            {\mathbb{J}v}%
            {w}%
            {\mathrm{euc}}%
        \qquad\text{for all } v,w\in\Rnum^{2\posdim}
    , 
    \end{displaymath} 
    cf.~\cite[Exercise~2.2.1]{MarsdenRatiu}. 
    The symplectic form gives rise to a Poisson bracket~\Poissonbracket{\cdot}{\cdot} 
    such that the correpsonding Poisson tensor reads as
    \begin{displaymath}
        \Poissonbracket{f}{g}(x,v)
        =
        -
        \scalarprod%
            {\grad{\mathrm{euc}}f(x,v)}%
            {\mathbb{J}\grad{\mathrm{euc}}g(x,v)}%
            {\mathrm{euc}}%
        =
        \Omega_{(x,v)}(df,dg)
    \end{displaymath}
    for all $f,g\in\smoothfunc{\Rnum^{2\posdim}}$ and $(x,v)\in\Rnum_x^\posdim\times\Rnum_v^\posdim$.

    In this terminology, the proof of~\cite[Lemma~3.4 part~ii)]{HypocoercMFAT} relies on 
    linking the operator~$(A,D)$ to the antisymmetric bilinear form~$(\mathcal{A},D)$ 
    of integrating minus the Poisson bracket wrt.~$\mu$:
    $
        \mathcal{A}(f,g)
        \vdef
        \int_{\Rnum^{2\posdim}} 
            - \Poissonbracket{f}{g}
        \ \diffd\mu
    $
    for all $f,g\in D=\smoothcompact{\Rnum^{2\posdim}}$.
    Via integration by parts one can show that 
    $
        \scalarprod%
            {Af}%
            {g}%
            {\bigL2{\mu}}
        =
        \mathcal{A}(f,g)
    $ holds for all $f,g\in\smoothcompact{\Rnum^{2\posdim}}$.
    This can be done \emph{without the assumption} of a weakly harmonic potential, 
    since the vector field action can be represented in terms of the Hamiltonian vector fields.
    Indeed: 
    Denote by~\Hamilton{f} the \emph{Hamiltonian vector field} of~$f\in\smoothfunc{\Rnum^{2\posdim}}$, i.\,e.\ 
    it fulfils $\Hamilton{f}(g)=\Poissonbracket{f}{g}$
    for all $g\in\smoothfunc{\Rnum^{2\posdim}}$.
    We get the explicit formula 
    $
        -\Hamilton{f}
        =
        \scalarprod%
            {\grad{v}f}%
            {\grad{x}}%
            {\mathrm{euc}}%
        - 
        \scalarprod%
            {\grad{x}f}%
            {\grad{v}}%
            {\mathrm{euc}}%
    $
    and conclude that~\Hamilton{f} is solenoidal by Schwarz's theorem.
    Hence, the following equality is true:
    \begin{displaymath}
        \mathcal{A}(f,g)
        =
        \int_{\Rnum^{2\posdim}} 
            \Hamilton{f}(\rho)
            \pdot g
        \ \diffd\Leb
        =
        \scalarprod%
            {Af}%
            {g}%
            {\bigL2{\mu}}
        \qquad\text{for all }f,g\in D
    ,
    \end{displaymath}
    where $\rho\vdef\frac{d\mu}{d\Leb}$ is the (loc-)density of~$\mu$.
\end{example}

\begin{proposition}[SAD-decomposition (2nd version)]
\label{prop:Langevin:SAD}
    The assertions of~\hyperref[lem:Langevin:SAD]{Lemma~\ref{lem:Langevin:SAD}}
    are true without the assumption of~$\Psi$ being weakly harmonic.
\end{proposition}
\begin{proof}
    We enhance the proof of~\hyperref[lem:Langevin:SAD]{Lemma~\ref{lem:Langevin:SAD}} part~\ref{itm:Langevin:SAD:A} 
    via the technique discussed in the previous example
    which works due to the particular choices of 
    the Gaussian fibre measure and the Sasaki metric wrt.\ the most natural Ehresmann connection on~\tantanbundle\posfold.

    Let's abbreviate by 
    $
        \rho
        \vdef
        \frac{d\mu}{d\Leb{\sasakimetric*}} 
        = 
        \frac{d\Leb{\wsasakimetric*}}{d\Leb{\sasakimetric*}} 
    $ 
    the (loc-)density of~$\mu$.
    Denote by~$\mathbb{J}$ minus the almost complex structure on~\tanbundle\posfold 
    constructed in~\cite{Dombrowski},  
    and let~$\Omega$ be the canonical symplectic form on~\tanbundle*\posfold.
    By construction, the Sasaki metric, $\mathbb{J}$ and~$\Omega$ are compatible, 
    cf.~\cite[page~341]{MarsdenRatiu}. 
    Hence, they define the same Poisson bracket \Poissonbracket{\cdot}{\cdot} on~\tanbundle\posfold via the assignments
    \begin{displaymath}
        \Omega_{v}(df,dg)
        \fedv
        \Poissonbracket{f}{g}(v)
        \vdef
        -
        \sasakimetric%
            {\grad{\sasakimetric*}f(v)}%
            {\mathbb{J}\grad{\sasakimetric*}g(v)}%
            {v}%
    \end{displaymath}
    for all $f,g\in\smoothfunc{\tanbundle\posfold}$ and $v\in\tanbundle\posfold$.
    Thus, for any fixed $f\in\smoothfunc{\tanbundle\posfold}$ 
    there is a unique Hamilton vector field~\Hamilton{f} 
    by~\cite[Proposition~10.2.1]{MarsdenRatiu}.
    Let a chart $(v^j)_{j=1}^{2\posdim}$ 
    that gives normal coordinates and respects the Ehresmann connection 
    such that $(\partial v^i)_{i=1}^{\posdim}$ provides local basis for vertical vector fields 
    and $(\partial v^{k+\posdim})_{k=1}^{\posdim}$ provides a basis for the horizontal vector fields; 
    we find that in these coordinates the Hamiltonian vector fields attain the form 
    \begin{displaymath}
        -\Hamilton{f}
        =
        \sum_{k=1}^{\posdim}
            \partial_{v^k}f \pdot \partial_{v^{k+\posdim}}
        -
        \sum_{i=1}^{\posdim}
            \partial_{v^{i+\posdim}}f \pdot \partial_{v^i}
    .
    \end{displaymath}
    In any such coordinates we easily compute that 
    $\divergence{\sasakimetric*}(\Hamilton{f})=0$ using Schwarz's Theorem, 
    in other words all~\Hamilton{f} are solenoidal.

    Define the antisymmetric bilinear form $(\mathcal{A},D)$ by
    $
        \mathcal{A}(f,g)
        \vdef
        \int_{\tanbundle\posfold} 
            - \Poissonbracket{f}{g}
        \ \diffd\mu
    $
    for all $f,g\in D=\smoothcompact{\tanbundle\posfold}$.
    From the Divergence Theorem it follows that
    \begin{displaymath}
        \mathcal{A}(f,g)
        =
        \int_{\tanbundle\posfold} 
            -\Hamilton{f}(g)
        \ \diffd\Leb{\wsasakimetric*}
        =
        \int_{\tanbundle\posfold} 
            \Hamilton{f}(\rho)
            \pdot g
        \ \diffd\Leb{\sasakimetric*}
    .
    \end{displaymath}
    Using our comments on the Sasaki gradient of a loc\=/density, 
    see~\hyperref[rem:grad-loc-density]{Remark~\ref{rem:grad-loc-density}},
    we can infer that 
    $\Hamilton{f}(\rho) = -\rho\pdot\spray{\wriemannmetric*}f$. 
    In a nutshell, 
    we localise in the support of~$f$ 
    via a partition of unity argument 
    where the corresponding open cover is formed by 
    charts $(v^j)_{j=1}^{2\posdim}$ 
    that are respecting the Ehresmann connection
    and also are restricted to domains of local trivialisation; 
    therein, we can use the local coordinate form of~\Hamilton{f} 
    and that the loc\=/density~$\rho$ trivialises to a product of exponential-type densities.
    Hence, we gain that 
    $
        \scalarprod%
            {Af}%
            {g}%
            {H}
        =
        \mathcal{A}(f,g)
    $ for all $f,g\in D$
    which finishes the proof.
\end{proof}

We are going to construct the projection~$P$ mentioned in~\ref{cond:D5}.
For every $f\in H$ we call the mapping
\begin{displaymath}
    \expect[\nu]{f}\colon\ 
    \posfold\to\Rnum,\ 
    x \longmapsto \int_{\tanbundle{\posfold}{x}} f \ \diffd\nu
\end{displaymath}
the \emph{fibrewise average of~$f$}.
Clearly, the operator $\expect[\nu]$ acts trivially on vertically lifted functions, 
since they are fibrewise constant: 
$\expect[\nu]{\vlfunc{f_0}} = f_0$ 
for all $f_0\in\bigL2{\posfold}{\wriemannmetric*}$.
We define $P_Sf\vdef\expect[\nu]{f}\bincirc\tanproj$ assigning 
to an element of the tangent space the average of~$f$ over the fibre corresponding to this element.
Thinking in a local trivialisation 
the vertical lift of the fibrewise average erases the dependency of~$f$ on the velocity component, 
since~$f$ can be thought as a bivariate function of position and velocity in this localisation.
The range of~$P_S$ is precisely
$P_S(H) = \tanproj^\ast\parentheses*{ \bigL2{\posfold}{\wriemannmetric*} }$, 
the set of vertically lifted \bigL2\=/functions on the weighted position manifold, 
and~$P_S$ is a projection.

Moreover, we have that
\begin{align*}
    \int_{\tanbundle{\posfold}} (P_Sf)^2 \ \diffd\mu
    &=
    \int_{\tanbundle{\posfold}} 
        (\expect[\nu]{f}\bincirc\tanproj)^2
    \ \diffd\Leb{\wriemannmetric*}\locprod\nu
    =
    \int_{\tanbundle{\posfold}} 
        (\expect[\nu]{f})^2\bincirc\tanproj 
    \ \diffd\Leb{\wriemannmetric*}\locprod\nu
    \\
    &=
    \int_{\tanbundle{\posfold}} 
        (\expect[\nu]{f})^2\bincirc\tanproj 
    \ \diffd\tanproj^\ast\Leb{\wriemannmetric*}
    =
    \int_{\posfold} 
        \expect[\nu]{f}^2
    \ \diffd\Leb{\wriemannmetric*}
\end{align*}
for all $f\in H$. 
This implies 
$
    \norm{P_Sf}{H}
    =
    \norm{ \expect[\nu]{f} }{\bigL2{\posfold}{\wriemannmetric*}}
    \leq
    \norm{ f }{H}
$
meaning that~$P_S$ is continuous with norm~1.
Last but not least, $P_S$ also is an orthogonal projection, 
cf.~\cite[Proposition~3.3]{Conway}. 
Finally, we define $P \vdef P_S - \scalarprod{\cdot}{1}{H}$.

\begin{lemma}[Properties of projection~$P$ and semigroup conservativity~\ref{cond:D5}, \ref{cond:D7}]
\label{lem:Langevin:D5D7}
    Let condition~\ref{cond:Langevin:P1} hold.
    Then, we have
    $P(H)\subseteq\opdomain{S}$, $SP=0$, 
    $P(D)\subseteq\opdomain{A}$ and $AP(D)\subseteq\opdomain{A}$.
    Furthermore, $1\in\opdomain{L}$ and $L1 = 0$.
\end{lemma}
\begin{proof}
    We follow the lines of the proof of~\cite[Lemma~3.4~(iv)]{HypocoercMFAT} 
    for the Euclidean case.

    First of all, we start with the statements related to~$S$.
    Let $\varphi\in\smoothcompact{\posfold;[0,1]}$ be a cut-off function\footnote{%
        An explicit choice of~$\varphi$ is given as follows: 
        Let~$h$ be a chart at~$o$ -- 
        wlog.\ the ball~\ball{o}{2} is contained in the chart domain.
        Define the auxiliary `mountain' function 
        \begin{displaymath}
            m\colon\ \Rnum\to\Rnum,\ 
            t\longmapsto \exp\parentheses*{-\frac1{t(1-t)}}\pdot\indicator{(0,1)}(t)
        .
        \end{displaymath}
        We transform the `mountain' to a `table mountain' 
        by the assignment 
        \begin{displaymath}
            \tau\colon\
            \Rnum^\posdim\to[0,1],\,
            y\longmapsto\frac{\int_0^{2-\norm{y}{2}} m(t)\ \diffd t}{\int_0^1 m(t)\ \diffd t} 
        .
        \end{displaymath}
        Then, the choice 
        $\varphi\vdef\tau\bincirc h\pdot\indicator{\ball{o}{2}}$
        yields a smooth function with the desired properties.
    }
    such that $\varphi=1$ on~\ball{o}{1} 
    and $\varphi=0$ on~\ball{o}{2}.
    Define $\varphi_n\vdef\varphi(\sfrac{\ID}{n})$ for all $n\in\Nnum\setminus\{0\}$.
    Note that there is some constant $c\in(0,\infty)$ such that
    \begin{displaymath}
        \abs{\grad{\riemannmetric*}\varphi_n(x)}{\riemannmetric*} 
        \leq 
        \frac{c}n\,
        \quad\text{and}\quad
        \abs{\operatorname{Hess}_{\riemannmetric*}\varphi_n(x)}{\infty} 
        \leq 
        \frac{c}{n^2}\,
    \end{displaymath}
    for all $x\in\posfold$ and $n\in\Nnum\setminus\{0\}$,
    cf.~\cite[Definition~3.3]{HypocoercMFAT}.
    Clearly, $(\varphi_n)_{n\in\Nnum}$ pointwisely converges to the constant function~1 as $n\to\infty$.
    Let $f_0\in\smoothcompact{\posfold}$ 
    and define
    \begin{equation}
    \label{eq:Langevin:approxvertical1}
        f_n
        \vdef 
        \vlfunc{f_0}\tensorprod\hlfunc{\varphi_n}
        =
        \vlfunc{f_0}\pdot\hlfunc{\varphi_n}
        \qquad\text{for all }n\in\Nnum\setminus\{0\}
    . 
    \end{equation}
    This yields a sequence in~$D$ converging to~$\vlfunc{f_0}$ in~$H$.
    From dominated convergence we can conclude that
    \begin{displaymath}
        Sf_n
        =
        \frac\alpha\beta\,\vlfunc{f_0}
        \pdot
        \laplace{\vertmetric*}\hlfunc{\varphi_n}
        - 
        \alpha \vlfunc{f_0}\pdot\canonfield\hlfunc{\varphi_n}
        \longrightarrow
        0
        \qquad\text{in }H\text{ as } n\to\infty
    ,
    \end{displaymath}
    where we use 
    that the function 
    $
        \tanbundle\posfold\to\Rnum,\,
        v \mapsto
        \abs{\canonfield(v)}{\sasakimetric*}
        =
        \abs{v}{\riemannmetric*}
    $
    is in~$H$ 
    as on each fibre $\tanbundle{\posfold}{x}\simeq\Rnum^\posdim$
    the function 
    $
        \abs{\ID_{\Rnum^\posdim}}{\riemannmetric*}
    $
    is in~\bigL2{\Rnum^\posdim}{\nu}.
    Since \operator{S} is closed, 
    we have shown that $\vlfunc{f_0}\in\opdomain{S}$, and more specifically $S\vlfunc{f_0}=0$.
    Next, we prove that~$P$ maps into the null space of \operator{S}. 
    As the range of~$P$ is contained in~$\tanproj^\ast\bigL2{\posfold}{\wriemannmetric*}$, 
    we pick an $g_0\in\bigL2{\posfold}{\wriemannmetric*}$ 
    and show that $\vlfunc{g_0}\in\opdomain{S}$ and $S\vlfunc{g_0}=0$.
    The space~\smoothcompact{\posfold} is dense in~\bigL2{\posfold}{\wriemannmetric*}, 
    so there is a sequence $(g_n)_{n\in\Nnum\setminus\{0\}}$ in~\smoothcompact{\posfold} 
    approximating~$g_0$ in \bigL2{\posfold}{\wriemannmetric*}.
    We have seen just before 
    that $(\vlfunc{g_n})_{n\in\Nnum\setminus\{0\}}$ is a sequence in~\opdomain{S}
    and $S\vlfunc{g_n}=0$ for all $n\in\Nnum$.
    Again, with closedness of \operator{S} 
    it follows that $\vlfunc{g_0}\in\opdomain{S}$ and $S\vlfunc{g_0}=0$.

    Now, we turn to the statements involving~$A$.
    Let $f_0\in\smoothcompact{\posfold}$ 
    and define the sequence $(f_n)_{n\in\Nnum\setminus\{0\}}$ as in~\eqref{eq:Langevin:approxvertical1}
    approximating $\vlfunc{f_0}$ in~$H$.
    Using 
    $
        \abs{\spray{\riemannmetric*}}{\sasakimetric*}
        =
        \abs{\ID_{\tanbundle\posfold}}{\riemannmetric*}\in H
    $ 
    as above and that 
    $\grad{\riemannmetric*}\Phi\in\bigL2{\posfold\to\tanbundle\posfold}{\wriemannmetric*}$
    we can infer via dominated convergence that
    \begin{align*}
        Af_n
        &=
        -\hlfunc{\varphi_n}\pdot\spray{\riemannmetric*}\vlfunc{f_0}
        +
        \frac1\beta\vlfunc{f_0}
        \pdot
        \parentheses*{\grad{\vertmetric*}\hlfunc{\Phi}}\hlfunc{\varphi_n}
        =
        -\hlfunc{\varphi_n}\pdot\spray{\riemannmetric*}\vlfunc{f_0}
        +
        \frac1\beta\vlfunc{f_0}
        \pdot
        \riemannmetric{\grad{\riemannmetric*}\Phi}{\grad{\riemannmetric*}\varphi_n}
        \\
        &\longrightarrow
        -\spray{\riemannmetric*}\vlfunc{f_0}
        +
        \frac1\beta\vlfunc{f_0}
        \pdot
        0
        \qquad\text{in } H\text{ as }n\to\infty
    .
    \end{align*}
    Closedness of \operator{A} implies $\vlfunc{f_0}\in\opdomain{A}$ as well as 
    $
        A\vlfunc{f_0}
        =
        -\spray{\riemannmetric*}\vlfunc{f_0}
    $.
    For the inclusion $P(D)\subseteq\opdomain{A}$ 
    it's enough to show that $1\in\opdomain{A}$ and $A1=0$.
    Similar as before the function 
    $
        \tanbundle\posfold\to\Rnum,\,
        v \mapsto
        \abs{\spray{\riemannmetric*}(v)}{\sasakimetric*}
        =
        \abs{v}{\riemannmetric*}
    $
    is in~$H$, 
    and thus dominated convergence gives us 
    \begin{displaymath}
        A\vlfunc{\varphi_n}
        =
        -\spray{\riemannmetric*}\vlfunc{\varphi_n}
        \longrightarrow
        0
        \qquad\text{in } H\text{ as }n\to\infty
    .
    \end{displaymath}
    In order to prove now that $AP(D)\subseteq\opdomain{A}$ 
    we adhere to our approximation strategy and define 
    $
        h_n
        \vdef
        \hlfunc{\varphi_n}\pdot\spray{\riemannmetric*}\vlfunc{f_0}
    $ 
    for all $n\in\Nnum\setminus\{0\}$. 
    The sequence $(h_n)_{n\in\Nnum}$ converges to~$\spray{\riemannmetric*}\vlfunc{f_0}$ 
    both pointwisely and in~$H$.
    We note that the function
    $
        \spray{\riemannmetric*}^2\vlfunc{f_0}
        =
        \spray{\riemannmetric*}(\spray{\riemannmetric*}\vlfunc{f_0})
    $
    is dominated by 
    \begin{displaymath}
        \norm{\spray{\riemannmetric*}}{\bigL2{\tanbundle\posfold\to\tantanbundle\posfold}{\mu}}
        \pdot
        \norm{\spray{\riemannmetric*}\vlfunc{f_0}}{H}
        \leq
        \norm{\spray{\riemannmetric*}}{\bigL2{\tanbundle\posfold\to\tantanbundle\posfold}{\mu}}^2
        \pdot
        \norm{\vlfunc{f_0}}{H}
        =
        \norm{\abs{\ID_{\tanbundle\posfold}}{\riemannmetric*}}{H}^2
        \pdot
        \norm{\vlfunc{f_0}}{H}
    \end{displaymath}
    due to the Cauchy-Bunyakovsky-Schwarz inequality applied twice.
    This dominating function is in~$H$
    as on each fibre $\tanbundle{\posfold}{x}\simeq\Rnum^\posdim$
    the function 
    $
        \abs{\ID_{\Rnum^\posdim}}^2
    $
    is in~\bigL2{\Rnum^\posdim}{\nu}.
    Together with afore-mentioned facts that 
    $\abs{\ID_{\tanbundle\posfold}}{\riemannmetric*}\in H$ 
    and 
    $\grad{\riemannmetric*}\Phi\in\bigL2{\posfold\to\tanbundle\posfold}{\wriemannmetric*}$
    this yields 
    \begin{align*}
        Ah_n
        &=
        -\hlfunc{\varphi_n}\pdot\spray{\riemannmetric*}^2\vlfunc{f_0}
        +
        \frac1\beta
        \pdot
        \parentheses*{\grad{\vertmetric*}\hlfunc{\Phi}}h_n
        \\
        &=
        -\hlfunc{\varphi_n}\pdot\spray{\riemannmetric*}^2\vlfunc{f_0}
        +
        \frac1\beta
        \pdot
        \parentheses*{
            \parentheses*{\spray{\riemannmetric*}\vlfunc{f_0}}
            \pdot
            \parentheses*{\grad{\vertmetric*}\hlfunc{\Phi}}\hlfunc{\varphi_n}
            +
            \hlfunc{\varphi_n}
            \pdot
            \parentheses*{\grad{\vertmetric*}\hlfunc{\Phi}}(\spray{\riemannmetric*}\vlfunc{f_0})
        }
        \\
        &\longrightarrow
        -\spray{\riemannmetric*}^2\vlfunc{f_0}
        +
        0
        +
        \frac1\beta
        \pdot
        \parentheses*{\grad{\vertmetric*}\hlfunc{\Phi}}(\spray{\riemannmetric*}\vlfunc{f_0})
        \qquad\text{in } H\text{ as }n\to\infty
    \end{align*}
    by dominated convergence. 
    Since \operator{A} is closed,
    the function $\spray{\riemannmetric*}\vlfunc{f_0}$ is an element of~\opdomain{A}.
 
    Finally, the statements on~$L$ follow the very same way:
    Let $f_0\in\smoothcompact{\posfold}$ 
    and define the sequence $(f_n)_{n\in\Nnum\setminus\{0\}}$ as in~\eqref{eq:Langevin:approxvertical1}; 
    repeat the previous steps and conclude from closedness of~\operator{L} 
    that $f=\vlfunc{f_0}\in\opdomain{L}$ and $L\vlfunc{f_0} = - A\vlfunc{f_0}$. 
    In particular, the sequence 
    $\parentheses*{L\vlfunc{\varphi_n}}_{n\in\Nnum}$ 
    converges in~$H$ to~0 as $n\to\infty$,
    and by closedness it follows 
    $1\in\opdomain{L}$ with $L1=0$.
\end{proof}

\begin{remark}
\label{rem:Langevin:APf-Eq2}
    As we deduced $Pf\in\opdomain{A}$ for all $f\in D$, 
    we can calculate that
    \begin{align}
    \label{eq:Langevin:APf-Eq1}
    \begin{aligned}
        APf
        &=
        -\spray{\wriemannmetric*}\parentheses*{\vlfunc{\expect[\nu]{f}}}
        =
        -\dualpair{%
            \spray{\riemannmetric*}%
        }{%
            d\expect[\nu]{f} \bincirc d\tanproj
        }
        \\
        &=
        -\dualpair{ \ID_{\tanbundle\posfold} }{ d\expect[\nu]{f} }
        =
        -d\expect[\nu]{f}
    .
    \end{aligned}
    \end{align}
    However, one can give this formula a more intuitive form:
    \begin{align}
    \label{eq:Langevin:APf-Eq2}
    \begin{aligned}
        APf
        &=
        -\spray{\wriemannmetric*}\parentheses*{Pf}
        =
        -\spray{\riemannmetric*}\parentheses*{\vlfunc{\expect[\nu]{f}}}
        =
        -\hormetric{%
            \spray{\riemannmetric*}
        }{%
            \grad{\hormetric*}(P_Sf)
        }
        \\
        &=
        -\hormetric{%
            \spray{\riemannmetric*} %
        }{%
            \horlift\parentheses*{ \grad{\riemannmetric*}\expect[\nu]{f} } %
        }
        =
        -\riemannmetric{%
            \ID_{\tanbundle\posfold} %
        }{%
            \grad{\riemannmetric*}\expect[\nu]{f} \bincirc \tanproj %
        }{\tanproj}
    .
    \end{aligned}
    \end{align}
    Unlike \hyperref[eq:Langevin:APf-Eq1]{Equation~\eqref{eq:Langevin:APf-Eq1}}, 
    we recognise that~\hyperref[eq:Langevin:APf-Eq2]{Equation~\eqref{eq:Langevin:APf-Eq2}} 
    is in perfect correspondence to~\cite[Equation~(3.12)]{HypocoercMFAT}.
\end{remark}

Turning to the remaining condition~\ref{cond:D2}, i.\,e.\ 
the question whether~\operator{L} generates a strongly continuous semigroup, 
we first look into the case of a smooth potential $\Psi\in\smoothfunc{\posfold}$.  
The proof relies on methods from~\cite{HelfferNier} 
as explained in~\cite[Section~4]{HypocoercJFA}.
We briefly quote a consequence of the Hörmander~Theorem, 
namely~\cite[Proposition~A.1]{HelfferNier}.
Let~$T$ be a second order differential operator 
on a Riemannian manifold~$(\basefold,\basemetric*)$
of the form $T = c + \Xfield_0 + \sum_{k=1}^\ell \Xfield_k$
with $c\in\smoothfunc{\basefold}$ 
and $\Xfield_k\in\smoothsec{\tanbundle\basefold}$ for all $k\in\{0,\ldots,\ell\}$.
We say that~$T$ satisfies the \emph{Hörmander condition} 
if at any point $b\in\basefold$ holds
\begin{displaymath}
    \dim(\Liespan{\Xfield_0,\ldots,\Xfield_\ell}{b})
    =
    \dim(\tanbundle{\basefold}{b}) 
    =
    \dim(\basefold)     
    ,
\end{displaymath}
where \Liespan{\Xfield_0,\ldots,\Xfield_\ell} denotes the generated Lie~algebra.

\begin{proposition}
    Let~$T$ satisfy the Hörmander condition 
    and let $f\in\Lloc1{\basefold}{\basemetric*}$ such that
    \begin{displaymath}
        \int_{\basefold} f \pdot T\psi \ \diffd\Leb{\basemetric*}
        = 
        0
        \qquad\text{for all }
        \psi\in\smoothcompact{\basefold}
    .
    \end{displaymath}
    Then, $f$ has a smooth representative.
\end{proposition}
\begin{proof}
    The proof of this proposition works as for~\cite[Proposition~A.1]{HypocoercJFA}: 
    It is done in chart domains, 
    and within these domains it's perfectly fine to consider 
    $\Xfield_j = \partial x^j = \frac{\partial}{\partial x^j}$, 
    where $\parentheses*{x^j}_{j=1}^{\dim(\basefold)}$ denotes local coordinates provided by the chart.
    Then, we have smooth representatives in chart domains 
    serving as starting point for a partition of unity argument.
    Thus, even if~$T$ is just available in local coordinate form, 
    we can apply the previous proposition 
    as soon as the respective chart domains form an open cover.
\end{proof}

\begin{lemma}[Hörmander condition for the Langevin generator]
\label{lem:Langevin:generator-Hörmander}
    Consider a smooth potential $\Psi\in\smoothfunc{\posfold}$.
    Then, the Langevin generator satisfies the Hörmander condition.
    More precisely, let $v\in\tanbundle\posfold$ 
    and $\parentheses*{x^j}_{j=1}^{\posdim}$ be local coordinates corresponding to a chart at~$\tanproj(v)$.
    Then, we have that
    \begin{align*}
        \dim\parentheses*{\Liespan{%
            \spray{\riemannmetric*},
            \vertlift\parentheses*{\grad_{\riemannmetric*}\Psi}, %
            \vertlift\parentheses*{\partial x^1}, %
            \ldots, %
            \vertlift\parentheses*{\partial x^\posdim}, %
            \canonfield
        }{v}}
        &=
        \dim(\tantanbundle{\posfold}{v})     
        \\
        &=
        \dim(\tanbundle\posfold) 
        =
        2\posdim 
    .
    \end{align*}
\end{lemma}
\begin{proof}
    Obviously, neither $\vertlift\parentheses*{\grad_{\riemannmetric*}\Psi}$ nor~\canonfield
    contribute anything to the generated Lie algebra.
    From~\cite[Proposition~5.1]{GudmundssonKappos} we know explicit forms 
    for the Lie brackets of vertically and horizontally lifted vector fields in any combination. 
    With that said, we have the equations
    \begin{displaymath}
        \Liebracket{\vertlift(\Xfield)}{\vertlift(\Yfield)}
        =
        0
        \qquad\text{and}\qquad
        \Liebracket{\horlift(\Yfield)}{\vertlift(\Xfield)}
        =
        \vertlift\parentheses*{\connection{\Yfield}{\riemannmetric*}\Xfield}
    \end{displaymath}
    for all $\Xfield,\Yfield\in\smoothsec{\tanbundle\posfold}$. 
    Thus, the only nontrivial pairing of the seeding vector fields 
    $
        \Xfield,\Yfield
        \in
        \{%
            \partial x^1,%
            \ldots,%
            \partial x^\posdim%
        \}%
    $
    is
    $
        \Liebracket{\horlift(\Yfield)}{\vertlift(\Xfield)}
        =
        \vertlift\parentheses*{\connection{\Yfield}{\riemannmetric*}\Xfield}
    $, 
    but clearly all vector fields of this form are linear dependent of the vertical vector fields  
    $
        \left\{%
            \vertlift\parentheses*{\partial x^1},%
            \ldots,%
            \vertlift\parentheses*{\partial x^{\posdim}}%
        \right\}%
    $
    generating the Lie algebra.
    This just shows that the collection
    $
        \{%
            \partial x^1,%
            \ldots,%
            \partial x^{\posdim}%
        \}%
    $
    might not serve our purpose
    even if we would build our generator with both kinds of liftings. 

    The statement concerning Lie brackets of one vertically and one horizontally lifted vector fields 
    does not apply to~\spray{\riemannmetric*}, 
    as a semispray can not arise as horizontal lift of a vector field.
    However, in~\hyperref[lem:app:bracket-with-spray]{Lemma~\ref{lem:app:bracket-with-spray}}
    we prove that 
    $
        \Liebracket{%
            \spray{\riemannmetric*}
        }{%
            \vertlift\parentheses*{\partial x^j}
        }
        =
        \horlift\parentheses*{\partial x^j}
        -
        \sum_{i=1}^{\posdim} 
            N_j^i \pdot \vertlift\parentheses*{\partial x^i}
    $ 
    for certain functions~$N_j^i$.
    This yields \posdim~many linear independent horizontal vector fields in the generated Lie~algebra 
    which finishes the proof.
\end{proof}

\begin{theorem}[\ref{cond:D2}~for smooth potentials using a hypoellipticity startegy]
\label{thm:Langevin:D2-smooth-case}
    Let $\Psi\in\smoothfunc{\posfold}$ be a smooth potential. 
    Then, $(L,D)$~is essentially m-dissipative. 
    Thus, its closure~\operator{L} generates a strongly continuous semigroup.
\end{theorem}
\begin{proof}
    From~\hyperref[lem:Langevin:SAD]{Lemma~\ref{lem:Langevin:SAD}} we know that
    $(L,D)$ is dissipative.
    In view of the Lumer-Philips~Theorem, 
    we have to show that 
    the range $(\ID_H-L)(D)$ is dense in~$H$.
    Let $f\in H$ fixed such that
    \begin{equation}
    \label{eq:aux:hypoellip-D2}
        \scalarprod{(\ID_H-L)u}{f}{H}
        =
        0
        \qquad\text{for all }u\in D
    .
    \end{equation}
    We claim that $f=0$.

    Due to the choice of~$f$ we have that 
    $\exp(-\vlfunc{\Phi})f\in\Lloc1{\tanbundle\posfold}{\Leb{\riemannmetric*}\locprod\nu}$ 
    and by~\hyperref[lem:Langevin:generator-Hörmander]{Lemma~\ref{lem:Langevin:generator-Hörmander}}
    we can assume that $\exp(-\vlfunc{\Phi})f$ is smooth.
    Let $(\varphi_n)_{n\in\Nnum}$ be a sequence of cut-off functions in \smoothcompact{\posfold;[0,1]}
    with $\varphi_n\rightarrow1$ pointwise as $n\to\infty$ 
    and such that 
    $
        \norm{ \grad{\riemannmetric*}\varphi_n }{\bigL\infty{\posfold\to\tanbundle\posfold}{\wriemannmetric*}}^2 
        \leq 
        \frac1n\,C
    $ 
    holds for all $n\in\Nnum\setminus\{0\}$ and some $C\in(0,\infty)$.
    Now, define $u_n\vdef\parentheses*{\vlfunc{\varphi_n}}^2f$ for all $n\in\Nnum$.
    It's clear that
    \begin{displaymath}
        \scalarprod{u_n}{f}{H}
        \overset{\eqref{eq:aux:hypoellip-D2}}{=}
        \scalarprod{Lu_n}{f}{H}
        =
        \scalarprod{Su_n}{f}{H}
        -
        \scalarprod{Au_n}{f}{H}
    .
    \end{displaymath}
    Due to the fact that $(S,D)$ is nonpositive definite and multiplying with $\varphi$ commutes with the action of~$S$ 
    we see that
    $
        \scalarprod{Su_n}{f}{H}
        =
        \scalarprod{S(\vlfunc{\varphi_n}f)}{\vlfunc{\varphi_n}f}{H}
        \leq
        0
    $.
    Similarly, using antisymmetry we get that
    \begin{displaymath}
        \scalarprod{Au_n}{f}{H}
        =   
        \scalarprod{%
            A\parentheses*{\vlfunc{\varphi_n}}\,\vlfunc{\varphi_n}f%
        }{f}{H}
        +
        \scalarprod{A(\vlfunc{\varphi_n}f)}{\vlfunc{\varphi_n}f}{H}
        =
        \scalarprod{%
            A\parentheses*{\vlfunc{\varphi_n}}\,\vlfunc{\varphi_n}f%
        }{f}{H}
    .
    \end{displaymath}
    Altogether, this yields the estimate
    \begin{displaymath}
        \scalarprod{u_n}{f}{H}
        =
        \int_{\tanbundle\posfold} \parentheses*{\vlfunc{\varphi_n}}^2f^2 \ \diffd\mu
        \leq
        \frac1n\,C
        \int_{\tanbundle\posfold} \vlfunc{\varphi_n} f^2 \ \diffd\mu
        \leq
        \frac1n\,C
        \norm{ f }{H}
    \end{displaymath}
    using~\eqref{eq:Langevin:APf-Eq2}.
    By dominated convergence this implies that 
    $\norm{ f }{H}^2 \leq 0$, 
    thus $f=0$.
\end{proof}

For the case of a loc-Lipschitzian potential, 
we leave the base weight aside for a moment, 
but keep the fibre weight.
In other words, 
we could think of~\tanbundle\posfold endowed with bundle weight $1\locprod\rho_F$.
Let 
$\Psi\in\Lloc1{\posfold}{\riemannmetric*}$,
and  
$L_0 \vdef \frac\alpha\beta\,\laplace{\wvertmetric*} +\, \spray{\riemannmetric*}$
be defined on the very same set~$D_0$ as before -- 
it should be clear at this point that 
$L_0$ has to consist of the weighted diffusion and the \emph{non}-corrected semispray 
instead of~\spray{\wriemannmetric*}. 

\begin{lemma}
\label{lem:Langevin:L0D0-ess-m-dissipative}
    $(L_0,D_0)$ is essentially m-dissipative on~\bigL2{\tanbundle\posfold}{\Leb{\riemannmetric*}\locprod\nu}.
\end{lemma}
\begin{proof}
    From~\autoref{thm:Langevin:D2-smooth-case} applied to the smooth case of the zero potential,
    we know that
    $(L_0,D)$ is essentially m-dissipative 
    on~$\bigL2{\tanbundle\posfold}{\Leb{\riemannmetric*}\locprod\nu}$.
    We have to show that 
    $(L_0,D)$ is contained in the closure of $(L_0,D_0)$.

    For any $f\in D$ there is an approximating sequence $(f_n)_{n\in\Nnum}$ in~$D_0$
    wrt.\ usual locally convex topology implying uniform convergence of all derivatives on compacts
    cf.~\hyperref[lem:D0isdense]{Lemma~\ref{lem:D0isdense}}.
    Furthermore, there is a common compact set in~\tanbundle{\posfold}  large enough
    containing~$\supp(f)$ and~$\supp(f_n)$ for all $n\in\Nnum$.
    Therefore, we have that
    $\sup_{v\in\tanbundle\posfold} \abs{ L_0 f_n(v) - L_0 f(v) } \rightarrow0$ 
    as $n\to\infty$. 
    Thus, $f_n\rightarrow f$ and $L_0 f_n \rightarrow L_0 f$ 
    in~\bigL2{\tanbundle\posfold}{\Leb{\riemannmetric*}\locprod\nu} as $n\to\infty$.
\end{proof}

Denote by~$\diriboundsobolev1\infty{\posfold}$ 
the closure of~\smoothcompact{\posfold} wrt.\ \sobolev1\infty\=/norm.
Define 
\begin{displaymath}
    D_1
    \vdef
    \tanproj^\ast\diriboundsobolev1\infty{\posfold}
    \tensorprod
    \kappa^\ast\smoothcompact{\posfold}
.
\end{displaymath}
Considering the operator $(L,D_1)$ we realise that
\hyperref[lem:Langevin:SAD]{Lemma~\ref{lem:Langevin:SAD}} still does apply -- 
the proof has to be adapted just slightly.
Further on, we define the unitary isomorphism 
\begin{displaymath}
    U\colon\ 
    H=\bigL2{\tanbundle\posfold}{\Leb{\wriemannmetric*}\locprod\nu} 
    \to
    \bigL2{\tanbundle\posfold}{\Leb{\riemannmetric*}\locprod\nu},\ 
    f \longmapsto \vlfunc{\exp\parentheses*{-\sfrac\Phi2}}\pdot f
.
\end{displaymath}
Note that $U(D_1) = D_1$.
Thus, we define the operator $\widetilde{L}\vdef ULU^{-1}$ on~$D_1$. 
One directly observes that 
\begin{displaymath}
    \widetilde{L} 
    = 
    \widetilde{S} - \widetilde{A} 
    \vdef
    U\laplace{\wvertmetric*}U^{-1} - UAU^{-1}
    \qquad\text{holds on } D_1 
\end{displaymath}
with operators $(\widetilde{S},D_1)$ symmetric, negative semidefinite 
and $(\widetilde{A},D_1)$ antisymmetric 
both on~\bigL2{\tanbundle\posfold}{\Leb{\riemannmetric*}\locprod\nu}, 
as well as 
\begin{align*}
    \widetilde{A}\vlfunc{f_0}
    &=
    -U
    \parentheses*{%
        \vlfunc{f_0}
        \pdot
        \riemannmetric%
            {\ID_{\tanbundle\posfold}}%
            {\grad{\riemannmetric*}\exp(\sfrac\Phi2)\bincirc\tanproj}%
            {\tanproj}
        +
        \vlfunc{\exp(\sfrac\Phi2)}
        \pdot
        \riemannmetric%
            {\ID_{\tanbundle\posfold}}%
            {\grad{\riemannmetric*}f_0\bincirc\tanproj}%
            {\tanproj}
    }
    \\
    &=
    -\frac12\,\vlfunc{f_0}
    \pdot
    \riemannmetric%
        {\ID_{\tanbundle\posfold}}%
        {\grad{\riemannmetric*}\Phi\bincirc\tanproj}%
        {\tanproj}
    -
    \riemannmetric%
        {\ID_{\tanbundle\posfold}}%
        {\grad{\riemannmetric*}f_0\bincirc\tanproj}%
        {\tanproj}
    \\
    &=
    -\frac12\,\vlfunc{f_0}
    \pdot
    \spray{\riemannmetric*}\vlfunc{\Phi}
    -
    \spray{\riemannmetric*}\vlfunc{f_0}
    \qquad\text{for all } f_0\in\diriboundsobolev1\infty{\posfold}
.
\end{align*}
For~$f=\vlfunc{f_0}\tensorprod\hlfunc{g_0} \in D_1$ we get that
\begin{align*}
    \widetilde{A}f
    &=
    -\frac12\,\vlfunc{f_0}\hlfunc{g_0}
    \pdot
    \spray{\riemannmetric*}\vlfunc{\Phi}
    -
    \hlfunc{g_0} \pdot \spray{\riemannmetric*}\vlfunc{f_0}
    +
    \vlfunc{f_0}\pdot
    \frac1\beta\,
    \parentheses*{\grad{\vertmetric*}\hlfunc{\Phi}} \parentheses*{\hlfunc{g_0}}
    \\
    &=
    -\frac12\,f
    \pdot
    \spray{\riemannmetric*}\vlfunc{\Phi}
    -
    \spray{\wriemannmetric*}f
    +
    \frac1\beta\,
    \parentheses*{\grad{\vertmetric*}\hlfunc{\Phi}} (f)
.
\end{align*}

In straight analogy to~\cite[Section~4]{HypocoercJFA}, 
the proof of the next lemma, 
which deals with the globally Lipschitzian case,
is based on a perturbation theorem for essentially m-dissipative operators. 
We present it here for sake of completeness. 

\begin{theorem}[Kato perturbation of an essentially m-dissipative operator]
\label{thm:perturbation-thm}
    Let an essentially m-dissipative operator~$Z$ 
    and a dissipative operator~$T$ 
    have common domain in some given Hilbert space with norm 
    $\norm{ \cdot }\vdef\sqrt{\scalarprod{\cdot}{\cdot}}$.
    Assume that there are constants $c_1\in\Rnum$ and $c_2\in(0,\infty)$ such that 
    \begin{displaymath}
        \norm{ Tf }^2
        \leq
        c_1\scalarprod{Zf}{f}
        +
        c_2\norm{ f }^2
    \end{displaymath}
    holds for all~$f$ from the common domain.
    Then, the perturbation~$Z+T$ of~$Z$ by~$T$ defined on the common domain 
    is essentially m-dissipative.
\end{theorem}
\begin{proof}
    See~\cite[Corollary~3.8, Lemma~3.9 and Problem~3.10]{Davies}.
\end{proof}

\begin{lemma}[Essential m-dissipativity in case of globally Lipschitzian potentials]
\label{lem:Langevin:D2-global-Lipschitz}
    Assume that $\Psi$ is globally Lipschitzian.
    Then, $(\widetilde{L},D_1)$ is essentially m-dissipative 
    on~\bigL2{\tanbundle\posfold}{\Leb{\riemannmetric*}\locprod\nu}.
    Hence, $(L,D_1)$ is essentially m-dissipative 
    on the space~$H=\bigL2{\tanbundle\posfold}{\Leb{\wriemannmetric*}\locprod\nu}$.   
\end{lemma}
\begin{proof}
    Define $Z\vdef L_0$ on~$D_1$.
    Then, $(Z,D_1)$ is a dissipative extension of $(Z,D_0)$.
    Thus, $(Z,D_1)$ is essentially m-dissipative 
    on~\bigL2{\tanbundle\posfold}{\Leb{\riemannmetric*}\locprod\nu}   
    by~\hyperref[lem:Langevin:L0D0-ess-m-dissipative]{Lemma~\ref{lem:Langevin:L0D0-ess-m-dissipative}}.
    Define the perturbation 
    \begin{displaymath}
        Tf\vdef 
        -\frac1\beta\,\parentheses*{\grad{\vertmetric*}\hlfunc{\Phi}}(f)
        +
        \frac12\, f \pdot\,\spray{\riemannmetric*}\vlfunc{\Phi}
    \end{displaymath}
    for all $f=\vlfunc{f_0}\tensorprod\hlfunc{g_0}\in D_1$.
    Since by~\ref{thm:named:Liouville-thm} 
    $(\spray{\riemannmetric*},D_1)$ is antisymmetric
    and also $(\widetilde{A},D_1)$ is antisymmetric
    in~\bigL2{\tanbundle\posfold}{\Leb{\riemannmetric*}\locprod\nu},   
    $(T,D_1)$ is antisymmetric as well.
    Thus, $(T,D_1)$ is dissipative.

    Choose~$g$ such that $f=Ug$.
    Using the Cauchy-Bunyakovsky-Schwarz inequality we get that
    \begin{align*}
        &\int_{\tanbundle\posfold}
            \parentheses*{%
                \grad{\vertmetric*}\hlfunc{\Phi}(f)%
            }^2
        \ \diffd\Leb{\riemannmetric*}\locprod\nu
        \\
        &=
        \int_{\tanbundle\posfold}
            \abs{%
                \vertmetric%
                    {\grad{\vertmetric*}\hlfunc{\Phi}}%
                    {\grad{\vertmetric*}f}%
            }^2
        \ \diffd\Leb{\riemannmetric*}\locprod\nu
        =
        \int_{\tanbundle\posfold}
            \abs{%
                \vertmetric%
                    {\grad{\vertmetric*}\hlfunc{\Phi}}%
                    {\grad{\vertmetric*}g}%
            }^2
        \ \diffd\Leb{\wriemannmetric*}\locprod\nu
        \\
        &\leq
        \norm{%
            \grad{\vertmetric*}\hlfunc{\Phi}%
        }{\bigL2{\tanbundle\posfold\to\tantanbundle\posfold}{\Leb{\riemannmetric*}\locprod\nu}}^2
        \pdot
        \int_{\tanbundle\posfold}
            \abs{ \grad{\vertmetric*}g }{\vertmetric*}^2
        \ \diffd\Leb{\wriemannmetric*}\locprod\nu
        \\
        &=
        \norm{%
            \grad{\riemannmetric*}\Phi%
        }{\bigL2{\posfold\to\tanbundle\posfold}{\riemannmetric*}}^2
        \pdot
        \scalarprod%
            {-\laplace{\wvertmetric*}g}%
            {g}%
            {\bigL2{\tanbundle\posfold}{\Leb{\wriemannmetric*}\locprod\nu}}
        \\
        &=
        \norm{%
            \grad{\riemannmetric*}\Phi%
        }{\bigL2{\posfold\to\tanbundle\posfold}{\riemannmetric*}}^2
        \pdot
        \scalarprod%
            {-\widetilde{S}f}%
            {f}%
            {\bigL2{\tanbundle\posfold}{\Leb{\riemannmetric*}\locprod\nu}}
    .
    \end{align*}
    by the integration by parts formula~\eqref{eq:weighted-ibp}.
    Abbreviate 
    $
        C_{\Phi}
        \vdef
        \norm{ \grad{\riemannmetric*}\Phi }{\bigL\infty{\posfold\to\tanbundle\posfold}{\riemannmetric*}}^2%
    $.
    Then, we immediately conclude 
    \begin{align*}
        \norm{ Tf }{\bigL2{\Leb{\riemannmetric*}\locprod\nu}}^2
        &\leq
        \frac1{\beta^2}\,
        C_\Phi
        \pdot
        \scalarprod%
            {-\widetilde{S}f}%
            {f}%
            {\bigL2{\Leb{\riemannmetric*}\locprod\nu}}
        +
        \frac14\,C_\Phi
        \pdot
        \norm{ f }{\bigL2{\Leb{\riemannmetric*}\locprod\nu}}^2
        \\
        &=
        c_1
        \scalarprod%
            {-Zf}%
            {f}%
            {\bigL2{\Leb{\riemannmetric*}\locprod\nu}}
        +
        c_2
        \norm{ f }{\bigL2{\Leb{\riemannmetric*}\locprod\nu}}^2
    \end{align*}
    with
    $
        c_1
        \vdef
        \frac2{\alpha\beta}\, C_\Phi 
    $
        and
    $
        c_2
        \vdef
        \frac14\,C_\Phi
    $, since we know
    $
        \scalarprod%
            {\spray{\riemannmetric*}f}%
            {f}%
            {\bigL2{\Leb{\riemannmetric*}\locprod\nu}}
        =
        0
    $.
    Finally, the claim follows applying~\autoref{thm:perturbation-thm} to $(Z+T,D_1)$.
\end{proof}

\begin{corollary}[\ref{cond:D2} for globally Lipschitzian potentials]
\label{cor:Langevin:D2-global-Lipschitz}
    Assume that $\Psi$ is a globally Lipschitzian potential.
    Then, $(L,D)$ is essentially m-dissipative on~$H$. 
\end{corollary}
\begin{proof}
    Note that $(L,D)$ is a dissipative extension of $(L,D_0)$. 
    Thus, we show that $(L,D_1)$ is contained in the closure of $(L,D_0)$ 
    and then apply~\hyperref[lem:Langevin:D2-global-Lipschitz]{Lemma~\ref{lem:Langevin:D2-global-Lipschitz}}.

    Let $f = \vlfunc{f_0}\in\tanproj^\ast\diriboundsobolev1\infty{\posfold}$, 
    $g\in\kappa^\ast\smoothcompact{\posfold}$ and 
    a sequence $(f_n)_{n\in\Nnum\setminus\{0\}}$ in~\smoothcompact{\posfold} such that 
    its vertical lifting approximates~$f$ in \sobolev12\=/sense, i.\,e.\  
    \begin{enumerate}
        \item
        $f_n\longrightarrow f_0$ as $n\to\infty$
        in $\bigL2{\posfold}{\riemannmetric*}$\=/sense and
        \item
        $\frac{\partial f_n}{\partial x^j} \longrightarrow \frac{\partial f_0}{\partial x^j} $ as $n\to\infty$
        in $\bigL2{\posfold}{\riemannmetric*}$\=/sense
        for any chart $x=(x^j)_{j=1}^{\posdim}$.
    \end{enumerate}
    This convergence is maintained under passing to
    $\tanproj^\ast\bigL2{\posfold}{\wriemannmetric*}$, i.\,e.\ 
    weighting the manifold. 
    Finally, we conclude that 
    \begin{displaymath}
        L\parentheses*{\vlfunc{f_n}\tensorprod g}
        \longrightarrow
        L\parentheses*{\vlfunc{f_0}\tensorprod g}
        =
        L\parentheses*{f\tensorprod g}
        \qquad\text{in } H \text{ as } n\to\infty
    .
    \end{displaymath}
\end{proof}

The final prove of this section basically is the same as in~\cite[Theorem~4.7]{HypocoercJFA}
as we have taken geometric effects into account before.

\begin{theorem}[\ref{cond:D2} for locally Lipschitzian potentials]
\label{thm:Langevin:D2-loc-Lipschitz}
    Let $\Psi$ be a loc-Lipschitzian potential bounded from below.
    Then, $(L,D)$ is essentially m-dissipative on~$H$. 
\end{theorem}
\begin{proof}
    Wlog.\ we assume that $\Psi\geq0$.
    Let $\varepsilon\in(0,\infty)$ and fix some $g\in D\setminus\{0\}$.
    Choose $\varphi,\psi\in D$ such that 
    \begin{displaymath}
        \varphi\vert_{\supp(g)}
        =
        1
        \text{,}\quad
        \psi\vert_{\supp(\varphi)}
        =
        1
        \quad\text{and}\quad
        0
        \leq
        \varphi
        \leq
        \psi
        \leq
        1
    .
    \end{displaymath}
    Let $f\in D$ arbitrary.
    Throughout the proof, 
    we add to the generators and invariant measures a subscript to indicate the corresponding potential, 
    e.\,g.\ $\mu_0 = \Leb{\sasakimetric*}$ in case of the zero potential.
    By construction 
    and using dissipativity of $(L_{\psi\Psi},D)$ on~\bigL2{\tanbundle\posfold}{\mu_{\psi\Psi}} 
    we get that
    \begin{align*}
        &\norm{ %
            (\ID-L_{\Psi})(\varphi f) -g %
        }{\bigL2{\mu_\Psi}}
        \\
        &\leq
        \norm{ %
            \varphi\parentheses*{(\ID-L_{\psi\Psi})f -g} %
        }{\bigL2{\mu_{\psi\Psi}}}
        +
        \norm{ f }{\bigL2{\mu_{\psi\Psi}}}
        \pdot
        \norm{ %
            \grad{\riemannmetric*}\varphi %
        }{\bigL\infty{\Leb{\wriemannmetric*}}}
        \\
        &\leq
        \norm{ (\ID-L_{\psi\Psi})f -g }{\bigL2{\mu_{\psi\Psi}}}
        +
        \norm{ (\ID-L_{\psi\Psi})f }{\bigL2{\mu_{\psi\Psi}}}
        \pdot
        \norm{ \grad{\riemannmetric*}\varphi }{\bigL\infty{\Leb{\wriemannmetric*}}}
    .
    \end{align*}
    Now, we tighten the requirements on~$\varphi$ via additionally demanding that
    \begin{displaymath}
        \norm{ %
            \grad{\riemannmetric*}\varphi %
        }{\bigL\infty{\posfold\to\tanbundle\posfold}{\wriemannmetric*}}
        =
        \norm{ %
            \grad{\riemannmetric*}\varphi %
        }{\bigL\infty{\wriemannmetric*}}
        <
        \frac\varepsilon4 \pdot \norm{ g }{\bigL2{\mu_0}}^{-1}
        =
        \frac\varepsilon4 \pdot \norm{ g }{\bigL2{\sasakimetric*}}^{-1}
    .
    \end{displaymath}
    Due to~\hyperref[cor:Langevin:D2-global-Lipschitz]{Corollary~\ref{cor:Langevin:D2-global-Lipschitz}} 
    $(L_{\psi\Psi},D)$ is essentially m-dissipative on~\bigL2{\tanbundle\posfold}{\mu_{\psi\Psi}}, 
    hence as a consequence of the Lumer-Philips Theorem, 
    there is $f\in D$ such that simultaneously hold 
    \begin{align*}
        \norm{ %
            (\ID-L_{\psi\Psi})f -g %
        }{\bigL2{\mu_{\psi\Psi}}}
        \leq
        \frac\varepsilon2
        \quad\text{and}\quad
        \norm{ %
            (\ID-L_{\psi\Psi})f %
        }{\bigL2{\mu_{\psi\Psi}}}
        \leq
        2
        \norm{ %
            g %
        }{\bigL2{\mu_{\psi\Psi}}}
    .
    \end{align*}
    For such an~$f$ we end up with
    $
        \norm{ %
            (\ID-L_{\Psi})(\varphi f) -g %
        }{\bigL2{\mu_{\psi\Psi}}}
        <
        \varepsilon
    $.
    In conclusion, we proved that 
    $(\ID-L_\Psi)(D)$ is dense in $H=\bigL2{\tanbundle\posfold}{\mu_\Psi}$.
\end{proof}

\subsection{Hypocoercivity conditions}
\label{subsec:Langevin:H}

\begin{lemma}[algebraic relation~\ref{cond:H1}]
\label{lem:Langevin:H1}
    Let condition~\ref{cond:Langevin:P1} hold. 
    Then, we have $PAP\vert_D = 0$.
\end{lemma}
\begin{proof}
    Recall~\hyperref[eq:Langevin:APf-Eq2]{Equation~\eqref{eq:Langevin:APf-Eq2}} 
    from~\hyperref[rem:Langevin:APf-Eq2]{Remark~\ref{rem:Langevin:APf-Eq2}}. 
    Furthermore, we are going to apply the formula for Gaussian integrals from~\cite[Lemma~3.1]{HypocoercJFA}:
    Let $f\in D$ 
    and consider polar coordinates in the fibre at $x\in\posfold$
    using~\cite[Satz~14.8]{Forster3}
    which is an application of the transformation formula and Fubini.
    We get that
    \begin{align*}
        &\int_{\tanbundle{\posfold}{x}} APf \ \diffd\nu
        \\
        &=
        \int_{(0,\infty)}\int_{\sphere{\posdim-1}} 
            -\riemannmetric%
                {v(r,u)}%
                {\grad{\riemannmetric*}\expect[\nu]{f}(x)}%
                {x}
            \pdot
            r^{\posdim-1}\,
            \frac%
                {d \nu(v(r,u))}%
                {d \Leb\otimes\spheresurfmeasure}
        \quad \diffd\spheresurfmeasure(u)\,\diffd r
        \\
        &=
        \int_{(0,\infty)} 
            0 
            \pdot
            r^{\posdim-1}\,
            \frac%
                {d \nu(v(r,u_0))}%
                {d \Leb\otimes\spheresurfmeasure}
        \ \diffd r
        = 
        0
    ,
    \end{align*}
    where $u_0\in\sphere{\posdim-1}$ is arbitrary, 
    and further \spheresurfmeasure denotes the surface measure of the sphere~\sphere{\posdim-1} and 
    \begin{displaymath}
        v\colon\ 
        (0,\infty)\times\sphere{\posdim-1} 
        \to 
        v((0,\infty)\times\sphere{\posdim-1}) \subseteq \tanbundle{\posfold}{x},\ 
        (r,u) \longmapsto v(r,u) 
    \end{displaymath}
    is the diffeomorphism corresponding to 
    $
        (0,\infty)\times\sphere{\posdim-1} \to F=\Rnum^{\posdim},\, 
        (r,u) \mapsto ru
    $.
    We point out that 
    this argument just works 
    as~$\nu$ is invariant wrt.\ rotations.
    
    We have seen that $P_SAP$ is trivial an~$D$, 
    so we can use orthogonality of~$P_S$ to obtain
    \begin{displaymath}
        0
        =
        \scalarprod{P_SAPf}{1}{\bigL2{\posfold}{\wriemannmetric*}}
        =
        \scalarprod{APf}{1}{H}
        \qquad\text{for all }f\in D
    .
    \end{displaymath}
    Thus, $PAP\vert_D=0$.
\end{proof}

\begin{lemma}[microscopic hypocoercivity~\ref{cond:H2}]
\label{lem:Langevin:H2}
    Let condition~\ref{cond:Langevin:P1} hold. 

    Then, condition~\ref{cond:H2} is fulfilled with 
    $\Lambda_m = \alpha$.
\end{lemma}
\begin{proof}
    Let $f\in D$.
    Using the Poincar\'e inequality for Gaussian measures, see~\cite{Beckner89},
    we deduce that
    \begin{align*}
        \scalarprod{-Sf}{f}{H}
        &=
        \frac\alpha\beta\,
        \scalarprod%
            {\grad{\wvertmetric*}f}%
            {\grad{\wvertmetric*}f}%
            {\bigL2{\tanbundle\posfold\to\tantanbundle\posfold}{\wsasakimetric*}}
        =
        \frac\alpha\beta\,
        \norm{ %
            \grad{\vertmetric*}f %
        }{\bigL2{\tanbundle\posfold\to\tantanbundle\posfold}{\wsasakimetric*}}^2
        \\
        &\geq
        \alpha
        \norm{ %
            f %
            - %
            \scalarprod%
                {f}%
                {1}%
                {\bigL2{\tanbundle\posfold}{\wvertmetric*}} %
        }{\bigL2{\tanbundle\posfold}{\wsasakimetric*}}^2
        =
        \alpha
        \norm{ %
            (\ID_H-P_S)f %
        }{H}^2
    \end{align*}
    and the claim follows.
\end{proof}

The strategy for proving condition~\ref{cond:H3} relies on~\cite[Corollary~2.13]{HypocoercJFA}. 
Most importantly, we have to prove that 
$(\ID_H-PA^2P,D)$ is essentially m-dissipative.
To do so, we characterise $\ID_H-PA^2P$ on~$D$ 
starting from~\hyperref[eq:Langevin:APf-Eq2]{Equation~\eqref{eq:Langevin:APf-Eq2}}
and show that the range $(\ID_H-PA^2P)(D)$ is dense in~$H$.

Let $s\colon\,(-\delta,\delta)\to\tanbundle\posfold$ be a curve such that
$s(0)=v$ and $s^\prime(0)=\spray{\riemannmetric*}(v)$ 
for $v\in\tanbundle\posfold$ fixed and some small $\delta\in(0,\infty)$. 
Let $x\vdef\tanproj(v)$.
The following computation relies on 
$\tanproj\bincirc s$ being a geodesic of~$\spray{\riemannmetric*}$
and the characterisation of the directional derivative  
in terms of parallel transport `$\paralleltransport$' along~$s$ given by the Levi-Civita connection:
\begin{align*}
\begin{aligned}
    -\spray{\riemannmetric*}APf(v)
    &\overset{\eqref{eq:Langevin:APf-Eq2}}{=}
    -\spray{\riemannmetric*} 
    \parentheses*{ %
        -\riemannmetric%
            {\ID_{\tanbundle\posfold}}%
            {\grad{\riemannmetric*}\expect*[\nu]{f} \bincirc\tanproj}%
            {\tanproj}%
    }(v)
    \\
    &=
    \sasakimetric%
        {\lim_{t\to0}\frac1t\, \parentheses*{%
            \paralleltransport_{s(t)}^{s(t)}(%
                \ID_{\tanbundle\posfold}\bincirc s(t)%
            ) %
            - %
            \paralleltransport_{s(0)}^{s(t)}(%
                \ID_{\tanbundle\posfold}\bincirc s(0) %
            )%
        }}%
        {\grad{\hormetric*}P_S f(v)}%
        {v}
    \\
    &\qquad\qquad+
    \hormetric%
        {\spray{\riemannmetric*}(v)}%
        {\connection{\spray{\riemannmetric*}}{\sasakimetric*}%
            \parentheses*{\grad{\hormetric*}P_Sf}(v)%
        }%
        {v}%
    \\
    &=
    \hormetric%
        {\spray{\riemannmetric*}(v)}%
        {\expect*[\nu]{\grad{\hormetric*}f}(v)}%
        {v}
    +
    \riemannmetric%
        {v}%
        {\connection{v}{\riemannmetric*}%
            \parentheses*{\grad{\riemannmetric*}\expect[\nu]{f}}(x)
        }%
        {x}
    \\
    &=
    \riemannmetric%
        {v}%
        {\grad{\riemannmetric*}\expect*[\nu]{f}(x)}%
        {x}
    +
    \riemannmetric%
        {v}%
        {\connection{v}{\riemannmetric*}%
            \parentheses*{\grad{\riemannmetric*}\expect[\nu]{f}}(x)
        }%
        {x}
.
\end{aligned}
\end{align*}
For the second equality we also used the metric compatibility of the Levi-Civita connection~\connection{}{\sasakimetric*}.
The last but one line is obtained 
using~\hyperref[lem:app:covaderivative-with-spray]{Lemma~\ref{lem:app:covaderivative-with-spray}}
for the second summand.
Now, we transform into polar coordinates in the fibre~$\tanbundle{\posfold}{x}$  
similar as in the proof of~\hyperref[lem:Langevin:H1]{Lemma~\ref{lem:Langevin:H1}}. 
With this ansatz we calculate
applying~\cite[Lemma~3.1]{HypocoercJFA} twice that 
\begin{align}
\label{eq:Langevin:A2P-part1}
\begin{aligned}
    &\int_{\tanbundle{\posfold}{x}}
        -\spray{\riemannmetric*} 
        \parentheses*{ %
            -\riemannmetric%
                {\ID_{\tanbundle\posfold}}%
                {\grad{\riemannmetric*}\expect*[\nu]{f} \bincirc\tanproj}%
                {\tanproj}%
        }
    \ \diffd\nu
    \\
    ={}&
    \int_{\tanbundle{\posfold}{x}}
        \riemannmetric%
            {v}%
            {\grad{\riemannmetric*}\expect*[\nu]{f}(x)}%
            {x}
    \ \diffd\nu(v)
    +
    \int_{\tanbundle{\posfold}{x}}
        \riemannmetric%
            {v}%
            {\connection{v}{\riemannmetric*}%
                \parentheses*{\grad{\riemannmetric*}\expect[\nu]{f}}(x)
            }%
            {x}
    \ \diffd\nu(v)
    \\
    ={}&
    \int_{(0,\infty)}\int_{\sphere{\posdim-1}}
        \riemannmetric%
            {v(r,u)}%
            {\connection{v(r,u)}{\riemannmetric*}%
                \parentheses*{\grad{\riemannmetric*}\expect[\nu]{f}}(x)
            }%
            {x}
    \\
    &\hphantom{\int_{(0,\infty)}\int_{\sphere{\posdim-1}}}\qquad\qquad
        \pdot
        r^{\posdim-1}\,
        \frac{d\nu(v(r,u))}{d\Leb\otimes\spheresurfmeasure}
    \quad \diffd\spheresurfmeasure(u)\,\diffd r
    \\
    ={}&
    \frac1{\posdim}
    \laplace{\riemannmetric*}\expect[\nu]{f}(x)
    \quad\pdot
    \int_{(0,\infty)}
        r^2\pdot
        \parentheses*{%
            \frac1{\spheresurfmeasure(\sphere{\posdim-1})}\,
            \int_{\sphere{\posdim-1}} 1\ \diffd\spheresurfmeasure
        }
        \pdot
        r^{\posdim-1}\,
        \frac{d\nu(v(r,u_0))}{d\Leb\otimes\spheresurfmeasure}
    \ \diffd r
    \\
    ={}&
    \frac1{\beta}
    \laplace{\riemannmetric*}\expect[\nu]{f}(x)
    \quad\pdot1
,
\end{aligned}
\end{align}
where we have taken $u_0\in\sphere{\posdim-1}$ arbitrary, 
since~$\nu$ is invariant wrt.\ rotations.
In order to arrive at the last but one line, 
we consider some chart $\parentheses*{x^j}_{j=1}^{\posdim}$ at~$x\in\posfold$ providing normal coordinates; 
in such coordinates the Levi-Civita connection is understood in terms of directional derivatives as
\begin{displaymath}
    \connection{\Yfield}\Xfield(x)
    =
    \sum_{i,j\in\{1,\ldots,\posdim\}} 
        \Yfield^i(x)\, \frac{\partial\Xfield^j(x)}{\partial x^i}\, x^j
\end{displaymath}
for all $\Xfield,\Yfield\in\smoothsec{\tanbundle\posfold}$ with local coordinate expression
$\Xfield = \sum_{j=1}^{\posdim} \Xfield^j\pdot\partial x^j$
and
$\Yfield = \sum_{j=1}^{\posdim} \Yfield^j\pdot\partial x^j$.
Thus, we can understand the mapping 
\begin{displaymath}
    \sphere{\posdim-1}\to\sphere{\posdim-1},\ 
    u \longmapsto \parentheses*{%
        \connection{u}{\riemannmetric*}\grad{\riemannmetric*}\expect[\nu]{f}%
    }(x)
\end{displaymath}
as the matrix in~\cite[Lemma~3.1]{HypocoercJFA}. 
The last step of~\hyperref[eq:Langevin:A2P-part1]{Equation~\eqref{eq:Langevin:A2P-part1}} is due to the fact that 
the mean of a chi-squared distribution equals the number of degrees of freedom, i.\,e.~\posdim in the present case. 

With the proof of part~\ref{itm:Langevin:SAD:A} of~\hyperref[lem:Langevin:SAD]{Lemma~\ref{lem:Langevin:SAD}}
and with~\cite[Lemma~3.1]{HypocoercJFA}
we similarly get that 
for every $v\in\tanbundle\posfold$ with $x\vdef\tanproj(v)$ holds
\begin{align}
\label{eq:Langevin:A2P-part2}
\begin{aligned}
    &
    P_S\parentheses*{%
        \frac1\beta\,\grad{\vertmetric*}\hlfunc{\Phi}(APf)%
        }
    (v)
    =
    \int_{\tanbundle{\posfold}{x}} 
        \spray{\riemannmetric*}\vlfunc{\Phi}
        \pdot
        APf
    \ \diffd\nu
    \\
    \overset{\eqref{eq:Langevin:APf-Eq2}}{=}{}&
    -\int_{\tanbundle{\posfold}{x}} 
        \hormetric%
            {\spray{\riemannmetric*}}%
            {\grad{\hormetric*}\vlfunc{\Phi}}
        \pdot
        \hormetric%
            {\spray{\riemannmetric*}}%
            {\grad{\hormetric*}(P_Sf)}
    \ \diffd\nu
    \\
    ={}&
    -\int_{\tanbundle{\posfold}{x}} 
        \riemannmetric%
            {\ID_{\tanbundle\posfold}}%
            {\grad{\riemannmetric*}\Phi(x)}%
            {x}
        \pdot
        \riemannmetric%
            {\ID_{\tanbundle\posfold}}%
            {\grad{\riemannmetric*}\expect[\nu]{f}(x)}%
            {x}
    \ \diffd\nu
    \\
    ={}&
    -\frac1{\beta}\,
    \riemannmetric%
        {\grad{\riemannmetric*}\Phi(x)}%
        {\grad{\riemannmetric*}\expect[\nu]{f}(x)}%
        {x}
    =
    -\frac1{\beta}\,
    \connection{\grad{\riemannmetric*}\Phi}(\expect[\nu]{f})(x)
.
\end{aligned}
\end{align}

Hence, we proved combining~\hyperref[eq:Langevin:A2P-part1]{Equation~\eqref{eq:Langevin:A2P-part1}} 
and~\hyperref[eq:Langevin:A2P-part2]{Equation~\eqref{eq:Langevin:A2P-part2}} that
\begin{align}
\label{eq:Langevin:PA2P}
\begin{aligned}
    PA^2Pf
    &=
    P_SA^2Pf
    =
    \frac1{\beta}
    \pdot
    \parentheses*{%
        \laplace{\riemannmetric*}\expect[\nu]{f}\bincirc\tanproj%
        - %
        \connection{\grad{\riemannmetric*}\Psi}\expect[\nu]{f}\bincirc\tanproj%
    }
    \\
    &=
    \frac1{\beta}
    \pdot
    \laplace{\whormetric*}\parentheses*{\vlfunc{\expect[\nu]{f}}}
    =
    \frac1{\beta}
    \pdot
    \laplace{\whormetric*}(P_Sf)
\end{aligned}
\end{align}
for all $f\in D$.
Compare our result to~\cite[Equation~(3.16)]{HypocoercMFAT}.
These preparations give shape to the following corollary, 
cf.~\cite[Proposition~3.9]{HypocoercMFAT}. 

\begin{corollary}[$PA^2P$ is essentially m-dissipative]
\label{cor:Langevin:PA2P-m-dissipative}
    Let condition~\ref{cond:Langevin:P1} hold. 
    Then, the range $(\ID_H-PA^2P)(D)$ is dense in~$H$, 
    thus $PA^2P$ is essentially m-dissipative on~$D$.
\end{corollary}
\begin{proof}
    Right away, we know that 
    $
        (PA^2P,D)
        \overset{\eqref{eq:Langevin:PA2P}}{=}
        \parentheses*{ \sfrac1\beta\laplace{\whormetric*}\bincirc P_S,D) }
    $ 
    is essentially m-dissipative on
    $
        P_S(H)
        =
        \tanproj^\ast\bigL2{\posfold}{\wriemannmetric*}
    $, 
    as $P_S(D) = \tanproj^\ast\smoothcompact{\posfold}$ 
    and~$(\laplace{\whormetric*},D)$ 
    is essentially self-adjoint in~$H$. 
    The later is true, 
    since~$(\laplace{\whormetric*},D_0)$ is essentially self-adjoint in~$H$ --
    as so is the Laplace-Beltrami on~$(\tanbundle\posfold,\whormetric*)$, 
    cf.\ the beginning of~\autoref{sec:preliminaries} --
    and the fact that~$D_0$ is dense in~$D$.

    Let $g\in H$ such that 
    \begin{displaymath}
        \scalarprod%
            {(\ID_H-PA^2P)f}%
            {g}%
            {H}
        =
        0
    \end{displaymath}
    for all $f\in D$
    and we claim that $g=0$.
    Our assumption immediately implies that 
    \begin{displaymath}
        0
        =
        \scalarprod%
            {(\ID_H-PA^2P)\vlfunc{f_0}}%
            {P_Sg}%
            {H}
        =
        \scalarprod%
            {\vlfunc{f_0} - \frac1\beta\,\laplace{\whormetric*}\vlfunc{f_0}}%
            {P_Sg}%
            {\tanproj^\ast\bigL2{\posfold}{\wriemannmetric*}}
    \end{displaymath}
    for all $f_0\in\smoothcompact{\posfold}$.
    Thus, $P_Sg=0$, 
    since the range 
    $
        \parentheses*{%
            \ID_H-\sfrac1{\beta}\,\laplace{\whormetric*}%
        }(\tanproj^\ast\smoothcompact{\posfold})
    $ 
    is dense in 
    $
        \tanproj^\ast\bigL2{\posfold}{\wriemannmetric*}
    $.
    Ultimately, this means that
    \begin{displaymath}
        \scalarprod{f}{g}{H} 
        =
        \scalarprod{PA^2Pf}{g}{H} 
        =
        \scalarprod%
            {\sfrac1\beta\laplace{\whormetric*}(P_Sf)}%
            {P_Sg}%
            {\tanproj^\ast\bigL2{\posfold}{\wriemannmetric*}} 
        =
        0
        \qquad\text{for all }f\in D
    ,
    \end{displaymath}
    which implies 
    $\scalarprod{f}{g}{H} = 0$
    for all $f\in D$, 
    hence $g=0$ as claimed.
\end{proof}

\begin{proposition}[macroscopic hypocoercivity~\ref{cond:H3}]
    Let the conditions~\ref{cond:Langevin:P1} and~\ref{cond:Langevin:P2} hold. 
    Then, condition~\ref{cond:H3} is fulfilled 
    with $\Lambda_M = \frac1\beta\,\Lambda$.
\end{proposition}
\begin{proof}
    Let $f\in D$.
    Since  
    $
        \parentheses*{PA^2P,D}
        =
        \parentheses*{\sfrac1\beta\,\laplace{\whormetric*}P_S,D}
    $
    pregenerates a variant of the weighted horizontal gradient form in the sense that 
    \begin{displaymath}
        \scalarprod{PA^2Pf}{g}{H}
        =
        -\frac1\beta
        \int_{\tanbundle\posfold}
            \hormetric%
                {\grad{\hormetric*}P_Sf}%
                {\grad{\hormetric*}P_Sg}%
        \ \diffd\mu
        \qquad\text{for all }f,g\in D
    ,
    \end{displaymath}
    we easily compute that
    \begin{align*}
        \norm{ APf }{H}^2
        &=
        \frac1\beta
        \int_{\tanbundle\posfold}
            \abs{ %
                \grad{\hormetric*}P_Sf(v) %
            }{\hormetric*}^2
        \quad \diffd\mu(v)
        =
        \frac1\beta
        \int_{\posfold}
            \abs{ %
                \grad{\riemannmetric*}\expect[\nu]{f}(x) %
            }{\riemannmetric*}^2
        \quad \diffd\Leb{\wriemannmetric*}(x)
        \\
        &=
        \frac1\beta\,
        \norm{ %
            \grad{\riemannmetric*}\expect[\nu]{f} %
        }{\bigL2{\posfold\to\tanbundle\posfold}{\wriemannmetric*}}^2
        \geq
        \frac1\beta\,\Lambda
        \norm{ %
            \expect[\nu]{f} %
            - %
            \scalarprod{\expect[\nu]{f}}{1}%
                {\bigL2{\posfold}{\wriemannmetric*}} %
        }{\bigL2{\posfold}{\wriemannmetric*}}^2
    \end{align*}
    by Poincar\'e inequality.
    Combining this estimates with the previous corollary, 
    then \cite[Corollary~2.13]{HypocoercJFA} finishes the proof.
\end{proof}

The remaining hypocoercivity condition~\ref{cond:H4} is checked 
via a standard procedure relying on~\cite[Lemma~2.14]{HypocoercJFA} and~\cite[Proposition~2.15]{HypocoercJFA}
cf.\ also~\cite[Proposition~3.11]{HypocoercMFAT}.

\begin{lemma}[boundedness of~$(BS,D)$, first part of~\ref{cond:H4}]
\label{lem:Langevin:H4-1}
    Let condition~\ref{cond:Langevin:P1} hold. 
    Then, with $c_1\vdef\frac12\,\alpha$ it holds that
    \begin{displaymath}
        \norm{ BSf }{H} 
        \leq 
        c_1 \norm{ (\ID_H-P_j)f }{H}
        \qquad\text{for all }f\in D
    \end{displaymath}
    and $P_j\in\{P,P_S\}$, $j\in\{1,2\}$.
\end{lemma}
\begin{proof}
    First, we show that 
    $
        SAP
        =
        \alpha
        AP
    $ 
    on~$D$. 
    For the first time, it will become important here that
    \spray{\riemannmetric*} is not just a semispray, 
    but actually a \emph{spray}, i.\,e.\ 
    additionally we have that 
    $
        \Liebracket{\canonfield}{\spray{\riemannmetric*}}
        =
        \spray{\riemannmetric*}
    $.
    This is due to the fact that 
    \spray{\riemannmetric*}~was chosen in correspondence to the Levi-Civita connection 
    which is an affine connection. 
    See~\cite{AmbrosePalaisSinger}.
    Let $f\in D$ be fixed. 
    Then, we immediately get that 
    \begin{displaymath}
        SAPf
        =
        SAP_Sf
        =
        -\alpha\canonfield(-\spray{\wriemannmetric*}(P_Sf))
        =
        \alpha
        \canonfield\parentheses*{ %
            \spray{\riemannmetric*}\vlfunc*{\expect[\nu]{f}}%
        }
    ,
    \end{displaymath}
    since we know from the proof of~\hyperref[lem:Langevin:SAD]{Lemma~\ref{lem:Langevin:SAD}}
    that $\scalarprod{Ah}{1}{H} = 0$ for all $h\in D$. 
    Using the Koszul formula we calculate that 
    for all $\Xfield\in\smoothsec{\tanbundle\posfold}$ holds
    \begin{displaymath}
        \canonfield 
        \hormetric%
            {\spray{\riemannmetric*}}%
            {\horlift\Xfield}%
        =
        0
        -
        \hormetric%
            {\Liebracket%
                {\canonfield}%
                {\spray{\riemannmetric*}}%
            }%
            {\horlift\Xfield}%
        +
        \hormetric%
            {\Liebracket%
                {\canonfield}%
                {\horlift\Xfield}%
            }%
            {\spray{\riemannmetric*}}%
    .
    \end{displaymath}
    Similar to~\cite[Proposition~5.1]{GudmundssonKappos} 
    mentioned before in~\hyperref[lem:Langevin:generator-Hörmander]{Lemma~\ref{lem:Langevin:generator-Hörmander}}, 
    one could use local coordinates for~\canonfield 
    in order to show that
    \Liebracket%
        {\canonfield}%
        {\horlift\Xfield}
    is purely vertical. 
    As~\spray{\riemannmetric*} even is a spray, 
    we gain that
    \begin{align*}
        SAPf 
        &=
        \alpha
        \canonfield 
        \hormetric%
            {\spray{\riemannmetric*}}%
            {\grad{\hormetric*}\vlfunc*{\expect[\nu]{f}}}%
        =
        -\alpha
        \hormetric%
            {\spray{\riemannmetric*}}%
            {\grad{\hormetric*}\vlfunc*{\expect[\nu]{f}}}%
        \\
        &=
        -\alpha
        \spray{\riemannmetric*} \vlfunc*{\expect[\nu]{f}}
        =
        \alpha(
        -\spray{\wriemannmetric*} \vlfunc*{\expect[\nu]{f}}
        )
        =
        \alpha APf
    \end{align*}
    Setting $c_1\vdef\frac\alpha2$
    the claim follows with~\cite[Lemma~2.14]{HypocoercJFA}.
\end{proof}

\begin{lemma}[boundedness of~$(BA(\ID_H-P),D)$, second part of~\ref{cond:H4})]
\label{lem:Langevin:H4-2}
    Let the potential conditions~\ref{cond:named:Langevin:P} hold. 
    Then, there exists a constant $c_2\in(0,\infty)$ such that
    \begin{displaymath}
        \norm{ BA(\ID_H-P)f }{H} 
        \leq 
        c_2 \norm{ (\ID_H-P_j)f }{H}
        \qquad\text{for all }f\in D
    \end{displaymath}
    and $P_j\in\{P,P_S\}$, $j\in\{1,2\}$.
\end{lemma}
\begin{proof}
    Let $f\in D$ and $g\vdef(\ID_H-PA^2P)f$.
    We know that 
    $g\in D(\adjoint{(BA)})$ 
    with $\adjoint{(BA)}g = - A^2Pf$, 
    cf.~\cite[Proposition~2.15]{HypocoercJFA}.
    Using our knowledge from the proof of part~\ref{itm:Langevin:SAD:A} of~\hyperref[lem:Langevin:SAD]{Lemma~\ref{lem:Langevin:SAD}}, 
    furthermore~\hyperref[eq:Langevin:A2P-part1]{Equation~\eqref{eq:Langevin:APf-Eq2}} 
    and the Cauchy-Bunyakovsky-Schwarz inequality
    we estimate that
    \begin{align*}
        \norm{ \adjoint{(BA)}g }{H}
        &\leq
        \norm{%
            v\longmapsto %
            \spray{\riemannmetric*}^2\parentheses*{\vlfunc{\expect[\nu]{f}}}(v) %
        }{H}
        +
        \norm*{%
            v\longmapsto \frac1\beta\,\grad{\vertmetric*}\hlfunc{\Phi}(APf)(v) %
        }{H}
        \\
        &\leq
        \norm{%
            \abs{ \spray{\riemannmetric*}^2\parentheses*{\vlfunc{\expect[\nu]{f}}} } %
        }{H}
        +
        \norm{%
            \abs{ \spray{\riemannmetric*}\vlfunc{\Phi} }{\hormetric*}%
            \pdot %
            \abs{ \spray{\riemannmetric*}(\vlfunc{\expect[\nu]{f}}) }{\hormetric*}
        }{H}
        \\
        &\leq
        \norm{ v\mapsto\abs{v}{\riemannmetric*}^2 }{H}
        \pdot
        \parentheses*{ %
            \norm{%
                \abs{ %
                    \operatorname{Hess}_{\sasakimetric*}\parentheses*{\vlfunc{\expect[\nu]{f}}} %
                } %
            }{H} %
            + %
            \norm{%
                \abs{ \grad{\hormetric*}\vlfunc{\Phi} }{\hormetric*}%
                \pdot %
                \abs{ \grad{\hormetric*}\vlfunc{\expect[\nu]{f}} }{\hormetric*}%
            }{H} %
        }
        \\
        &=
        \frac\posdim\beta
        \parentheses*{ %
            \norm{%
                \vlfunc{\abs{ %
                    \operatorname{Hess}_{\riemannmetric*}\parentheses*{\expect[\nu]{f}} %
                }}%
            }{H} %
            + %
            \norm{%
                \vlfunc{\abs{ \grad{\riemannmetric*}\Phi }{\riemannmetric*}} %
                \pdot %
                \vlfunc{\abs{ \grad{\riemannmetric*}\expect[\nu]{f} }{\riemannmetric*}} %
            }{H} %
        }
    .
    \end{align*}
    Due to the form of~$PA^2P$ we derived in~\hyperref[eq:Langevin:PA2P]{Equation~\eqref{eq:Langevin:PA2P}}, 
    we know that~$u \vdef Pf$ solves the elliptic equation 
    \begin{align*}
        &
        u 
        - 
        \frac1{\beta}\,
        \laplace{\whormetric*}P_Su
        =
        g
        \\
        \text{in }\qquad&
        \left\{ 
            u\in\tanproj^\ast\bigL2{\posfold}{\wriemannmetric*}
        \ \middle\vert\ 
            \exists\, f_0\in\smoothcompact{\posfold}\colon\ 
            u 
            = 
            \vlfunc{f_0} 
            -
            \scalarprod{\vlfunc{f_0}}{1}{H} 
        \right\}
    .
    \end{align*}
    As we assumed the necessary potential conditions, 
    the a priori estimates of Dolbeault, Mouhot and Schmeiser, 
    cf.~\cite[Appendix]{HypocoercJFA}, 
    yield existence of a constant~$c_2\in(0,\infty)$ independent of~$Pf$ and~$g$
    such that
    \begin{displaymath}
        \norm{ \adjoint{(BA)}g }{H}
        \leq
        c_2
        \pdot
        \norm{ Pg }{H}
        \leq
        c_2
        \pdot
        \norm{ g }{H}
    .
    \end{displaymath}
    Now, \cite[Propositions~2.15]{HypocoercJFA} does apply which finishes the proof.
\end{proof}

Collecting the individual results of~\autoref{subsec:Langevin:D} and~\autoref{subsec:Langevin:H}
we can infer~\hyperref[thm:Langevin:main]{Theorem~\ref{thm:Langevin:main}}
using the~\ref{thm:hypocoercivity-thm-named}.

%% file: constrained.tex
\section{Hypocoercivity for the fibre lay-down model}
\label{sec:fld}

In this section we demonstrate that the hypocoercivity method does apply
to the fibre lay-down model with the position space being a Riemannian manifold.
In other words, we generalise the results from~\cite{HypocoercJFA} 
to the case of higher-order SDAEs with abstract position manifolds.
The statement, which is an application of the~\ref{thm:hypocoercivity-thm-named}, reads as follows:
\begin{theorem}[Hypocoercivity of the geometric fibre lay-down dynamic]\hfill
\label{thm:fld:main}

    Let $\sigma\in(0,\infty)$ 
    and $(\posfold,\riemannmetric*)$ be a Riemannian manifold satisfying~\ref{cond:M}. 
    Furthermore, let the potential conditions~\ref{cond:named:Langevin:P} hold.
    We assume that  
    $
        \spray{\riemannmetric*}\vert_{\unittanbundle\posfold}\vlfunc{\Psi} 
        = 
        (\posdim-1)\,\hlfunc{\Psi}
    $
    holds.
    Denote by~$\nu$ the normalised surface measure on~$(\sphere{\posdim-1},\borel{\sphere{\posdim-1}})$
    and define 
    $
        \mu
        \vdef
        \Leb{\wriemannmetric*}\locprod\nu  
    $.

    Then, the fibre lay-down operator 
    \begin{displaymath}
        \parentheses*{
            L, 
            \smoothcompact{\unittanbundle{\posfold}}
        }
        =
        \parentheses*{
            \frac{\sigma^2}2\,\spherelaplace
            +
            \spray{\riemannmetric*}
            -
            \frac1{\posdim-1}\,
            \spheregrad(\spray{\riemannmetric*}\vlfunc{\Psi}),
            \smoothcompact{\unittanbundle{\posfold}}
        }
    \end{displaymath}
    is closable in $H\vdef\bigL2{\tanbundle{\unittanbundle\posfold}}{\mu}$. 
    Moreover, its closure~\operator{L} generates a strongly continuous contraction semigroup $(T_t)_{t\in[0,\infty)}$.
    Finally, there are constants $\kappa_1,\kappa_2\in(0,\infty)$ 
    computable in terms of $\posdim=\dim(\posfold)$, $\Lambda$, $c$ and $\sigma$  
    such that for all $g\in H$ and times $t\in[0,\infty)$ holds 
    \begin{displaymath}
        \norm{ T_tg-\scalarprod{g}{1}{H} }{H} 
        \leq 
        \kappa_1 \euler^{-\kappa_2 t}\norm{ g-\scalarprod{g}{1}{H} }{H}
    .
    \end{displaymath}
\end{theorem}
Langevin-type models serve as surrogate models for fibre dynamics under a turbulent regime, 
since the model including all physical details, 
see~\cite{MW06}, 
leads to enormous computational effort.
In view of the production process, 
it's a reasonable model assumption that
velocities should be normalised, 
as fibre filaments are extruded with constant speed, 
see e.\,g.\ \cite{KlarMaringerWegener,GS13,PhD-Stilgenbauer} and various references therein. 
Hence, in contrast to~\autoref{sec:Langevin} 
we look at an SDAE on~\tanbundle\posfold 
incorporating the `polynomial-type' normalisation assumption on the velocities. 
In distinction from fibre lay-down applications with a different focus like in~\cite{Lindner-Holger17}, 
we subsume possible side conditions on the position variables just in the position manifold~\posfold.
Whenever we speak of an SDAE in this paper, 
there is no ambiguity wrt.\ the nature of the algebraic side condition.

We implement the algebraic side condition geometrically  
via replacing the standard fibre $F=\Rnum^{\posdim}$ by $F=\sphere{\posdim-1}$, 
thus the tangent bundle over~\posfold by the unit tangent bundle 
$\unittanproj\colon\,\unittanbundle\posfold \longmapsin \posfold$.
This modification has several side effects, 
as the unit tangent bundle is a submanifold 
$\configfold\vdef\unittanbundle\posfold\leq\tanbundle\posfold$
and consequently we consider the SDAE as an SDE in the sub\-/fibre bundle
$\tanunittanbundle\posfold \leq \tantanbundle\posfold$.
Our choice of the Sasaki metric on~\tanbundle\posfold is important here,
as the normal bundle of~\unittanbundle\posfold, 
which is just the quotient bundle 
$\sfrac%
    { \tantanbundle\posfold\vert_{\unittanbundle\posfold} }%
    { \tanunittanbundle\posfold }%
$ in the first place, 
can be realised as the orthogonal complement
$
    \normalunittanbundle\posfold
    \vdef
    \tanunittanbundle\posfold^{\perp}
$
of~\tanunittanbundle\posfold wrt.\ the chosen metric.
This gives us the normal bundle really as a sub\-/fibre bundle
$\normalunittanbundle\posfold \leq \tantanbundle\posfold$
and the Whitney sum 
$
    \tantanbundle\posfold \vert_{\unittanbundle\posfold}
    = 
    \tanunittanbundle\posfold 
    \oplus 
    \normalunittanbundle\posfold
$.
It should not be too surprising at this point that
we can just restrict the horizontal bundle as 
$
    \horbundle{\unittanbundle\posfold}
    \vdef
    \horbundle{\tanbundle\posfold}\vert_{\unittanbundle\posfold}
$, 
but are forced to replace the vertical lift by another lifting procedure.
We think that the problem has been explained very well in~\cite[Abschnitt~2.1]{Fabrega}:
Let's consider $v\in\unittanbundle{\posfold}$ and $a\in\tanunittanbundle{\posfold}{v}$. 
Then, there exists a $w\in\tanbundle{\posfold}{\tanproj(v)}$ such that $w$ is vertically lifted to $a$ at $v$, 
i.\,e.\ $\vertlift_v(w)=a$. 
However, we do not find a vector field $\Xfield\in\smoothsec{\tanbundle\posfold}$ 
-- not even a local vector field --
such that
$\vertlift_v(\Xfield) = a$ 
and 
$
    \vertlift(\Xfield)\bigr\vert_{\unittanbundle\posfold} 
    \in \smoothsec{\tanunittanbundle\posfold}
$
hold.
In other words, the vertical lift of a vector field does not need to be tangent to the unit tangent bundle.
Thus, we adapt the vertical lift slightly in order to guarantee the lift of elements in \tanbundle{\posfold}{\tanproj(v)} are elements of \tanbundle{\unittanbundle{\posfold}}{v}.
For the next definition see also~\cite{BoeckxVanhecke97,BoeckxVanhecke01}. 

\begin{definition}[tangential lift]
\label{def:tanlift}
    Let $x\in\posfold$, $v\in\tanbundle{\posfold}{x}$ and $u\in\unittanbundle{\posfold}{x}$. 
    The \emph{tangent lift of $v$} is defined as
    \begin{displaymath}
        \tanlift_u(v) 
        \vdef
        \vertlift_u\left( v - \riemannmetric{v}{u}\pdot u \right)
        =
        \vertlift_u(v) - \riemannmetric{v}{u}\pdot\normalfield{\sasakimetric*}(u)
        ,
    \end{displaymath}
    where the unit normal vector field 
    $\normalfield{\sasakimetric*}\in\smoothsec{\normalunittanbundle\posfold}$ has the following properties:
    \begin{displaymath}
        \dualpair{\normalfield{\sasakimetric*}}{d\tanproj} 
        = 0
        \qquad\text{and}\qquad
        \dualpair{\normalfield{\sasakimetric*}}{d\kappa} 
        = \ID_{\unittanbundle\posfold}
        .
    \end{displaymath}
    We call~\normalfield{\sasakimetric*} the \emph{Sasakian normal vector field}, 
    since this normal vector field depends on our choice of the Sasaki metric on~\tanbundle\posfold
    as outlined above.
\end{definition}

The tangential lift enables us to write 
$
    \tanunittanbundle\posfold
    =
    \tanlift(\unittanbundle\posfold)
    \oplus
    \horbundle{\unittanbundle\posfold}
$ 
as an Ehresmann connection
and furthermore the decomposition
\begin{displaymath}
    \tantanbundle\posfold\vert_{\unittanbundle\posfold}
    =
    \tanlift(\unittanbundle\posfold)
    \oplus
    \horbundle{\unittanbundle\posfold}
    \oplus 
    \normalunittanbundle\posfold
.
\end{displaymath}
That said, starting from~\autoref{subsec:fld:D}
we might write~\vlfunc{f_0} for given $f_0\in\smoothcompact{\posfold}$ and 
from the context it should be clear that we mean 
$
    \unittanproj^\ast f_0 
    = 
    \vlfunc{f_0}\vert_{\unittanbundle\posfold}
    \in\smoothcompact{\unittanbundle\posfold}
$.
Similarly, we just write~$\spray{\riemannmetric*}$
instead of~$\spray{\riemannmetric*}\vert_{\unittanbundle\posfold}$ et cetera.
Be that as it may, 
one should pay attention to read e.\,g.\ 
$
    \grad{\hormetric*}\vlfunc{f_0} 
    = 
    \horlift\vert_{\unittanbundle\posfold}(\grad{\riemannmetric*}f_0)
$
with $\horlift\vert_{\unittanbundle\posfold}$ 
being the restriction of the horizontal lift to~\unittanbundle\posfold.
Moreover, the horizontal lift of~$f_0$ is up to constants characterised by
$
    \dualpair{a}{d\hlfunc{f_0}} 
    =
    \dualpair{a}{df_0\bincirc d\kappa} 
$
for all $a\in\tanunittanbundle\posfold$, 
and one may think it as 
$
    \hlfunc{f_0}(v)
    =
    \int_{\tanunittanbundle\posfold{v}} df_0\bincirc d\kappa\,\diffd\Leb
$. 

Alas, the term `tangential' lift used in the literature is inherently flawed 
even more than the terms `vertical' and `horizontal'.
As terms like `tangential gradient' might cause serious confusion, 
we depart from our naming scheme as follows. 

\begin{definition}[spherical lift, gradient, divergence and Laplacian]
\hfill 
    \begin{enumerate}[label={(\roman*)},ref={(\roman*)}]
        \item 
        For all $\Xfield\in\smoothsec{\tanbundle\posfold}$ 
        we call $\tanlift(\Xfield)\in\smoothsec{\tanbundle\posfold}{\tanunittanbundle\posfold}$ the \emph{spherical lift of~\Xfield}.
        \item
        The \emph{spherical gradient~\spheregrad} is defined by the relation  
        \begin{displaymath}
            \spheregrad f
            \in
            \smoothsec{\tanlift(\unittanbundle\posfold)}
            \quad\text{and}\quad
            \sasakimetric%
                {\spheregrad f}%
                {\tanlift{\Xfield}}%
            =
            \tanlift(\Xfield)f
        \end{displaymath}
        for $f\in\smoothfunc{\unittanbundle\posfold}$ 
        and $\Xfield\in\smoothsec{\tanbundle\posfold}$ arbitrary.
        \item
        The \emph{spherical divergence~\spheredivergence} is defined via 
        \begin{displaymath}
            \Liederivative{\tanlift\Xfield} \vertmetric*\vert_{\tanunittanbundle\posfold}
            =
            \spheredivergence(\tanlift\Xfield)
            \pdot
            \vertmetric*\vert_{\tanunittanbundle\posfold}
        \end{displaymath}
        for $\Xfield\in\smoothsec{\tanbundle\posfold}$ arbitrary, 
        where~$\Liederivative$ again denotes the Lie~derivative.
        \item
        The \emph{spherical Laplace-Beltrami operator~\spherelaplace} 
        is defined by 
        $\spherelaplace\vdef\spheredivergence(\spheregrad)$, 
        as usual.
    \vspace*{-3ex}
    \end{enumerate}
\end{definition}

Basically, these are just the natural modifications of the elementary objects within our calculus.
E.\,g.\ it's easily verified that
$\spheregrad\hlfunc{f_0} = \tanlift(\grad{\riemannmetric*}f_0\bincirc\unittanproj)$ 
holds for all $f_0\in\smoothfunc{\posfold}$.
Other statements from~\hyperref[lem:Sasaki-operators]{Lemma~\ref{lem:Sasaki-operators}} can be translated similarly.
In this paper we only consider the normalised surface measure on the sphere, 
therefore weighted vertical gradients etc.\ do not appear.
We propose the notation 
$\grad{\vertmetric* {\scriptscriptstyle\vert\mathrm{S}}}\vdef\spheregrad$
for that we can use a boldfaced `v' to signify the weighted vertical gradient~\wspheregrad 
in case of a nonconstant fibre weight.

The algebraic side condition yields some more interesting effects.
We want to fix a very important consequence of the relation
\begin{equation}
\label{eq:eigenvalue-eq}
    \laplace{F}\ID_F 
    = 
    -(\posdim-1)\ID_F
    \qquad\text{for } F=\sphere{\posdim-1}
,
\end{equation} 
where the Laplacian is taken componentwise in standard Euclidean coordinates.
See~\cite[Lemma~3.2]{HypocoercJFA} 
and for a general proof on eigenvalues of the spherical Laplace-Beltrami 
we refer e.\,g.\ to~\cite[Theorem~1.4.5]{SphericalHarmAna}.
\begin{lemma}
\label{lem:Laplace-IDunittan}
    It holds 
    $\laplace{\sphere}\spray{\riemannmetric*} = -(\posdim-1) \spray{\riemannmetric*}$, 
    where we also denote by~$\laplace{\sphere}$ the spherical tensor Laplacian
    and think of the vector field~\spray{\riemannmetric*} as a $(1,0)$\=/tensor field.
\end{lemma}
\begin{proof}
    For fixed $w\in\unittanbundle\posfold$ with $x\vdef\unittanproj(w)$ we get that
    \begin{align*}
        \hormetric%
            {\laplace{\sphere}\spray{\riemannmetric*}%
                \vert_{\unittanbundle{\posfold}{x}}%
            }%
            {\horlift(w)}%
            {w}
        &=
        \laplace{\sphere}
        \hormetric%
            {\spray{\riemannmetric*}%
                \vert_{\unittanbundle{\posfold}{x}}%
            }%
            {\horlift(w)}%
            {w}
        \\
        &=
        \laplace{\sphere}
        \riemannmetric%
            {\ID_{\unittanbundle{\posfold}{x}}}%
            {w}%
            {x}
        =
        \riemannmetric%
            {\laplace{F}\ID_{F}}%
            {w}%
            {x}
    .
    \end{align*}
    Use~\hyperref[eq:eigenvalue-eq]{Equation~\eqref{eq:eigenvalue-eq}}.
\end{proof}

Consider the following Stratonovich~SDE in~\unittanbundle\posfold: 
\begin{equation}
\label{eq:fld}
    \diffd \eta
    =
    \spray{\riemannmetric*}
    \ \diffd t
    +
    \tanlift_\eta(-\grad{\riemannmetric*}\Psi) 
    \ \diffd t
    + 
    \sigma \pdot \tanlift_\eta\left( %
        \sum_{j=1}^\posdim \frac{\partial}{\partial{x_\eta^j}} %
        \right)
    \stratonovich\diffd W_t
,
\end{equation}
where the chart $\parentheses*{x_\eta^1,x_\eta^2,\ldots,x_\eta^\posdim}$ at $\unittanproj(\eta)$
provides normal coordinates
and~$\sigma$ is a nonnegative diffusion parameter.
Note that neither we rescale the potential 
nor we incorporate a friction term.
We call~\hyperref[eq:fld]{Equation~\eqref{eq:fld}}
the \emph{fibre lay-down equation on~\posfold} 
or just \emph{geometric fibre lay-down model}.
The corresponding Kolmogorov generator attains the form
\begin{equation}
\label{eq:fld-generator}
    L
    =
    \spray{\riemannmetric*}
    -
    \tanlift(\grad{\riemannmetric*}\Psi)
    +
    \frac{\sigma^2}2\,\spherelaplace
.
\end{equation}
We call~$L$ as in~\hyperref[eq:fld-generator]{Equation~\eqref{eq:fld-generator}}
the \emph{fibre lay-down generator}.
Per se this operator is defined for all smooth functions on the tangent space, 
whilst the obvious choice for the domain of test functions is 
$D\vdef\smoothcompact{\unittanbundle\posfold}$.
Modificating the proof of~\hyperref[lem:D0isdense]{Lemma~\ref{lem:D0isdense}} slightly 
we get that
\begin{displaymath}
        D_{0\vert\mathrm{S}}
        \vdef
        \unittanproj^\ast\smoothcompact{\posfold}
        \tensorprod
        \kappa^\ast\smoothcompact{\posfold}
\end{displaymath}
is dense in~$D$ again.

In the next sections we restrict ourselves to computations substantially different from~\autoref{sec:Langevin}.
Briefly speaking, the differences occur due to the change of the fibre measure space 
and affect some of the constants.

\subsection{Data conditions}
\label{subsec:fld:D}

\begin{definition}[model Hilbert space~\ref{cond:D1}]
\label{def:fld:D1}
    Consider the probability space 
    \begin{displaymath} 
        (E,\mathfrak{E},\mu)
        =
        \parentheses*{%
            \unittanbundle{\posfold}, 
            \borel{\unittanbundle{\posfold}},
            \Leb{\wsasakimetric*}
        }
    ,
    \end{displaymath} 
    where $\Leb{\wsasakimetric*} = \Leb{\wriemannmetric*} \locprod \nu$ 
    is the weighted Sasaki volume measure with 
    $\wriemannmetric*$ weighted by 
    $\rho_{\posfold}\vdef \exp(-\Psi)$ such that 
    \Leb{\wriemannmetric*} is a probability measure on~$(\posfold,\borel{\posfold})$, 
    and~$\nu$ 
    is the normalised surface measure on 
    $
        (F,\mathfrak{F}) 
        =
        \parentheses*{ \sphere{\posdim-1},\borel{\sphere{\posdim-1}} }
    $, 
    thus the fibre weight is a constant factor.
    The model Hilbert space is 
    $H \vdef \bigL2{E}{\mu} = \bigL2{\unittanbundle{\posfold}}{\wsasakimetric*}$.
\end{definition}

Note that our choice of~$\nu$ is the only possible for a probability measure on the measurable fibre space 
with a density that is invariant wrt.\ rotations.
Furthermore, we point out that up to slight modifications 
we could keep the set of conditions~\ref{cond:named:Langevin:P} of~\autoref{sec:Langevin}. 
The assumption of a weakly harmonic potential in~\hyperref[lem:Langevin:SAD]{Lemma~\ref{lem:Langevin:SAD}}, 
which in the end was not necessary,
turns into another condition that can not be overcome so easily.
Specifically, we require the potential to satisfy the relation 
\begin{equation}
\label{eq:fld:potential-condition}
    \spray{\riemannmetric*}\vlfunc{\Psi} 
    = 
    (\posdim-1)\,\hlfunc{\Psi}
    \qquad\text{on }\unittanbundle\posfold
\end{equation} 
up to an additional constant summand. 
This will become evident during the proof of the oncoming lemma.
Indeed, the fibre lay-down generator attains the form analogous to~\cite[Equation~3.18]{HypocoercJFA} under this assumption.
Later on in~\hyperref[lem:fld:SAD2]{Lemma~\ref{lem:fld:SAD2}} 
we seemingly get rid of~\hyperref[eq:fld:potential-condition]{Assumption~\eqref{eq:fld:potential-condition}}
using the Poisson bracket again. 
But it turns out that this result doesn't fit our purposes 
and we want to add~\hyperref[eq:fld:potential-condition]{Assumption~\eqref{eq:fld:potential-condition}}
to the set of conditions on the potential.

\begin{example}[%
    {\hyperref[eq:fld:potential-condition]{Assumption~\eqref{eq:fld:potential-condition}}} %
    for Euclidean position space]
    Let $\posfold=\Rnum_x^\posdim$ be endowed with standard Euclidean metric 
    $
        \riemannmetric*
        =
        \parentheses*{\scalarprod{\cdot}{\cdot}{\mathrm{euc}}}_{x\in\Rnum^\posdim}$.
    Then, the Riemannian semispray~\spray{\mathrm{euc}} effectively is just the identity mapping.
    The interested reader easily verifies this in local coordinates.
    More formally, 
    every 
    $
        a^\prime
        \in
        \tanbundle*{\tanbundle\posfold}
        \simeq
        \parentheses*{\Rnum_x^\posdim\times\Rnum_v^\posdim}^\ast
    $
    is identified with an $a=(a_v,a_x)^\intercal\in\Rnum^{2\posdim}$ via 
    $
        a^\prime(x,v)
        =
        \scalarprod{v}{a_v}{\mathrm{euc}}
        +
        \scalarprod{x}{a_x}{\mathrm{euc}}
    $
    for all $(x,v)\in\Rnum_x^\posdim\times\Rnum_v^\posdim$.
    Then, the semispray~\spray{\mathrm{euc}} is characterised by
    \begin{displaymath}
        \parentheses*{\hlfunc{\Omega}\bincirc\spray{\mathrm{euc}}}_w(a)
        =
        \scalarprod{w}{a_x-a_v}{\mathrm{euc}}
        \qquad\text{for all }(x,w)\in\Rnum_x^\posdim\times\Rnum_v^\posdim,a\in\Rnum^{2\posdim}
    \end{displaymath}
    cf.~\hyperref[rem:Riemann-spray-revisited]{Remark~\ref{rem:Riemann-spray-revisited}}.
    We have chosen the notation $a=(a_v,a_x)^\intercal$ for sake of readability in view of this remark.
    Hence, we can think of the mapping~$\spray{\mathrm{euc}}(w)$ 
    as the gradient of $z\mapsto U_w(z)\vdef\scalarprod{w}{z}{\mathrm{euc}}$.
    When restricting the semispray to $\configfold=\unittanbundle{\Rnum_x^\posdim}$, 
    i.\,e.\ $w\in\sphere{\posdim-1}$,
    we calculate via usual integration by parts that
    \begin{align*}
        &\int_{\unittanbundle{\Rnum^\posdim}{x}} 
            \spray{\riemannmetric*}f
        \ \diffd\nu
        =
        \int_{\unittanbundle{\Rnum^\posdim}{x}} 
            \scalarprod{\spheregrad U_v(z)}{\spheregrad f(v)}{\mathrm{euc}}
        \ \diffd\nu(v)
        \\
        ={}&
        -
        \int_{\unittanbundle{\Rnum^\posdim}{x}} 
            \spherelaplace U_v(z) \pdot f(v)
        \ \diffd\nu(v)
        \overset{\eqref{eq:eigenvalue-eq}}{=}
        (\posdim-1) 
        \int_{\unittanbundle{\Rnum^\posdim}{x}} 
            \scalarprod{v}{z}{\mathrm{euc}} \pdot f(v)
        \ \diffd\nu(v)
    \end{align*}
    for all $f\in\smoothcompact{\unittanbundle{\Rnum^\posdim_x}}$, $z\in\Rnum^\posdim$.
    Cf.~\cite[Lemma~3.3]{HypocoercJFA}.

    Now, let $f_0\in\smoothfunc{\Rnum_x^\posdim}$
    and $a=(a_v,a_x)^\intercal\in\tanbundle\configfold$ 
    with $d\tanproj(a)=a_x$ as well as $d\kappa(a)=a_v=v\vdef\tanproj(a)$.
    In this situation we have that
    \begin{displaymath}
        \dualpair{a}{d\hlfunc{f_0}}
        =
        \dualpair{(a_v,a_x)}{df_0 \bincirc d\kappa} 
        =
        \dualpair{v}{df_0\vert_{\unittanbundle{\Rnum^\posdim}}} 
        =
        \scalarprod%
            {v}%
            {\frac{\grad{x}f_0}{\abs{\grad{x}f_0}}\bincirc\unittanproj(v)}%
            {\mathrm{euc}}
    .
    \end{displaymath}
    Thus, 
    $
        \hlfunc{f_0}(z)
        = 
        U_{%
            \sfrac{\grad{x}f_0}{\abs{\grad{x}f_0}}(x)%
        }(z)
    $
    for all $z\in\unittanbundle{\Rnum^\posdim}{x}$.
    Both results can be combined as follows:
    \begin{align*}
        \int_{\unittanbundle{\Rnum^\posdim}{x}} 
            \spray{\mathrm{euc}}\vlfunc{\Psi}
            \pdot
            \hlfunc{g_0}
        \ \diffd\nu
        &=
        \int_{\unittanbundle{\Rnum^\posdim}{x}} 
            \spray{\mathrm{euc}}\parentheses*{
                \vlfunc{\Psi}\pdot\hlfunc{g_0}
            }
        \ \diffd\nu
        \\
        &=
        (\posdim-1)
        \int_{\unittanbundle{\Rnum^\posdim}{x}} 
            \scalarprod%
                {v}%
                {\frac{\grad{\mathrm{euc}}\Psi(x)}{\Psi(x)\,\abs{\grad{\mathrm{euc}}\Psi(x)}}}%
                {\mathrm{euc}}
            \pdot
            \vlfunc{\Psi}(v) \, \hlfunc{g_0}(v)
        \ \diffd\nu(v)
        \\
        &=
        (\posdim-1)
        \int_{\unittanbundle{\Rnum^\posdim}{x}} 
            \hlfunc{\Psi}(v)
            \pdot
            \hlfunc{g_0}(v)
        \ \diffd\nu(v)
    \end{align*}
    holds for all $g_0\in\smoothcompact{\Rnum_x^\posdim}$ 
    with the particular choice of 
    $
        z 
        = 
        \frac1{\Psi(x)}\,\frac{\grad{\mathrm{euc}}\Psi(x)}{\abs{\grad{\mathrm{euc}}\Psi(x)}}
    $.
    Note that we assume $\Psi>0$ wlog.\ in view of~\ref{cond:Langevin:P1}.
    Hence, \hyperref[eq:fld:potential-condition]{Assumption~\eqref{eq:fld:potential-condition}} always is fulfilled.
\end{example}

\begin{lemma}[SAD-decomposition~\ref{cond:D3},~\ref{cond:D4},~\ref{cond:D6}]
\label{lem:fld:SAD}
    Let the potential~$\Psi$ be loc-Lipschitzian 
    such that \hyperref[eq:fld:potential-condition]{Assumption~\eqref{eq:fld:potential-condition}} is fulfilled. 
    Consider the SAD\-/decomposition $L=S-A$ on~$D$ with
    \begin{align*}
        Sf 
        &\vdef
        \frac{\sigma^2}2\,\spherelaplace f 
        \\
        \quad\text{and}\quad
        Af 
        &=
        -\spray{\wriemannmetric*}f
        \vdef
        -\spray{\riemannmetric*}f
        +
        \frac1{\posdim-1}\,
        \spheregrad(\spray{\riemannmetric*}\vlfunc{\Psi})
    \end{align*}
    for all $f\in D$.

    Then, the following assertions hold:
    \begin{enumerate}[label={(\roman*)}]
        \item 
        \label{itm:fld:SAD:S}
        $(S,D)$ is symmetric and negative semidefinite.
        \item 
        \label{itm:fld:SAD:A}
        $(A,D)$ is antisymmetric.
        \item 
        \label{itm:fld:SAD:L}
        For all $f\in D$ we have that 
        $Lf \in \bigL1{\unittanbundle{\posfold}}{\mu}$ and 
        $\int_{\unittanbundle{\posfold}} Lf\ \diffd\mu = 0$.
    \end{enumerate}
\end{lemma}
\begin{proof}
\hfill

    \begin{enumerate}[label={(\roman*)}]
        \item 
        Using integration by parts we see that 
        $(S,D)$~pregenerates the weighted spherical gradient form on~\unittanbundle\posfold.
        Cf.\ the proof of part~\ref{itm:Langevin:SAD:S} 
        of~\hyperref[lem:Langevin:SAD]{Lemma~\ref{lem:Langevin:SAD}}.
        \item
        The adjoint of 
        $
            \spheregrad\hlfunc{\Psi}
            \overset{\eqref{eq:fld:potential-condition}}{=}
            \frac1{\posdim-1}\, \spheregrad(\spray{\riemannmetric*}\vlfunc{\Psi})
        $
        wrt.~\bigL2{\unittanbundle\posfold}{\wsasakimetric*}\=/scalar product 
        is computed using~\hyperref[lem:Laplace-IDunittan]{Lemma~\ref{lem:Laplace-IDunittan}} as
        \begin{align*}
            \adjoint*{ %
                \frac1{\posdim-1}\, %
                \grad{\sphere}(\spray{\riemannmetric*}\vlfunc{\Psi}) %
            }
            &=
            -
            \frac1{\posdim-1}\,
            \grad{\sphere}(\spray{\riemannmetric*}\vlfunc{\Psi})
            -
            \frac1{\posdim-1}\,
            \laplace{\sphere}(\spray{\riemannmetric*}\vlfunc{\Psi})
            \\
            &=
            -
            \frac1{\posdim-1}\,
            \grad{\sphere}(\spray{\riemannmetric*}\vlfunc{\Psi})
            +
            \spray{\riemannmetric*}\vlfunc{\Psi}
        .
        \end{align*}
        The rest follows as in the proof of part~\ref{itm:Langevin:SAD:A} 
        of~\hyperref[lem:Langevin:SAD]{Lemma~\ref{lem:Langevin:SAD}}.
        \item
        Follows with the parts~\ref{itm:fld:SAD:S}
        and~\ref{itm:fld:SAD:A}.
    \vspace*{-3ex}
    \end{enumerate}
\end{proof}

Giving up on the form of~$(A,D)$ as in~\cite[Equation~(3.22)]{HypocoercJFA}
and turning to a more intuitive one 
in view of~\hyperref[eq:fld-generator]{Equation~\eqref{eq:fld-generator}}, 
we gain weaker assumptions on the potential in general.

\begin{lemma}[SAD-decomposition (2nd version)]
\label{lem:fld:SAD2}
    If we define
    $
        \spray{\wriemannmetric*}
        \vdef
        \spray{\riemannmetric*}
        -\tanlift\parentheses*{\grad{\riemannmetric*}\Psi}
    $,
    then the assertions of~\hyperref[lem:fld:SAD]{Lemma~\ref{lem:fld:SAD}}
    are true without the~\hyperref[eq:fld:potential-condition]{Assumption~\eqref{eq:fld:potential-condition}} on~$\Psi$.
\end{lemma}
\begin{proof}
    Indeed, we can copy the proof of~\hyperref[prop:Langevin:SAD]{Proposition~\ref{prop:Langevin:SAD}} 
    and it's enough that the configuration manifold~\configfold is a submanifold of~\tanbundle\posfold.

    First, we notice that both~$\Omega$ and~$\mathbb{J}$ 
    restricted to~\tanbundle*\configfold or~\tanbundle\configfold respectively  
    are still a symplectic form and an almost complex structure respectively.
    They are compatible with the restricted Sasaki metric, 
    thus generate the same Poisson bracket on~\configfold 
    which in turn defines Hamiltonian vector fields
    $\Hamilton{f}\in\smoothsec{\tanbundle\configfold}$ 
    for all $f\in\smoothfunc{\configfold}$. 
    Second, we find that 
    $
        \Hamilton{f}(\rho)
        =
        -\rho\pdot\parentheses*{
            \spray{\riemannmetric*}f 
            - 
            \tanlift\parentheses*{\grad{\riemannmetric*}\Psi}(f)
        }
    $ 
    for all $f\in D=\smoothcompact{\configfold}$, 
    by investigating the action of the Hamiltonian vector field 
    in local coordinates for~\configfold 
    that respect the Ehresmann connection 
    and also provide a local trivialisation. 
\end{proof}

\label{fld:D2-loc-Lipschitz}
In contrast to~\hyperref[prop:Langevin:SAD]{Proposition~\ref{prop:Langevin:SAD}} 
the statement of~\hyperref[lem:fld:SAD2]{Lemma~\ref{lem:fld:SAD2}} 
is not of much use for an application of the hypocoercivity method, 
even though the operator 
$(A,D)=(-\spray{\wriemannmetric*},D)$ 
there is the more natural formulation.
The simple reason is 
that the two competing definitions of~\spray{\wriemannmetric*} not necessarily coincide on~$D$. 
But we need the operator as in~\hyperref[lem:fld:SAD]{Lemma~\ref{lem:fld:SAD}} 
during the characterisation of $(PA^2P,D)$ 
specifically in~\hyperref[eq:fld:A2P-part2]{Equation~\eqref{eq:fld:A2P-part2}}.
However, this calculation is part of checking the hypocoercivity assumptions,  
whereas condition~\ref{cond:D2}, 
the existence of a nice semigroup, 
can be checked following the same steps as in~\autoref{subsec:Langevin:D} 
for the antisymmetric operator of~\hyperref[lem:fld:SAD2]{Lemma~\ref{lem:fld:SAD2}}.
\smallskip

Moving on, the fibrewise average is defined the very same way as in~\autoref{sec:Langevin} 
just with `\unittanbundle{\posfold}{x}' instead of `\tanbundle{\posfold}{x}'.
Also, the form of the operator $(AP,D)$ 
given in~\hyperref[eq:Langevin:APf-Eq2]{Equation~\eqref{eq:Langevin:APf-Eq2}} 
just changes marginally:
\begin{equation}
\label{eq:fld:APf-Eq2}
    APf
    =
    -\riemannmetric{%
            \ID_{\unittanbundle\posfold} %
        }{%
            \grad{\riemannmetric*}\expect[\nu]{f} \bincirc \unittanproj %
        }{\unittanproj}
    \qquad\text{for all }f\in D
.
\end{equation}

Indeed, the other statements concerning data conditions
translate to the fibre lay-down model on~\posfold with minor modifications. 
This is a little bit different 
when it comes to the hypocoercivity conditions in the next section.
Nevertheless, we want to draw the readers attention to the fact
the reasoning for essential m-dissipativity of $(L,D)$ 
under the assumption of loc-Lipschitzian potentials 
does barely depend on the fibre measure space.
Arguments gleaned in~\cite[Section~4]{HypocoercJFA} 
on the notoriously subtile question of core property and~$L$ generating a semigroup stay valid. 

Since the standard fibre is compact now, 
we can simplify the proof~\hyperref[lem:Langevin:D5D7]{Lemma~\ref{lem:Langevin:D5D7}} a bit, 
similar to~\cite[Lemma~3.8]{HypocoercJFA}.
\begin{lemma}
    Let condition~\ref{cond:Langevin:P1} hold.
    Then, we have
    $P(H)\subseteq\opdomain{S}$, $SP=0$, 
    $P(D)\subseteq\opdomain{A}$ and $AP(D)\subseteq\opdomain{A}$.
    Furthermore, $1\in\opdomain{L}$ and $L1 = 0$.
\end{lemma}
\begin{proof}
    The range~$P(H)$ is identified with a subset of~\bigL2{\posfold}{\wriemannmetric*} 
    via the vertical lift.
    For any $f_0\in\bigL2{\posfold}{\wriemannmetric*}$ 
    there is an \bigL2\=/approximating sequence $(f_n)_{n\in\Nnum\setminus\{0\}}$ in~\smoothcompact{\posfold}. 
    Since the standard fibre $F=\sphere{\posdim-1}$ is compact, 
    it holds $\vlfunc{f_n}\in D$ and $S(\vlfunc{f_n}) = 0$ for all $n\in\Nnum\setminus\{0\}$. 
    We conclude that
    $\vlfunc{f_0}\in\opdomain{S}$ and $\vlfunc{f_0}\in\nullspace{S}$
    as $\vlfunc{f_n}\longrightarrow \vlfunc{f_0}$ in~$H$ as $n\to\infty$ and 
    $\operator{S}$ is closed.
 
    We fix an $f\in D$.
    Choose $o\in\posfold$ and 
    an open ball~\ball{o}{r} centred at~$o$ with radius $r\in(0,\infty)$ wrt.\ the intrinsic metric on~$(\posfold,\riemannmetric*)$ such that
    the support of~$f$ is completely contained in 
    $
        \unittanproj^{-1}(\ball{o}{r})
        \subseteq
        \unittanbundle\posfold
    $.
    Than, the support of~$\expect[\nu]{f}$ is contained in~\ball{o}{r}.
    Thus, $P_Sf\in D$.
    Therefore, $P(D)\subseteq D \subseteq \opdomain{A}$.

    Besides, we calculate via chain rule that
    \begin{align}
    \label{eq:fld:APf-Eq1}
    \begin{aligned}
        APf
        &=
        -\spray{\wriemannmetric*}\parentheses*{\vlfunc{\expect[\nu]{f}}}
        =
        -\dualpair{%
            \spray{\riemannmetric*}%
        }{%
            d\expect[\nu]{f} \bincirc d\tanproj
        }
        \\
        &=
        -\dualpair{ \ID_{\tanbundle\posfold} }{ d\expect[\nu]{f} }
        =
        -d\expect[\nu]{f}
    .
    \end{aligned}
    \end{align}
    Consequently, $AP(D)\subseteq D \subseteq \opdomain{A}$, 
    as the right-hand side of~\hyperref[eq:fld:APf-Eq1]{Equation~\eqref{eq:fld:APf-Eq1}} 
    is smooth with compact support.
    Let $\varphi\in\smoothcompact{\posfold;[0,1]}$ be a cut-off function
    such that $\varphi=1$ on~\ball{o}{1} 
    and $\varphi=0$ on~\ball{o}{2}.
    Define $\varphi_n\vdef\varphi(\sfrac{\ID}{n})$ for all $n\in\Nnum\setminus\{0\}$.
    Note that 
    $
        \abs{\grad{\riemannmetric*}\varphi_n(x)}{\riemannmetric*} 
        \leq 
        \frac1n\,
        \norm{\grad{\riemannmetric*}\varphi}{\bigL\infty{\riemannmetric*}} 
    $ 
    for all $x\in\posfold$ and $n\in\Nnum\setminus\{0\}$.
    By construction, we have that
    \begin{displaymath}
        AP\vlfunc{\varphi_n}
        =
        -\spray{\riemannmetric*}\vlfunc{\varphi_n} 
        =
        \dualpair{%
            \spray{\riemannmetric*}%
        }{%
            d\varphi_n \bincirc d\tanproj
        }
        =
        - d\varphi_n
        \longrightarrow
        0
        \qquad\text{as }n\to\infty
    \end{displaymath}
    pointwise and in \bigL2\=/sense.
    Since $\operator{A}$ is closed, 
    we have $1\in\opdomain{A}$ and $A1 = 0$.
 
    Since we know form~\hyperref[itm:fld:SAD:S]{Lemma~\ref{lem:fld:SAD} part~\ref{itm:fld:SAD:S}} 
    that $S\vlfunc{\varphi_n} = 0$ for all $n\in\Nnum$, 
    we have $L\vlfunc{\varphi_n} = - A\vlfunc{\varphi_n}$ for all $n\in\Nnum$. 
    The sequence $\parentheses*{L\vlfunc{\varphi_n}}_{n\in\Nnum}$ 
    converges in~$H$ to~0 as $n\to\infty$.
\end{proof}

\subsection{Hypocoercivity conditions}
\label{subsec:fld:H}

\begin{lemma}[algebraic relation~\ref{cond:H1}]
\label{lem:fld:H1}
    Let $\Psi$ be loc-Lipschitzian such that 
    $\Leb{\wriemannmetric*} = \exp(-\Psi)\,\Leb{\riemannmetric*}$
    is a probability measure on $(\posfold,\borel{\posfold})$. 
    Then, we have $PAP\vert_D = 0$.
    Cf.~\cite[Proposition~3.11]{HypocoercJFA}.
\end{lemma}
\begin{proof}
    Using only~\cite[Lemma~3.1]{HypocoercJFA}
    and~\hyperref[eq:fld:APf-Eq2]{Equation~\eqref{eq:fld:APf-Eq2}}
    we calculate that
    \begin{displaymath}
        \int_{\unittanbundle{\posfold}{x}} APf \ \diffd\nu
        \overset{\eqref{eq:fld:APf-Eq2}}{=}
        \int_{\unittanbundle{\posfold}{x}}  
            -\riemannmetric{%
                \ID_{\unittanbundle\posfold} %
                }{%
                \grad{\riemannmetric*}\expect[\nu]{f} (x) %
                }{x}
        \ \diffd\nu
        =
        0
    \end{displaymath}
    holds for all $f\in D$ and $x\in\posfold$.
    The rest of the proof works as in~\hyperref[lem:Langevin:H1]{Lemma~\ref{lem:Langevin:H1}}.
\end{proof}

\begin{lemma}[microscopic coercivity~\ref{cond:H2}]
\label{lem:fld:H2}
    Let $\Psi$ loc-Lipschitzian such that 
    $\Leb{\wriemannmetric*} = \exp(-\Psi)\, \Leb{\riemannmetric*}$
    is a probability measure on~$(\posfold,\borel{\posfold})$.
    Then, condition~\ref{cond:H2} holds with 
    $\Lambda_m = (\posdim-1)\frac{\sigma^2}2$.
    Cf.~\cite[Proposition~3.12]{HypocoercJFA}.
\end{lemma}
\begin{proof}
    The proof works the same way as in~\hyperref[lem:Langevin:H2]{Lemma~\ref{lem:Langevin:H2}}
    using the Poincar\'e inequality for the spherical measure, see~\cite[Theorem~2]{Beckner89}.
\end{proof}

For proving condition~\ref{cond:H3}, we want to characterise the operator $(PA^2P,D)$ 
as a weighted horizontal Laplace-Beltrami composed with fibrewise average again.
Then, we get essential m-dissipativity of this operator 
as in~\hyperref[cor:Langevin:PA2P-m-dissipative]{Corollary~\ref{cor:Langevin:PA2P-m-dissipative}}.
Mirroring the computations in~\hyperref[eq:Langevin:A2P-part1]{Equation~\eqref{eq:Langevin:A2P-part1}}
we calculate that
\begin{align}
\label{eq:fld:A2P-part1}
\begin{aligned}
    &\int_{\unittanbundle{\posfold}{x}}
        -\spray{\riemannmetric*} 
        \parentheses*{ %
            -\riemannmetric%
                {\ID_{\unittanbundle\posfold}}%
                {\grad{\riemannmetric*}\expect*[\nu]{f} \bincirc\unittanproj}%
                {\unittanproj}%
        }
    \ \diffd\nu
    \\
    \overset{\eqref{eq:fld:APf-Eq2}}{=}&
    \int_{\unittanbundle{\posfold}{x}}
        \riemannmetric%
            {v}%
            {\grad{\riemannmetric*}\expect*[\nu]{f}(x)}%
            {x}
    \ \diffd\nu(v)
    +
    \int_{\unittanbundle{\posfold}{x}}
        \riemannmetric%
            {v}%
            {\connection{v}{\riemannmetric*}%
                \parentheses*{\grad{\riemannmetric*}\expect[\nu]{f}}(x)
            }%
            {x}
    \ \diffd\nu(v)
    \\
    ={}&
    \frac1\posdim\,
    \laplace{\riemannmetric*}\expect[\nu]{f}(x)
\end{aligned}
\end{align}
for all $v\in\unittanbundle\posfold$ with $x\vdef\unittanproj(v)$.

As for~\hyperref[eq:Langevin:A2P-part2]{Equation~\eqref{eq:Langevin:A2P-part2}} 
we use the proof of part~\ref{itm:fld:SAD:A} of~\hyperref[lem:fld:SAD]{Lemma~\ref{lem:fld:SAD}}
to get that
\begin{align}
\label{eq:fld:A2P-part2}
\begin{aligned}
    &
    P_S\parentheses*{%
        \frac1{\posdim-1}\,
        \grad{\sphere}(\spray{\riemannmetric*}\vlfunc{\Psi})
        (APf)%
        }
    (v)
    =
    \int_{\unittanbundle{\posfold}{x}} 
        \spray{\riemannmetric*}\vlfunc{\Psi}
        \pdot
        APf
    \ \diffd\nu
    \\
    ={}&
    -\frac1\posdim\,
    \riemannmetric%
        {\grad{\riemannmetric*}\Psi(x)}%
        {\grad{\riemannmetric*}\expect[\nu]{f}(x)}%
        {x}
    =
    -\frac1\posdim\,
    \connection{\grad{\riemannmetric*}\Psi}(\expect[\nu]{f})(x)
\end{aligned}
\end{align}
for all $v\in\unittanbundle\posfold$ with $x\vdef\unittanproj(v)$.

Together~\hyperref[eq:fld:A2P-part1]{Equation~\eqref{eq:Langevin:A2P-part1}} 
and~\hyperref[eq:fld:A2P-part2]{Equation~\eqref{eq:Langevin:A2P-part2}} 
imply the relation
\begin{align}
\label{eq:fld:PA2P}
\begin{aligned}
    PA^2Pf
    &=
    P_SA^2Pf
    =
    \frac1\posdim
    \pdot
    \parentheses*{%
        \laplace{\riemannmetric*}\expect[\nu]{f}\bincirc\unittanproj%
        - %
        \connection{\grad{\riemannmetric*}\Psi}\expect[\nu]{f}\bincirc\unittanproj%
    }
    \\
    &=
    \frac1\posdim
    \pdot
    \laplace{\whormetric*}\parentheses*{\vlfunc{\expect[\nu]{f}}}
    =
    \frac1\posdim
    \pdot
    \laplace{\whormetric*}(P_Sf)
\end{aligned}
\end{align}
for all $f\in D$.
Compare this to~\cite[Equation~(3.27)]{HypocoercJFA}.

\begin{proposition}[macroscopic coercivity~\ref{cond:H3}]
    Let $\Psi$ loc-Lipschitzian such that 
    $\Leb{\wriemannmetric*} = \exp(-\Psi)\, \Leb{\riemannmetric*}$
    is a probability measure on~$(\posfold,\borel{\posfold})$ 
    satisfying the~\hyperref[eq:Langevin:Poincare]{Poincar\'e inequality~\eqref{eq:Langevin:Poincare}}.
    Then, condition~\ref{cond:H3} is fulfilled 
    with $\Lambda_M = \frac1\posdim\,\Lambda$.
    Cf.~\cite[Proposition~3.14]{HypocoercJFA}.
\end{proposition}
\begin{proof}
    As before, we compute that for all $f\in D$ holds
    \begin{align*}
        \norm{ APf }{H}^2
        &=
        \int_{\posfold} \int_{\unittanbundle{\posfold}{x}}
            (APf)^2\bigr\vert_{\unittanbundle{\posfold}{x}}
        \ \diffd\nu\, \diffd\Leb{\wriemannmetric*}(x)
        \\
        &\overset{\eqref{eq:fld:APf-Eq2}}{=}
        \int_{\posfold} \int_{\unittanbundle{\posfold}{x}}
            \riemannmetric%
                {v}%
                {\grad{\riemannmetric*}\expect[\nu]{f}(x)}%
                {x}^2
        \quad \diffd\nu(v)\, \diffd\Leb{\wriemannmetric*}(x)
        \\
        &=
        \frac1\posdim
        \int_{\posfold}
            \abs{ %
                \grad{\riemannmetric*}\expect[\nu]{f}(x) %
            }{\riemannmetric*}^2
        \quad \diffd\Leb{\wriemannmetric*}(x)
        =
        \frac1\posdim\,
        \norm{ %
            \grad{\riemannmetric*}\expect[\nu]{f} %
        }{\bigL2{\posfold\to\tanbundle\posfold}{\wriemannmetric*}}^2
        \\
        &\geq
        \frac1\posdim\,\Lambda
        \norm{ %
            \expect[\nu]{f} %
            - %
            \scalarprod{\expect[\nu]{f}}{1}%
                {\bigL2{\posfold}{\wriemannmetric*}} %
        }{\bigL2{\posfold}{\wriemannmetric*}}^2
    \end{align*}
    using the Poincar\'e inequality of the weighted base measure.
    The claim follows with \cite[Corollary~2.13]{HypocoercJFA}, 
    since $(PA^2P,D)$ is essentially m-dissipative
    due to the modification of~\hyperref[cor:Langevin:PA2P-m-dissipative]{Corollary~\ref{cor:Langevin:PA2P-m-dissipative}}
    to the case of $\configfold=\unittanbundle\posfold$.
\end{proof}

This time, we do not even need~\spray{\riemannmetric*} to be a spray
when checking the first part of condition~\ref{cond:H4}.
Compare the following Lemmas~\ref{lem:fld:H4-1} and~\ref{lem:fld:H4-2} 
as well as their proofs
to~\cite[Proposition~3.15]{HypocoercJFA}.

\begin{lemma}[boundedness of~$(BS,D)$, first part of~\ref{cond:H4}]
\label{lem:fld:H4-1}
    Let $\Psi$ be loc-Lipschitzian such that 
    $\Leb{\wriemannmetric*} = \exp(-\Phi)\, \Leb{\riemannmetric*}$
    is a probability measure on $(\posfold,\borel{\posfold})$. 
    Then, with $c_1\vdef(\posdim-1)\frac{\sigma^2}4$ it holds that
    \begin{displaymath}
        \norm{ BSf }{H} 
        \leq 
        c_1 \norm{ (\ID-P_j)f }{H}
        \qquad\text{for all }f\in D
    \end{displaymath}
    and $P_j\in\{P,P_S\}$, $j\in\{1,2\}$.
\end{lemma}
\begin{proof}
    Let $f\in D$ be fixed. 
    Then, we observe that 
    \begin{align*}
        SAPf
        &=
        SAP_Sf
        =
        \frac{\sigma^2}2\,
        \laplace{\sphere}(-\spray{\wriemannmetric*}(P_Sf))
        \\
        &=
        \frac{\sigma^2}2\,(\posdim-1)
        \pdot
        \spray{\wriemannmetric*}(P_Sf)
        =
        -\frac{\sigma^2}2\,(\posdim-1)
        APf
    \end{align*}
    by~\hyperref[lem:Laplace-IDunittan]{Lemma~\ref{lem:Laplace-IDunittan}}, 
    and since 
    $\scalarprod{Ah}{1}{H} = 0$ holds for all $h\in D$
    by~\hyperref[lem:fld:SAD]{Lemma~\ref{lem:fld:SAD}} part~\ref{itm:fld:SAD:A}. 
\end{proof}

\begin{lemma}[boundedness of~$(BA(\ID-P),D)$, second part of~\ref{cond:H4})]
\label{lem:fld:H4-2}
    Let all the conditions of~\hyperref[thm:fld:main]{Theorem~\ref{thm:fld:main}}
    on the potential hold.
    Then, there exists a constant $c_2\in(0,\infty)$ such that
    \begin{displaymath}
        \norm{ BA(\ID-P)f }{H} 
        \leq 
        c_2 \norm{ (\ID-P_j)f }{H}
        \qquad\text{for all }f\in D
    \end{displaymath}
    and $P_j\in\{P,P_S\}$, $j\in\{1,2\}$.
\end{lemma}
\begin{proof}
    Let $f\in D$ and $g\vdef(\ID-PA^2P)f$ 
    as in the proof of~\hyperref[lem:Langevin:H4-2]{Lemma~\ref{lem:Langevin:H4-2}}. 
    Now, the relevant estimate reads as 
    \begin{displaymath}
        \norm{ \adjoint{(BA)}g }{H}
        \leq
        \norm{%
            \vlfunc{\abs{ %
                \operatorname{Hess}_{\riemannmetric*}\parentheses*{\expect[\nu]{f}} %
            }}%
        }{H} %
        + %
        \frac1{\posdim}
            \norm{%
                \vlfunc{\abs{ \grad{\riemannmetric*}\Phi }{\riemannmetric*}} %
                \pdot %
                \vlfunc{\abs{ \grad{\riemannmetric*}\expect[\nu]{f} }{\riemannmetric*}} %
            }{H} %
    .
    \end{displaymath}
    In view of~\hyperref[eq:fld:PA2P]{Equation~\eqref{eq:fld:PA2P}} 
    we have the solution $u\vdef Pf$ of the elliptic equation 
    \begin{align*}
        &
        u 
        - 
        \frac1{\posdim}\,
        \laplace{\whormetric*}P_Su
        =
        g
        \\
        \text{in }\qquad&
        \left\{ 
            u\in\unittanproj^\ast\bigL2{\posfold}{\wriemannmetric*}
        \ \middle\vert\ 
            \exists\, f_0\in\smoothcompact{\posfold}\colon\ 
            u 
            = 
            \vlfunc{f_0} 
            -
            \scalarprod{\vlfunc{f_0}}{1}{H} 
        \right\}
    .
    \end{align*}
    The proof is completed as in~\hyperref[lem:Langevin:H4-2]{Lemma~\ref{lem:Langevin:H4-2}}.
\end{proof}

Combining the results of~\autoref{subsec:fld:D} and~\autoref{subsec:fld:H} 
our main theorem, 
\hyperref[thm:fld:main]{Theorem~\ref{thm:fld:main}}, 
follows from the~\ref{thm:hypocoercivity-thm-named}.

%% file: ergodicity.tex
\section{Existence of martingale solutions and $L^2$-exponential ergodicity}
\label{sec:ergo}

Finally, we show existence of $L$\=/martingale solutions to the SDEs investigated in this article.
The strong mixing of the corresponding semigroups with exponential rate of convergence then implies 
their $\bigL2$\=/exponential ergodicity.
Let the configuration manifold 
$\configfold\in\{\tanbundle\posfold,\unittanbundle\posfold\}$ 
be the state space $E=\configfold$  
as mentioned in the end of~\autoref{sec:intro}.
For basic notions used in the following theorem 
we refer to~\cite{Stannat}, \cite{Trutnau00} and \cite{Trutnau03}.

\begin{theorem}[existence of martingale solutions]
\label{thm:ergo:existence}
If $\configfold=\tanbundle\posfold$,
let the assumptions of~\hyperref[thm:Langevin:main]{Theorem~\ref{thm:Langevin:main}} hold. 
If $\configfold=\unittanbundle\posfold$,
let the assumptions of~\hyperref[thm:fld:main]{Theorem~\ref{thm:fld:main}} hold. 
Then, there is a Hunt process 
\begin{displaymath}
        \boldsymbol{\mathrm{HP}}
        =
        \parentheses*{
            \Omega,
            \mathfrak{A},
            \filtration=\parentheses*{\filtration{t}}_{t\in[0,\infty)},
            \eta=\parentheses*{\eta_t}_{t\in[0,\infty)},
            \parentheses*{\pathprob_v}_{v\in\configfold}
        }
\end{displaymath}
properly associated in the resolvent sense with~\operator{L} 
having infinite life-time and continuous paths $\pathprob_v$\=/almost surely for all $v\in\configfold$.
I.\,e.\ if $(G_a)_{a\in(0,\infty)}$ denotes the resolvent 
corresponding to~\operator{L}, 
then the transition resolvent $(R_a)_{a\in(0,\infty)}$ yields
a quasi-continuous version~$R_af$ of~$G_af$ 
for all $f\in\bigL2{\configfold}{\mu}$ and $a\in(0,\infty)$, 
where $R_af(v)=\int_{(0,\infty)}\exp(-as) \expect[v]{f(\eta_t)}\, \Leb(\diffd s)$.

Moreover, for quasi every initial point $v\in\configfold$ 
the probability measure~$\pathprob_v$ solves the martingale problem for 
$\parentheses*{L,C^2_c(\configfold)}$, i.\hspace{0.08em}e.\
$\boldsymbol{\mathrm{HP}}$ is a martingale solution to 
either~\eqref{eq:geomLangevin} if $\configfold=\tanbundle\posfold$
or to~\eqref{eq:fld} if $\configfold=\unittanbundle\posfold$
for quasi every initial point $v\in\configfold$. 
\end{theorem}
\begin{proof}
    As we have seen before, 
    $D=\smoothcompact{\configfold}$ is a core of~\operator{L}, 
    see~\hyperref[thm:Langevin:D2-loc-Lipschitz]{Theorem~\ref{thm:Langevin:D2-loc-Lipschitz}}
    and the explanations on page~\pageref{fld:D2-loc-Lipschitz}.
    Observe that~$D$ also is an algebra which separates the points of~\configfold.
    Thus, \operator{L} defines a generalised Dirichlet form fulfilling the assumptions of~\cite[Theorem~IV.2.2]{Stannat}. 
    This theorem provides a special standard process $\boldsymbol{\mathrm{HP}}$
    properly associated with~\operator{L} in the resolvent sense.
    Now, infinite life-time follows from~\ref{cond:D7}, 
    i.\,e.\ conservativity, 
    together with~\cite[Theorem~IV.3.8~(ii)]{Stannat}. 
    Moreover, continuous paths are obtained via~\cite[Theorem~3.3]{Trutnau03}.
    Summarising, $\boldsymbol{\mathrm{HP}}$ is a Hunt process indeed. 
    For the statement concerning the martingale problem 
    see~\cite[Corollary~1]{ConradGrothausH1infty} and its proof. 
    Note that there are even some finer statements on the martingale problem, 
    cf.~\cite{Trutnau00}.
\end{proof}

Now, we turn to the matter of ergodicity.
Consider the probability measure~\pathprob on~$(\Omega,\mathfrak{A})$ be given as 
\begin{displaymath}
    \pathprob(A)
    \vdef
    \int_{\configfold} \pathprob_v(A)\ \mu(\diffd v)
    \qquad\text{for all }A\in\mathfrak{A}
.
\end{displaymath}
We estimate for all $g\in\bigL2{\configfold}{\mu}$ with $\expect[\mu]{g}=0$ and $t\in(0,\infty)$ that
\begin{align}
    &\norm*{\frac1t\int_{[0,t)} g(\eta_s) \,\Leb(\diffd s)}{\bigL2{\pathprob}}^2
    =
    \int_\Omega
        \frac1{t^2}
        \int_{[0,t)^2}
            g(\eta_s)g(\eta_u) 
        \ \Leb(\diffd (s,u))
    \ \diffd\pathprob
    \notag\\
    &=
    \begin{aligned}
    \label{eq:ergo:Fubini}
        \frac1{t^2}
        \int_{[0,t)^2}
            \expect[\pathprob]{g(\eta_s)g(\eta_u)} 
        \,\Leb(\diffd (s,u))
        =
        \frac2{t^2}
        \int_{[0,t)}\int_{[0,s)}
            \expect[\pathprob]{g(\eta_s)g(\eta_u)} 
        \,\Leb(\diffd u)\,\Leb(\diffd s)
    \end{aligned}
    \\
    &=
    \begin{aligned}
    \label{eq:ergo:Markov-property}
        \frac2{t^2}
        \int_{[0,t)}\int_{[0,s)}
            \scalarprod{T_{s-u}g}{g}{\bigL2{\mu}}
        \,\Leb(\diffd u)\,\Leb(\diffd s)
    \end{aligned}
    \\
    &=
    \begin{aligned}
    \label{eq:ergo:trafo}
        \frac4{t^2}
        \int_{[0,2t)}\int_{[0,t)}
            \scalarprod{T_vg}{g}{\bigL2{\mu}}
        \,\Leb(\diffd v)\,\Leb(\diffd w)
        =
        \frac8t
        \int_{[0,t)}
            \scalarprod{T_vg}{g}{\bigL2{\mu}}
        \,\Leb(\diffd v)
    \end{aligned}
    \\
    &\leq
    \frac8t\,\norm{g}{\bigL2{\mu}}
    \int_{[0,t)}
        \norm{T_{v}g}{\bigL2{\mu}}
    \,\Leb(\diffd v)
    \notag
.
\end{align}
At step~\eqref{eq:ergo:Fubini} we use Fubini. 
Afterwards at step~\eqref{eq:ergo:Markov-property}, 
we can ensure $u<s$ for symmetry reasons 
and transform expectation wrt.~\pathprob using the (weak) Markov property.
Then, at step~\eqref{eq:ergo:trafo} 
we apply the 2D-transformation formula with $v=s-u$ and $w=s+u$, 
and finish with the Cauchy-Bunyakovsky-Schwarz inequality. 

The previous estimate shows that
using our main theorems -- 
depending on~\configfold 
either~\hyperref[thm:Langevin:main]{Theorem~\ref{thm:Langevin:main}}
or~\hyperref[thm:fld:main]{Theorem~\ref{thm:fld:main}} --
we not only can infer convergence to~0 as $t\to\infty$, 
but also the rate of convergence is explicitly computable.
Indeed, with~$g$ as before we gain
\begin{align*}
    \norm*{\frac1t\int_{[0,t)} g(\eta_s) \,\Leb(\diffd s)}{\bigL2{\pathprob}}^2
    &\leq
    \frac8t\,\norm{g}{\bigL2{\mu}}
    \int_{[0,t)}
        \norm{T_{v}g}{\bigL2{\mu}}
    \,\Leb(\diffd v)
    \\
    &\leq
    \frac8t\,\norm{g}{\bigL2{\mu}}^2
    \int_{[0,t)}
        \kappa_1\euler^{-v\kappa_2}
    \,\Leb(\diffd v)
    \\
    &=
    \frac8t\pdot\frac{\kappa_1}{\kappa_2}\,
    \parentheses*{
        1
        -
        \euler^{-t\kappa_2}
    }
    \pdot
    \norm{g}{\bigL2{\mu}}^2
.
\end{align*}
Thus, we proved the following corollary 
after reducing everything to zero-mean functions wrt.~$\mu$.

\begin{corollary}[{\bigL2}\=/exponentially ergodicity with optimal rate and explicit constants]
\label{cor:ergodicity}
If $\configfold=\tanbundle\posfold$,
let the assumptions of~\hyperref[thm:Langevin:main]{Theorem~\ref{thm:Langevin:main}} hold. 
If $\configfold=\unittanbundle\posfold$,
let the assumptions of~\hyperref[thm:fld:main]{Theorem~\ref{thm:fld:main}} hold. 
Moreover, let~$\kappa_1$ and~$\kappa_2$ be the constants form these respective theorems. 
Then, we have 
\begin{displaymath}
    \norm*{ 
        \frac1t\int_{[0,t)} f(\eta_s) \,\Leb(\diffd s) - \expect[\mu]{f} 
    }{\bigL2{\pathprob}}
    \leq
    \frac2{\sqrt{t}}
    \pdot
    \sqrt{
        \frac{2\kappa_1}{\kappa_2}\,
        \parentheses*{
            1
            -
            \euler^{-t\kappa_2}
        }
    }
    \pdot
    \norm{f-\expect[\mu]{f}}{\bigL2{\mu}}
\end{displaymath}
for all $t\in(0,\infty)$.
\end{corollary}

\begin{remark}
    In the title of~\hyperref[cor:ergodicity]{Corollary~\ref{cor:ergodicity}}
    we claim that the rate~$t^{-\sfrac12}$ is optimal.
    This is obvious in the case that 
    the spectrum of the generator~\operator{L} has, 
    apart from the eigenvalue zero, 
    the largest element $-\kappa<0$
    which is an eigenvalue of~\operator{L}.
    Evidently, all the inequality in the estimates prior to~\hyperref[cor:ergodicity]{Corollary~\ref{cor:ergodicity}}
    are equalities when choosing the function~$g$ there as the eigenvector corresponding to~$-\kappa$.
    Hence, the rate of convergence in~\hyperref[cor:ergodicity]{Corollary~\ref{cor:ergodicity}} 
    is sharp with $\kappa_1=1$ and $\kappa_2=\kappa$. 

    In situations where~\operator{L} can be controlled by a Lyapunov function, 
    see e.\,g.~\cite{HerzogMattingly} in case of purely Euclidean setting, 
    one obtains also exponential rates of convergence for the corresponding semigroups; 
    even in (weighted) total variation distance.
    This implies pointwise convergence of the semigroup applied to test functions at an exponential rate. 
    But even this convergence with an exponential rate would not give a better rate as the one in~\hyperref[cor:ergodicity]{Corollary~\ref{cor:ergodicity}}.

    As in~\cite{ConradGrothausNparticle}, 
    we call the martingale solutions to the SDEs investigated in this article \emph{$L^2$\=/exponential ergodic}, 
    i.\,e.\ ergodic with a rate that corresponds to exponential convergence of the corresponding semigroups.
\end{remark}

%% file: app-loc-coord.tex
\section{Some expressions in local coordinates}
\label{app:loc-coord}

A chart $x=\parentheses*{x^j}_{j=1}^{\posdim}$ with domain $U\subseteq\posfold$ 
induces local coordinates 
$\parentheses*{v^k}_{k=1}^{2\posdim}$ for the preimage $V\vdef\tanproj^{-1}(U)$ in a natural way:
\begin{displaymath}
    v^j 
    \vdef 
    x^j \bincirc \tanproj
    \quad\text{and}\quad
    v^{\posdim+j} 
    \vdef
    \parentheses*{d x^j}^\sharp
    \qquad\text{for } j\in\{1,\ldots,\posdim\}
    .
\end{displaymath}
We might write $\tanproj^\ast x =x\bincirc\tanproj=(v^j)_{j=1}^{\posdim}$ and
$d x = (v^{\posdim+i})_{j=1}^{\posdim}$, 
where the latter shorthand doesn't lead to confusion as 
we denote the Riemannian volume form by~$\diffd\Leb{\riemannmetric*}$.

From~\cite[Lemma~4.1]{GudmundssonKappos}
we know some particular vertical and horizontal lifts:
\begin{displaymath}
    \vertlift\left( %
        \frac{\partial}{\partial x^i} %
    \right)
    =
    \frac{\partial}{\partial v^{\posdim+i}}
    \quad\text{and}\quad
    \horlift\left( %
        \frac{\partial}{\partial x^i} %
    \right)
    =
    \frac{\partial}{\partial v^i}
    -
    \sum_{j,\ell\in\{1,\ldots,\posdim\}}
        \left( \Gamma^\ell_{ij}\bincirc\tanproj \right)
        v^{\posdim+j}
        \frac{\partial}{\partial v^{\posdim+\ell}}
\end{displaymath}
for all $i\in\{1,\ldots,\posdim\}$.
Next, the following expression of semisprays has been taken 
from~\cite[Section~1]{metric-nonlin-connections-Bucataru}:
A semispray~\spray reads in local coordinate form as 
\begin{displaymath}
    \spray
    =
    \sum_{j=1}^{\posdim}
        v^{\posdim+j} \partial v^j
        - 2 G^j(\tanproj^\ast x,d x)\, \partial v^{\posdim+j}
.
\end{displaymath}
The family $\parentheses*{G^j}_{j=1}^{\posdim}$ of functions on~$V$
is characterised by functions
$ 
    N^i_j 
    = 
    \frac{\partial G^i}{\partial v^{\posdim+j}}
$, 
$i,j\in\{1,\ldots,\posdim\}$, 
which are given pointwisely by
\begin{displaymath}
    \left. \horlift\left( \frac{\partial}{\partial x^i} \right) \right\vert_{u}
    =
    \left. \frac{\partial}{\partial v^i} \right\vert_{u}
    -
    \sum_{j=1}^{\posdim} N^j_i(u)  
    \left. \frac{\partial}{\partial v^{\posdim+j}} \right\vert_{u}
\end{displaymath}
for $u\in\tanbundle{\posfold}$ arbitrary.
For instance, 
\cite[Equation~(19)]{metric-nonlin-connections-Bucataru}
yields coefficients $(G^j)$ of a semispray 
corresponding to a given Lagrangian~\Lagrange.
Recall~\hyperref[ex:spray-as-gradient]{Example~\ref{ex:spray-as-gradient}}.

In the following lemmas, 
we prove some formulae used in~\autoref{sec:Langevin}.

\begin{lemma}
\label{lem:app:bracket-with-spray}
    For all $k\in\{1,\ldots,\posdim\}$ holds that 
    \begin{displaymath}
        \Liebracket{%
            \spray %
        }{%
            \vertlift\parentheses*{\partial x^k} %
        }
        =
        \horlift\parentheses*{\partial x^k}
        -
        \sum_{j=1}^{\posdim} 
            N_k^j \pdot \vertlift\parentheses*{\partial x^j}
    .
    \end{displaymath}
\end{lemma}
\begin{proof}
    Let $k\in\{1,\ldots,\posdim\}$.
    Then, we calculate that
    \begin{align*}
        \Liebracket{%
            \spray %
        }{%
            \vertlift\parentheses*{\partial x^k} %
        }
        &=
        \Liebracket{%
            \spray %
        }{%
            \partial v^{\posdim+k} %
        }
        =
        \sum_{j=1}^{\posdim}
            \Liebracket{%
                v^{\posdim+j} \partial v^j %
            }{%
                \partial v^{\posdim+k} %
            }
            -
            2\Liebracket{%
                G^j\partial v^{\posdim+j} %
            }{%
                \partial v^{\posdim+k} %
            }
        \\
        &=
        \sum_{j=1}^{\posdim}\,
            {\underbrace{%
                \frac{\partial v^{\posdim+j}}{\partial v^{\posdim+k}} %
            }_{=\delta_{jk}}}
            \pdot
            \partial v^j
            +
            v^{\posdim+j} 
            \underbrace{\Liebracket{%
                \partial v^j %
            }{%
                \partial v^{\posdim+k} %
            }}_{=0}
        \\
        &\qquad\qquad
            -
            2\biggl(%
                \frac{\partial G^j}{\partial v^{\posdim+k}}
                \pdot
                \partial v^{\posdim+j}
                +
                G_j
                \underbrace{\Liebracket{%
                    \partial v^{\posdim+j} %
                }{%
                    \partial v^{\posdim+k} %
                }%
                }_{=0}%
            \biggr)
        \\
        &=
        \sum_{j=1}^{\posdim} 
            \delta_{jk} \pdot \partial v^j
            - 
            2 N_k^j \pdot \partial v^{\posdim+j}
        =
        \partial v^k
        - 2
        \sum_{j=1}^{\posdim} 
            N_k^j \pdot \partial v^{\posdim+j}
        \\
        &=
        \horlift\parentheses*{\partial x^k}
        -
        \sum_{j=1}^{\posdim} 
            N_k^j \pdot \partial v^{\posdim+j}
        =
        \horlift\parentheses*{\partial x^k}
        -
        \sum_{j=1}^{\posdim} 
            N_k^j \pdot \vertlift\parentheses*{\partial x^j}
    .
    \end{align*}
\end{proof}

The next formula appears to be rather intuitive:

\begin{lemma}
\label{lem:app:covaderivative-with-spray}
    For all $\Xfield\in\smoothsec{\tanbundle\posfold}$ holds that
    \begin{displaymath}
        \hormetric%
            {\connection{\spray}{\sasakimetric*}\horlift(\Xfield)}%
            {\spray}%
        =
        \riemannmetric%
            {\connection{\ID_{\tanbundle\posfold}}{\riemannmetric*}(\Xfield\bincirc\tanproj)}%
            {\ID_{\tanbundle\posfold}}%
            {\tanproj}
    .
    \end{displaymath}
\end{lemma}
\begin{proof}
    Let $\Xfield\in\smoothsec{\tanbundle\posfold}$.
    Then, the Koszul formula 
    describing the Levi-Civita connection~\connection{}{\sasakimetric*} wrt.~\sasakimetric* uniquely
    reads as in our special instance as 
    \begin{align*}
        2\,
        \hormetric%
            {\connection{\spray}{\sasakimetric*}\horlift(\Xfield)}%
            {\spray}%
        &=
        \spray\parentheses*{%
            \hormetric%
                {\horlift(\Xfield)}%
                {\spray}%
        }
        +
        \horlift(\Xfield)\parentheses*{%
            \hormetric%
                {\spray}%
                {\spray}%
        }
        \\
        &\qquad
        -\spray\parentheses*{%
            \hormetric%
                {\spray}%
                {\horlift(\Xfield)}%
        }
        -
        \hormetric%
            {\horlift(\Xfield)}%
            {\Liebracket{\spray}{\spray}}%
        -
        \hormetric%
            {\spray}%
            {\Liebracket{\horlift(\Xfield)}{\spray}}%
        \\
        &\qquad
        +
        \hormetric%
            {\spray}%
            {\Liebracket{\spray}{\horlift(\Xfield)}}%
        \\
        &=
        \horlift(\Xfield)\parentheses*{%
            \vert \ID_{\tanbundle\posfold}\vert_{\riemannmetric*}^2
        }
        +
        2\hormetric%
            {\spray}%
            {\Liebracket{\spray}{\horlift(\Xfield)}}%
        \\
        &=
        \hormetric%
            {\horlift(\Xfield)}%
            {\grad{\hormetric*}\parentheses*{%
                \vert \ID_{\tanbundle\posfold}\vert_{\riemannmetric*}^2}%
            }
        +
        2\hormetric%
            {\spray}%
            {\Liebracket{\spray}{\horlift(\Xfield)}}%
    .
    \end{align*}
    First, we note that
    \begin{displaymath}
        d\tanproj\Liebracket{\spray}{\horlift(\Xfield)}
        =
        \Liebracket{d\tanproj\spray}{d\tanproj\horlift(\Xfield)}
        =
        \Liebracket{\ID_{\tanbundle\posfold}}{\Xfield\bincirc\tanproj}
        = 
        \connection{\ID_{\tanbundle\posfold}}{\sasakimetric*}{(\Xfield\bincirc\tanproj)}
    .
    \end{displaymath}
    Second, the value of~$\vert \ID_{\tanbundle\posfold}\vert_{\riemannmetric*}^2$
    does not specifically depend on the current position  
    and therefore it could be approximated just by functions from $\kappa^\ast\smoothcompact{\posfold}$. 
    In other words, this function is a horizontal lift 
    and the horizontal gradient of a horizontal lift equals~0 always. 
    Hence, the claim is proven.
\end{proof}

%% file: app-misc.tex
\section{Miscellaneous}
\label{app:misc}

The following lemma is pretty elementary and rather an intuitive statement on weighted metric spaces.
Obviously, its assumptions are fulfilled by the exponential type weight as chosen in this article.
\begin{lemma}
\label{lem:weighted-M-complete}
    Let some given base weight~$\rho_{\posfold}$ strictly positive and loc\=/Lipschitzian. 
    Denote by~$d_{\riemannmetric*}$ and~$d_{\wriemannmetric*}$ the metrics wrt.\ 
    the nonweighted and weighted Riemannian metric respectively. 
    Then, those metrics induce equivalent topologies. 
    In particular, by~\ref{cond:M:complete} 
    the weighted manifold $(\posfold,\wriemannmetric*)$ is complete as a metric space.
\end{lemma}
\begin{proof}
    Let $(x_n)_{n\in\Nnum}$ be a $d_{\wriemannmetric*}$\=/Cauchy sequence. 
    Since~\posfold is finite dimensional, 
    this sequence is contained in a compactum~$K\subseteq\posfold$. 
    On~$K$ the weight function is continuous, 
    thus it attains minimum and maximum. 
    Let $\varepsilon\in(0,\infty)$. 
    Then, we have for some $N_\varepsilon\in\Nnum$ and all $n_1,n_2\geq N_\varepsilon$ that 
    \begin{displaymath}
        \frac\varepsilon2
        >
        d_{\wriemannmetric*}(x_{n_1},x_{n_2})
        \geq
        \inf_{y\in K} \rho_{\posfold}(y)^2
        \pdot 
        d_{\riemannmetric*}(x_{n_1},x_{n_2})
    .
    \end{displaymath}
    Hence, $(x_n)_{n\in\Nnum}$ is a $d_{\riemannmetric*}$\=/Cauchy sequence 
    and by~\ref{cond:M:complete} it converges to~$x\in\posfold$.
    This $d_{\riemannmetric*}$\=/limit also is the $d_{\wriemannmetric*}$\=/limit:
    \begin{displaymath}
        d_{\wriemannmetric*}(x_{n_1},x)
        \leq
        d_{\wriemannmetric*}(x_{n_1},x_{n_2})
        +
        \norm%
            {\rho_{\posfold}^2}%
            {\bigL\infty{K}{\Leb{\riemannmetric*}}}
        \pdot 
        d_{\riemannmetric*}(x_{n_2},x)
        < 
        2 \pdot \frac\varepsilon2
        =
        \varepsilon
    \end{displaymath}
    for $n_2\geq N_\varepsilon$ large enough.

    If we start with a $d_{\riemannmetric*}$\=/Cauchy sequence, 
    then it is a $d_{\wriemannmetric*}$\=/Cauchy sequence
    by a similar estimate.
    If the $d_{\wriemannmetric*}$\=/limit exists, 
    it's easily verified that it also is the $d_{\riemannmetric*}$\=/limit.
\end{proof}

%% file: basefile.bbl
\newcommand{\etalchar}[1]{$^{#1}$}
\providecommand{\bysame}{\leavevmode\hbox to3em{\hrulefill}\thinspace}
\providecommand{\MR}{\relax\ifhmode\unskip\space\fi MR }
\providecommand{\MRhref}[2]{%
  \href{http://www.ams.org/mathscinet-getitem?mr=#1}{#2}
}
\providecommand{\href}[2]{#2}
\begin{thebibliography}{{J{\o}}78}

\bibitem[APS60]{AmbrosePalaisSinger}
W.~{Ambrose}, R.~S. {Palais}, and I.~M. {Singer}, \emph{{Sprays.}}, {Anais
  Acad. Brasil. Ci.} \textbf{32} (1960), 163--178 (English).

\bibitem[BCD11]{BucataruConstantinescuDahl}
Ioan {Bucataru}, Oana {Constantinescu}, and Matias~F. {Dahl}, \emph{A geometric
  setting for systems of ordinary differential equations}, International
  Journal of Geometric Methods in Modern Physics \textbf{08} (2011), no.~06,
  1291--1327.

\bibitem[{Bec}89]{Beckner89}
William {Beckner}, \emph{{A generalized Poincar\'e inequality for Gaussian
  measures.}}, {Proc. Am. Math. Soc.} \textbf{105} (1989), no.~2, 397--400
  (English).

\bibitem[BGV92]{Heat-kernels-BGV}
Nicole~{Berline} and Ezra {Getzler} and Mich\'ele {Vergne}, \emph{{Heat
  kernels and Dirac operators}}, Springer, 1992.

\bibitem[{Bis}15]{Bismut}
Jean-Michel {Bismut}, \emph{{Hypoelliptic Laplacian and probability.}}, {J.
  Math. Soc. Japan} \textbf{67} (2015), no.~4, 1317--1357 (English).

\bibitem[BM58]{BottMilnor}
Raoul {Bott} and John~W. {Milnor}, \emph{{On the parallelizability of the
  spheres.}}, {Bull. Am. Math. Soc.} \textbf{64} (1958), 87--89 (English).

\bibitem[BO09]{Bou-RabeeOwhadi}
Nawaf {Bou-Rabee} and Houman {Owhadi}, \emph{{Stochastic variational
  integrators}}, {IMA JNA} \textbf{29} (2009), no.~2, 421--443.

\bibitem[{Buc}06]{metric-nonlin-connections-Bucataru}
Ioan {Bucataru}, \emph{{Metric nonlinear connections.}}, {Differential Geometry
  and its Applications} \textbf{25} (2006), no.~3, 335--343 (English).

\bibitem[{Buc}12]{habil-Bucataru}
\bysame, \emph{{Geometric structures for differential equations fields.}},
  2012, habilitation thesis elaborated at Universitatea "Alexandru Ioan Cuza"
  din Iai\c{s}i.

\bibitem[BV01]{BoeckxVanhecke01}
E.~Boeckx and L.~Vanhecke, \emph{{Unit Tangent Sphere Bundles with Constant
  Scalar Curvature.}}, Czechoslovak Mathematical Journal \textbf{51} (2001),
  no.~3, 523--544.

\bibitem[BVA97]{BoeckxVanhecke97}
E.~Boeckx, L.~Vanhecke, and G.~Auchmuty, \emph{{Characteristic reflections on
  unit tangent sphere bundles.}}, Houston J. Math. \textbf{23} (1997), no.~3,
  427--448.

\bibitem[Car92]{doCarmo}
Manfredo P.~do Carmo, \emph{Riemannian geometry}, Birkhäuser, Boston, Mass.
  u.a., 1992 (English).

\bibitem[CG08]{ConradGrothausH1infty}
Florian {Conrad} and Martin {Grothaus}, \emph{{Construction of n-particle
  Langevin dynamics for \(H^{1,\infty}\)-potentials via generalized Dirichlet
  forms.}}, {Potential Anal.} \textbf{28} (2008), no.~3, 261--282 (English).

\bibitem[CG10]{ConradGrothausNparticle}
\bysame, \emph{{Construction, ergodicity and rate of convergence of
  \(N\)-particle Langevin dynamics with singular potentials.}}, {J. Evol. Equ.}
  \textbf{10} (2010), no.~3, 623--662 (English).

\bibitem[CKW12]{CoffeyKalmykovWaldron}
William~T. {Coffey}, Yuri~P. {Kalmykov}, and John~T. {Waldron}, \emph{{The
  Langevin equation. With applications to stochastic problems in physics,
  chemistry and electrical engineering. 3rd ed.}}, 3rd ed. ed., vol.~27,
  Hackensack, NJ: World Scientific, 2012 (English).

\bibitem[{Con}90]{Conway}
John~B. {Conway}, \emph{{A course in functional analysis. 2nd ed.}}, 2nd ed.
  ed., vol.~96, New York etc.: Springer-Verlag, 1990 (English).

\bibitem[{Dav}80]{Davies}
Edward~Brian {Davies}, \emph{{One-parameter semigroups.}}, {London Mathematical
  Society, Monographs, No. 15. London etc.: Academic Press, A Subsidiary of
  Harcourt Brace Jovanovich, Publishers. VIII}, 1980, p.~230.

\bibitem[DMS09]{DMS09}
Jean {Dolbeault}, Cl\'{e}ment {Mouhot}, and Christian {Schmeiser},
  \emph{{Hypocoercivity for kinetic equations with linear relaxation terms}},
  {Comptes Rendus Mathematique} \textbf{347} (2009), no.~9, 511 -- 516.

\bibitem[DMS15]{DMS15}
\bysame, \emph{{Hypocoercivity for linear kinetic equations conserving mass}},
  Trans. Amer. Math. Soc. \textbf{367} (2015), no.~6, 3807--3828. \MR{3324910}

\bibitem[{Dom}62]{Dombrowski}
Peter {Dombrowski}, \emph{{On the geometry of the tangent bundle}}, {J. Reine
  Angew. Math.} \textbf{210} (1962), 73--88 (English).

\bibitem[DX13]{SphericalHarmAna}
Feng {Dai} and Yuan {Xu}, \emph{{Approximation theory and harmonic analysis on
  spheres and balls.}}, New York, NY: Springer, 2013 (English).

\bibitem[{Fol}99]{Analysis-Folland}
Gerald~B. {Folland}, \emph{{Real Analysis: Modern Techniques and Their
  Applications}}, Wiley, 1999 (English).

\bibitem[{For}12]{Forster3}
Otto {Forster}, \emph{{Analysis 3: Maß- und Integrationstheorie,
  Integralsätze im IRn und Anwendungen}}, Vieweg+Teubner Verlag, 2012
  (German).

\bibitem[FyE01]{Fabrega}
Johannes~Ulrich {F\'abrega}~y Escatllar, \emph{{Die Sasaki-Metrik auf dem
  Tangential- und dem Sphärenbündel}}, diploma thesis, University of Cologne,
  2001.

\bibitem[{Gaf}51]{Gaffney-harmonic-op}
Matthew~P. {Gaffney}, \emph{{The harmonic operator for exterior differential
  forms.}}, {Proc. Natl. Acad. Sci. USA} \textbf{37} (1951), 48--50 (English).

\bibitem[{Gaf}54a]{Gaffney-special-Stokes}
\bysame, \emph{{A special Stoke's theorem for complete Riemannian manifolds.}},
  {Ann. Math. (2)} \textbf{60} (1954), 140--145 (English).

\bibitem[{Gaf}54b]{Gaffney-heat-eq}
\bysame, \emph{{The heat equation method of Milgram and Rosenbloom for open
  Riemannian manifolds.}}, {Ann. Math. (2)} \textbf{60} (1954), 458--466
  (English).

\bibitem[{Gan}64]{Gangolli}
R.~{Gangolli}, \emph{{On the construction of certain diffusions on a
  differentiable manifold.}}, {Z. Wahrscheinlichkeitstheor. Verw. Geb.}
  \textbf{2} (1964), 406--419 (English).

\bibitem[GK02]{GudmundssonKappos}
Sigmundur {Gudmundsson} and Elias {Kappos}, \emph{{On the geometry of tangent
  bundles.}}, {Expo. Math.} \textbf{20} (2002), no.~1, 1--41 (English).

\bibitem[GKMW07]{GKMW07}
T.~{G\"otz}, A.~{Klar}, N.~{Marheineke}, and R.~{Wegener}, \emph{{A stochastic
  model and associated Fokker-Planck equation for the fiber lay-down process in
  nonwoven production processes.}}, {SIAM J. Appl. Math.} \textbf{67} (2007),
  no.~6, 1704--1717 (English).

\bibitem[GKOS01]{GKOS}
Michael {Grosser}, Michael {Kunzinger}, Michael {Oberguggenberger}, and Roland
  {Steinbauer}, \emph{{Geometric theory of generalized functions with
  applications to general relativity}}, Mathematics and its Applications, vol.
  537, Kluwer Academic Publishers, Dordrecht, 2001. \MR{1883263}

\bibitem[Gli97]{Gliklikh}
Yuri Gliklikh, \emph{{Global Analysis in Mathematical Physics. Geometric and
  Stochastic Methods}}, Springer-Verlag New York, 1997 (English).

\bibitem[GMS12]{unpublishedGS}
Martin {Grothaus}, Johannes {Maringer}, and Patrik {Stilgenbauer},
  \emph{{Geometry, mixing properties and hypocoercivity of a degenerate
  diffusion arising in technical textile industry}}, preprint
  at~\url{arXiv:1203.4502v1} [math.PR], 2012.

\bibitem[{Goe}59]{Goetz}
Abraham {Goetz}, \emph{{On measures in fibre bundles}}, Colloq. Math.
  \textbf{7} (1959), 11--18. \MR{0151576}

\bibitem[GS13]{GS13}
Martin {Grothaus} and Patrik {Stilgenbauer}, \emph{Geometric langevin equations
  on submanifolds and applications to the stochastic melt-spinning process of
  nonwovens and biology}, Stochastics and Dynamics \textbf{13} (2013), no.~04,
  1350001.

\bibitem[GS14]{HypocoercJFA}
M.~Grothaus and P.~Stilgenbauer, \emph{{Hypocoercivity for Kolmogorov backward
  evolution equations and applications}}, JFA \textbf{267} (2014), no.~10,
  3515--3556.

\bibitem[GS16]{HypocoercMFAT}
Martin {Grothaus} and Patrik {Stilgenbauer}, \emph{{Hilbert space
  hypocoercivity for the Langevin dynamics revisited}}, Methods Funct. Anal.
  Topology \textbf{22} (2016), no.~2, 152--168 (English). \MR{3522857}

\bibitem[Heb99]{Hebey}
Emmanuel Hebey, \emph{{Nonlinear Analysis on Manifolds. Sobolev Spaces and
  Inequalities}}, Courant Lecture Notes in Mathematics, vol.~5, New York
  University, Courant Institute of Mathematical Sciences, New York; American
  Mathematical Society, Providence, RI, 1999. \MR{1688256}

\bibitem[HM19]{HerzogMattingly}
David~P. {Herzog} and Jonathan~C. {Mattingly}, \emph{{Ergodicity and Lyapunov
  functions for Langevin dynamics with singular potentials.}}, {Commun. Pure
  Appl. Math.} \textbf{72} (2019), no.~10, 2231--2255 (English).

\bibitem[HN05]{HelfferNier}
Bernard {Helffer} and Francis {Nier}, \emph{{Hypoelliptic estimates and
  spectral theory for Fokker-Planck operators and Witten Laplacians.}}, vol.
  1862, Berlin: Springer, 2005 (English).

\bibitem[{Hor}66]{Horvath}
John {Horvath}, \emph{{Topological vector spaces and distributions. Vol. I.}},
  {Addison-Wesley Series in Mathematics.) Reading, Mass.-Palo Alto-London- Don
  Mills, Ontario: Addison-Wesley Publishing Company. XII, 449 p. (1966).},
  1966.

\bibitem[{Hsu}02]{Hsu}
Elton~P. {Hsu}, \emph{{Stochastic analysis on manifolds.}}, vol.~38,
  Providence, RI: American Mathematical Society (AMS), 2002 (English).

\bibitem[ILP15]{ILP-JFA}
Alberto {Ibort Latre}, Fernando {Lled\'o Macau}, and Juan~Manuel {Pérez
  Pardo}, \emph{{Self-adjoint extensions of the Laplace–Beltrami operator and
  unitaries at the boundary.}}, JFA \textbf{268} (2015), no.~3, 634 -- 670.

\bibitem[jLM89]{LawsonMichelsohn}
H.~Blaine jun. {Lawson} and Marie-Louise {Michelsohn}, \emph{{Spin geometry.}},
  Princeton, NJ: Princeton University Press, 1989 (English).

\bibitem[jM69]{McKean}
H.~P. jun. {McKean}, \emph{{Stochastic integrals.}}, {New York-London: Academic
  Press XIII, 140 p. (1969).}, 1969.

\bibitem[{J{\o}}78]{Joergensen}
E.~{J{\o}rgensen}, \emph{{Construction of the Brownian motion and the
  Ornstein-Uhlenbeck process in a Riemannian manifold on basis of the
  Gangolli-McKean injection scheme.}}, {Z. Wahrscheinlichkeitstheor. Verw.
  Geb.} \textbf{44} (1978), 71--87 (English).

\bibitem[{Ker}58]{Kervaire}
Michel~A. {Kervaire}, \emph{{Non-parallelizability of the $n$-sphere for $n >
  7$.}}, {Proc. Natl. Acad. Sci. USA} \textbf{44} (1958), 280--283 (English).

\bibitem[KMS93]{KolvarMichorSlovak}
Ivan {Kol\'{a}\v{r}}, Peter~W. {Michor}, and Jan {Slov\'{a}k}, \emph{{Natural
  Operations in differential geometry}}, Springer-Verlag Berlin Heidelberg,
  1993 (English).

\bibitem[KMW12]{KlarMaringerWegener}
Axel {Klar}, Johannes {Maringer}, and Raimund {Wegener}, \emph{{A 3D model for
  fiber lay-down in nonwoven production processes.}}, {Math. Models Methods
  Appl. Sci.} \textbf{22} (2012), no.~9, 1250020, 18 (English).

\bibitem[{Kol}00]{Kolokoltsov}
V.~N. {Kolokoltsov}, \emph{{Semiclassical analysis for diffusions and
  stochastic processes.}}, vol. 1724, Berlin: Springer, 2000 (English).

\bibitem[{Lan}95]{SergeLang}
Serge {Lang}, \emph{{Differential and Riemannian manifolds. 3rd ed.; Repr. of
  the orig. 1972.}}, 3rd ed.; repr. of the orig. 1972 ed., vol. 160,
  Heidelberg: Springer-Verlag, 1995 (English).

\bibitem[{Lee}12]{Lee}
John {Lee}, \emph{{Introduction to Smooth Manifolds}}, Springer, 2012
  (English).

\bibitem[LMS{\etalchar{+}}17]{Lindner-Holger17}
Felix {Lindner}, Nicole {Marheineke}, Holger {Stroot}, Alexander {Vibe}, and
  Raimund {Wegener}, \emph{{Stochastic dynamics for inextensible fibers in a
  spatially semi-discrete setting.}}, {Stoch. Dyn.} \textbf{17} (2017), no.~2,
  29 (English).

\bibitem[MR99]{MarsdenRatiu}
Jerrold {Marsden} and Tudor {Ratiu}, \emph{{Introduction to Mechanics and
  Symmetry. A Basic Exposition of Classical Mechanical Systems}},
  Springer-Verlag New York, 1999 (English).

\bibitem[MW07]{MW06}
Nicole {Marheineke} and Raimund {Wegener}, \emph{{Fiber dynamics in turbulent
  flows: specific Taylor drag.}}, {SIAM J. Appl. Math.} \textbf{68} (2007),
  no.~1, 1--23 (English).

\bibitem[{Nic}96]{Nicolaescu}
Liviu~I. {Nicolaescu}, \emph{{Lectures On The Geometry Of Manifolds}}, World
  Scientific Publishing, 1996 (English).

\bibitem[{Pé}13]{PhD-Perez-Pardo}
Juan~Manuel {Pérez Pardo}, \emph{{On the Theory of Self-adjoint Extensions of
  the Laplace-Beltrami Operator, Quadratic Forms and Symmetry}}, Ph.D. thesis,
  Universidad Carlos III de Madrid, 2013.

\bibitem[{Ros}13]{Rosenberg}
Jonathan {Rosenberg}, \emph{{Levi-Civita's Theorem for Noncommutative Tori}},
  SIGMA \textbf{9} (2013), 9 (English).

\bibitem[RS80]{ReedSimonI}
Michael {Reed} and Barry {Simon}, \emph{{Methods of modern mathematical
  physics. I: Functional analysis. Rev. and enl. ed.}}, {New York etc.:
  Academic Press, A Subsidiary of Harcourt Brace Jovanovich, Publishers, XV,
  400 p. \$ 24.00 (1980).}, 1980.

\bibitem[{Sak}96]{Sakai}
Takashi {Sakai}, \emph{{Riemannian geometry. Transl. from the Japanese by
  Takashi Sakai.}}, vol. 149, Providence, RI: AMS, American Mathematical
  Society, 1996 (English).

\bibitem[{Sas}58]{Sasaki58}
Shigeo {Sasaki}, \emph{{On the differential geometry of tangent bundles of
  Riemannian manifolds.}}, {Tohoku Math. J. (2)} \textbf{10} (1958), 338--354
  (English).

\bibitem[{Sol}95]{Soloveitchik}
M.~R. {Soloveitchik}, \emph{{Fokker-Planck equation on a manifold. Effective
  diffusion and spectrum.}}, {Potential Anal.} \textbf{4} (1995), no.~6,
  571--593 (English).

\bibitem[{Sta}99]{Stannat}
Wilhelm {Stannat}, \emph{{The theory of generalized Dirichlet forms and its
  applications in analysis and stochastics.}}, vol. 678, Providence, RI:
  American Mathematical Society (AMS), 1999 (English).

\bibitem[{Sti}14]{PhD-Stilgenbauer}
Patrik {Stilgenbauer}, \emph{{The Stochastic Analysis of Fiber Lay-Down
  Models}}, Ph.D. thesis, Technische Universität Kaiserslautern (TUK), 2014.

\bibitem[TO62]{TachibanaOkumura}
Shunichi {Tachibana} and M.~{Okumura}, \emph{{On the almost-complex structure
  of tangent bundles of Riemannian spaces.}}, {Tohoku Math. J. (2)} \textbf{14}
  (1962), 156--161 (English).

\bibitem[{Tru}00]{Trutnau00}
Gerald {Trutnau}, \emph{{Stochastic calculus of generalized Dirichlet forms and
  applications to stochastic differential equations in infinite dimensions.}},
  {Osaka J. Math.} \textbf{37} (2000), no.~2, 315--343 (English).

\bibitem[{Tru}03]{Trutnau03}
\bysame, \emph{{On a class of non-symmetric diffusions containing fully
  non-symmetric distorted Brownian motions.}}, {Forum Math.} \textbf{15}
  (2003), no.~3, 409--437 (English).

\bibitem[{Vil}07]{Villani07}
C\'edric {Villani}, \emph{{Hypocoercive diffusion operators.}}, {Boll. Unione
  Mat. Ital., Sez. B, Artic. Ric. Mat. (8)} \textbf{10} (2007), no.~2, 257--275
  (English).

\bibitem[{Vil}09]{Villani09}
\bysame, \emph{{Hypocoercivity.}}, {Mem. Am. Math. Soc.} \textbf{950} (2009),
  1--141 (English).

\bibitem[WI81]{IkedaWatanabe}
Shinzo {Watanabe} and Nobuyuki {Ikeda}, \emph{{Stochastic Differential
  Equations and Diffusion Processes}}, North Holland, 1981 (English).

\bibitem[{Wol}73]{Wolf73}
Joseph~A. {Wolf}, \emph{{Essential self adjointness for the Dirac operator and
  its square.}}, {Indiana Univ. Math. J.} \textbf{22} (1973), 611--640
  (English).

\end{thebibliography}
